\documentclass{amsart}

\usepackage{latexsym,amssymb,amsmath, amsfonts}
\usepackage{amsthm}
\usepackage{amscd}
\usepackage{bbm}

\DeclareMathOperator{\sign}{sign}
\DeclareMathAlphabet{\mathpzc}{OT1}{pzc}{m}{it}

\numberwithin{equation}{section}
\usepackage{hyperref}
\hypersetup{
    colorlinks,
    citecolor=blue,
    filecolor=black,
    linkcolor=blue,
    urlcolor=black
}
\usepackage{kantlipsum} 

\calclayout

\newcommand{\R}{\mathbb{R}}

\newcommand{\be}{\begin{equation}}
\newcommand{\en}{\end{equation}}
\newcommand{\ee}{\end{equation}}

\DeclareMathOperator{\supp}{supp}

\newcommand{\bt}{\begin{theorem}}
\newcommand{\et}{\end{theorem}}
\newcommand{\bp}{\begin{proof}}
\newcommand{\ep}{\end{proof}}
\newcommand{\bc}{\begin{cor}}
\newcommand{\ec}{\end{cor}}
\newcommand{\bl}{\begin{lemma}}
\newcommand{\el}{\end{lemma}}
\newcommand{\bprop}{\begin{prop}}
\newcommand{\eprop}{\end{prop}}

\newtheorem{theorem}{Theorem}[section]

\newtheorem{remark}[theorem]{Remark}
\newtheorem{lemma}[theorem]{Lemma}

\newtheorem{prop}[theorem]{Proposition}

\newtheorem{cor}[theorem]{Corollary}
\newtheorem{claim}[theorem]{Claim}
\DeclareMathAlphabet{\mathpzc}{OT1}{pzc}{m}{it}
\newtheorem{thmx}{Theorem}

\author{Oscar Ria\~no}
\address{Department of Mathematics \& Statistics, Florida International University, Miami, FL 33199, USA}\email{orianoca@fiu.edu}
\date{March 2021}
\title[On persistence properties for solutions of the fractional Korteweg-de Vries equation]{On persistence properties in weighted spaces for solutions of the fractional Korteweg-de Vries equation}
\keywords{Fractional KdV equation, Burgers-Hilbert equation, Weighted Sobolev spaces, Unique continuation principles}

\begin{document}

\begin{abstract} 
Persistence problems in weighted spaces have been studied for different dispersive models involving non-local operators. Generally, these models do not propagate polynomial weights of arbitrary magnitude, and the maximum decay rate is associated with the dispersive part of the equation. Altogether, this analysis is complemented by unique continuation principles that determine optimal spatial decay. This work is intended to establish the above questions for a weakly dispersive perturbation of the inviscid Burgers equation. More precisely, we consider the fractional Korteweg-de Vries equation,  which comprises the Burgers-Hilbert equation and dispersive effects weaker than those of the Benjamin-ono equation.  
\end{abstract}

\maketitle

\section{Introduction}\label{intro}

We consider the initial value problem (IVP) associated to the fractional Korteweg-de Vries equation
\begin{equation}\label{BDBO}
    \left\{\begin{aligned}
    &\partial_t u -\partial_xD^{\alpha} u+u\partial_x u =0,\quad (x,t)\in \R^{2}, \, -1\leq  \alpha<1,\, \alpha\neq 0, \\
    &u(x,0)=u_0(x),
    \end{aligned}\right.
\end{equation}
where the operator $D^s$ denotes the homogeneous derivative of order $s\in \mathbb{R}$ defined by
$$D^s f=(\mathcal{H}\partial_x)^s f=(|\xi|^s \widehat{f}(\xi,\eta))^{\vee},$$
and $\mathcal{H}$ stands for the Hilbert transform 
\begin{equation*}
    \mathcal{H}f(x)=\frac{1}{\pi} \, p.v. \int_{\mathbb{R}}\frac{f(y)}{x-y} dy=\big(-i\sign(\xi)\widehat{f}\big)^{\vee}(x).
\end{equation*}
The family \eqref{BDBO} with $\alpha \geq -1$ comprises a wide varity of models of physical relevance. We recall that the case $\alpha=2$ corresponds to the widely studied Korteweg-de Vries equation (KdV), and $\alpha=1$ is referred as the Benjamin-Ono equation (BO) (see \cite{benjamin,KdV,Ono}). Besides, the case $\alpha=0$, after a suitable change of variables agrees with the inviscid Burger's equation. Setting $\alpha=-1/2$, the equation in \eqref{BDBO} was proposed in \cite{Vera3} as a model in the study of water waves in two dimensions with infinite depth. Actually, the case $\alpha=-1/2$ is reminiscent for large frequencies of the Whitman equation (see \cite{KLEIN1,whitham2011linear} and the reference therein)
\begin{equation}\label{whitman}
\begin{aligned}
\partial_t u -\partial_x \Gamma u+u\partial_x u =0,\quad (x,t)\in \R^{2},
\end{aligned}
\end{equation}
where $\Gamma$ is defined by the Fourier multiplier $\Big(\frac{\tanh(\xi)}{\xi} \Big)^{1/2}$. Similarly, the case $\alpha=1/2$ in \eqref{BDBO} mimics the dispersion of the Whitman equation in the presence of strong surface tension, i.e., \eqref{whitman} with $\Gamma$ defined by $(1+\nu \xi^2)^{1/2}\Big(\frac{\tanh(\xi)}{\xi} \Big)^{1/2}$, where $\nu$ measures the surface tension effects. In general, the equations \eqref{BDBO} with $\alpha>-1$ have been used in a diversity of wave phenomena, including vorticity waves in the coastal zone, we refer to \cite{naumkin,shrira_voronovich}. Another model of physical relevance is the case $\alpha=-1$, or the so-called Burgers-Hilbert equation \cite{Biello,BuHil}
\begin{equation}\label{BHeqution}
\partial_tu+\mathcal{H}u+u\partial_x u=0.
\end{equation}
This model was introduced to study nonlinear constant-frequency waves.

The equations in \eqref{BDBO} are also useful from physics or continuum mechanics to understand the competition between dispersion and nonlinearity. They are convenient in the study of phenomena where the iteration between these effects is investigated through a certain fractional parameter. We remark that one can address such questions by fixing the dispersion and varying the nonlinearity instead, e.g., $u^p u_x $ with $ p \geq 1 $ (notice that there are inherent difficulties to treat the case where $p$ is not an integer number). As an example on the study of the effects of dispersion on a quadratic nonlinearity, it has been considered if the formation of shock proper of the Burger's equation persists in the presence of dispersion such as in the equation in \eqref{BDBO}, or on the contrary, the dispersive effect dominates the dynamic. This problem is completely different for dispersion $-1\leq \alpha<0$ and $0 <\alpha \leq 1$. In the former case, it seems that the hyperbolic aspects dominate, and so it is conjectured shock formation for all dispersion $1\leq \alpha<0$ in \eqref{BDBO}. To the best of our knowledge, it has been established wave breaking (shock formation) for solutions of the Cauchy problem \eqref{BDBO} when $-1 \leq \alpha<-1/3$ (see \cite{Vera2,Hur_2014,Jean_Wang}). For the latter case $0<\alpha<1$, the dispersive effects dominate as suggested by the numerical simulations presented in \cite{KLEIN1,KLEIN2}, where it is conjecture that no shock formation exists. Besides, it is expected that when $1/2<\alpha<1$, \eqref{BDBO} generates global solutions (in \cite{MOVPIL} this conjecture was proved for $6/7<\alpha <1$), and when $0<\alpha \leq 1/2$ with $\alpha \neq 1/3$ 
 one expects different types of blow-up. In contrast, in \cite{Herr_Ione_Keni_Koch}, it was proved that for $\alpha \geq 1$ the solutions of the Cauchy problem are global and therefore no finite time blow-up occurs. We also point out some other results involving the equation \eqref{BDBO}, such as the study of the propagation on regularity phenomena in \cite{Argen1}, and the local unique continuation principles established in \cite{KPVUNI1,KDPVUNI2}.

Concerning some invariant of the equation in \eqref{BDBO}, setting  $-1<\alpha<1$, $\alpha \neq 0$, real solutions of the IVP \eqref{BDBO} formally satisfy the following conserved quantities (time invariant)
\begin{equation}\label{conservequanti}
\begin{aligned}
I_1(u)&=\int_{\mathbb{R}} u(x,t) \, dx, \hspace{2cm} I_2(u)=\int_{\mathbb{R}} u^2(x,t)\, dx, \\
I_3(u)&=\int_{\mathbb{R}} \big(D^{\alpha/2} u\big)^2(x,t)-\frac{1}{3}u^3(x,t)\, dx.
\end{aligned}
\end{equation} 
Similarly, real solutions of the IVP \eqref{BDBO} with $\alpha=-1$ satisfy $I_2(u)$ and $I_3(u)$ above. Note that by Sobolev embedding $H^{1/6}(\mathbb{R}) \subset L^{3}(\mathbb{R})$, and so  $I_1(u)$ is well-defined whenever $\alpha \geq 1/3$. Moreover, \eqref{BDBO} with $-1<\alpha<1$, $\alpha \neq 0$ is invariant under the scaling transformation $u_{\lambda}(x,t)=\lambda^{\alpha}u(\lambda x, \lambda^{1+\alpha}t)$ for any positive number $\lambda$. Thus, the critical Sobolev index corresponding to the scale is $s_{c,\alpha}=\frac{1}{2}-\alpha$. In particular, when $\alpha=1/2$ the equation is $L^2$-critical.

Regarding solitary waves for \eqref{BDBO}, we recall that they are solutions of the form $u(x,t)=Q_c(x-ct)$ with suitable decay that must satisfy the equation
\begin{equation}\label{gsEqu}
D^{\alpha}Q_c+cQ_c-\frac{1}{2}Q_c^2=0,
\end{equation}
where $c>0$. As the Hamiltonian does not make sense when $\alpha <1/3$, we do not expect existence of solitary wave for such dispersions (see \cite{Linares2014} for a proof of nonexistence when $\alpha<0$). However, solitary waves exist when $\alpha>1/3$, 
we refer to \cite{Pava_2018,Arnesen,Frank2013,linares2015} and the references therein for a more detailed analysis. We emphasize that positive solutions of \eqref{gsEqu} satisfy the regularity condition $Q_c\in H^{\alpha+1}(\mathbb{R})\cap C^{\infty}(\mathbb{R})$, and the decay
\begin{equation}\label{decayGroundS}
\frac{C_1}{1+|x|^{1+\alpha}}\leq Q_c(x) \leq \frac{C_2}{1+|x|^{1+\alpha}},
\end{equation}
for any $x \in \mathbb{R}$, and $0<C_1<C_2$ depending on $\alpha$.

Next, we recall some well-posedness results for the initial value problem \eqref{BDBO}. We follow Kato's notion of well-posedness, which consists of existence, uniqueness, persistence property (if $u_0\in X$ functional space, then the corresponding solutions describes a continuous curve in $X$, in other words, $u\in  C([0,T];X)$) and continuous dependence of the map data-solution. From this standpoint, in the case $0<\alpha<1$, the best known well-posedness results in $L^2(\mathbb{R})$ Sobolev spaces were established in \cite{MOVPIL} (see also \cite{Linares2014} for previous conclusion on this matter). It was shown that the IVP \eqref{BDBO} for $\alpha \in (0,1)$ is locally well-posed (LWP) in $H^s(\mathbb{R})$, whenever $s>3/2-5\alpha/4$, and globally well-posed (GWP) in $H^{\alpha/2}(\mathbb{R})$ as soon as $\alpha>6/7$.  In contrast, by standard energy estimates, the IVP \eqref{BDBO} with dispersion $\alpha\in[-1,0)$ is LWP in $H^s(\mathbb{R})$, $s>3/2$. To the best of our knowledge, there are not results in Kato's sense addressing well-posedness for the Cauchy problem \eqref{BDBO} with $\alpha\in[-1,0)$ in lower regularity spaces $H^s(\mathbb{R})$, $s<3/2$. It should be noted that the previous well-posedness conclusions for $0<\alpha<1$ were approached by compactness method, this is due to the results proved in \cite{molin}, stating that the IVP associated to \eqref{BDBO} in the range $0<\alpha<1$ cannot be solved by a contraction argument based on the corresponding integral equation in any Sobolev space $H^s(\mathbb{R})$, $s\in \mathbb{R}$.

This work aims to provide a detailed analysis of the effects of weak dispersion on the persistence of solutions in weighted spaces.  More precisely, we will show that the spatial behavior of solutions of the equation \eqref{BDBO} is mostly controlled by dispersive effects, which in turn determine the maximum decay allowed by the model. This result contrasts with other studies on the equation (such as wave breaking), where for weak dispersions, the effects of nonlinearity seems to be stronger on the dynamics of the model. We recall that in \cite{FLinaPonceWeBO,FLinaPioncedGBO,FonPO,Iorio1,Iorio2}, similar conclusions were obtained for the Benjamin-Ono equation and the dispersion generalized Benjamin-Ono equation, i.e., \eqref{BDBO} with $1\leq \alpha <2$. 

Let us now state our results. But first, since part of our analysis depends on weighted energy estimates, we require a precise existence theory in $H^s(\mathbb{R})$ compatible with this approach. In this direction, we present the following local well-posedness result. 
\begin{thmx}\label{localtheo}
Let $-1\leq \alpha<1$ and $s>s_{\alpha}$, where $s_{\alpha}=3/2$ for $-1\leq \alpha <0$, and $s_{\alpha}=3/2-3\alpha/8$ for $0<\alpha<1$. Then for any $u_0 \in H^{s}(\mathbb{R})$, there exists a positive time $T(\|u_0\|_{H^s})$ and a unique solution $u$ of \eqref{BDBO} such that
\begin{equation}\label{classofsolu}
u\in C([0,T];H^s(\mathbb{R}))\cap L^1([0,T];W^{1,\infty}(\mathbb{R})).
\end{equation}
Moreover the flow-map $u_0 \mapsto u(t)$ is continuous in the $H^s$-norm.
\end{thmx}
The case $-1\leq \alpha<0$ in Theorem \ref{localtheo} is obtained by standard parabolic regularization argument and the embedding $H^{s}(\mathbb{R})\hookrightarrow W^{1,\infty}(\mathbb{R})$ for $s>3/2$. The local result for $0<\alpha<1$ is due to Linares, Pilod and Saut \cite{Linares2014}. Before stating our results, as a further preliminary, we introduce the weighted Sobolev spaces
\begin{equation}\label{weightespace}
Z_{s,r}(\mathbb{R})=H^{s}(\mathbb{R})\cap L^{2}(| x|^{2r} \, dx), \hspace{0.5cm} s,r \in \mathbb{R}
\end{equation}
and 
\begin{equation}\label{weightespacedot}
\dot{Z}_{s,r}(\mathbb{R})=\left\{f\in H^{s}(\mathbb{R})\cap L^{2}(| x|^{2r} \, dx):\, \widehat{f}(0)=0 \right\}, \hspace{0.5cm} s,r \in \mathbb{R}.
\end{equation}
To motive our results in the above spaces, we observe that the linear part of equation \eqref{BDBO} 
\begin{equation}\label{commulinearpa}
\partial_t-\partial_xD^{\alpha} \qquad \text{ commutes with }\qquad x+(1+\alpha)tD^{\alpha}.
\end{equation}
Thus, for the case $0<\alpha<1$, it is natural to address well-posedness conclusions in the spaces $Z_{s,r}(\mathbb{R})$ where the regularity and the decay satisfy $s\geq \alpha r$. However, notice that this is no longer the case when $-1\leq \alpha <0$. Now, we are in the condition to present our first result:	
\begin{theorem}\label{Theowellpos}
Let $s_{\alpha}=3/2$ for $-1\leq \alpha <0$, and $s_{\alpha}=3/2-3\alpha/8$ for $0<\alpha<1$.
\begin{itemize}
\item[(i)] If $\alpha \in [-1,1)\setminus\{0\}$, then the IVP \eqref{BDBO} is LWP in $Z_{s,r}(\mathbb{R})$, whenever $r\in (0,3/2+\alpha)$ and $s\geq \max\{s_\alpha^{+},\alpha r\}$.
\item[(ii)] Let $\alpha=-1$ and $s>3/2$. Then
\begin{itemize}
\item the IVP \eqref{BDBO} is LWP in the space 
\begin{equation*}
HZ_{s,1/2}(\mathbb{R}):=\{f\in Z_{s,1/2}(\mathbb{R}): \|f\|_{Z_{s,1/2}}+ \||x|^{1/2}\mathcal{H}f\|_{L^2}<\infty\}.
\end{equation*}
\item If $1/2<r<3/2$. Then, the IVP \eqref{BDBO} is LWP in the space $\dot{Z}_{s,r}(\mathbb{R})$.
\end{itemize}
\item[(iii)] If $\alpha \in (-1,1)\setminus\{0\}$, then the IVP \eqref{BDBO} is LWP in $\dot{Z}_{s,r}(\mathbb{R})$, whenever $r\in [3/2+\alpha,5/2+\alpha)$ and $s\geq \max\{s_\alpha^{+},\alpha r\}$.
\end{itemize}
\end{theorem}

The previous theorem shows that the flow-map data solution of \eqref{BDBO} preserves $L^2$ spaces involving precise conditions of decay and regularity determined by the dispersion. This result is useful in the study of the classes of weights (such as polynomial or exponential) propagated by solutions of \eqref{BDBO}.

The proof of Theorem \ref{Theowellpos} follows from weighted energy estimates inspired by the results of Fonseca, Linares, and Ponce in \cite{FLinaPioncedGBO,FonPO} for the cases $1\leq \alpha <2$. However, our conclusions involve several independent difficulties expected from the iteration between low dispersive effects and weighted spaces. Mainly, motivated by \eqref{BDBO}, the proof of Theorem \ref{Theowellpos} requires a detailed study of polynomial decay and negative derivatives of the solution. Besides, some additional difficulties are proper of lower dispersions as can be seen from the lack of $L^2$-integrability at the origin of the Fourier multiplier associated to $D^{\alpha}$, i.e., $|\xi|^{\alpha}$ with $-1<\alpha \leq -1/2$. Thus, the proof of Theorem \ref{Theowellpos} for such dispersions requires incorporating new equations obtained after applying a controlled number of derivatives in conjunction with a projector operator to the equation in \eqref{BDBO} (see \eqref{diferenequalderivlocal} below). We apply this strategy on several occasions during the proof of Theorem \ref{Theowellpos}. As a consequence, we require to deduce a commutator estimate (see Lemma \ref{lemmacomm3}) dealing with projectors and derivatives. 

\begin{remark}
\begin{itemize}
\item[(a)] Setting $1\leq \alpha <2$ in Theorem \ref{Theowellpos}, we obtain the same decay rates  in \cite{FLinaPioncedGBO,FonPO}. This shows some continuity between spatial decay and dispersion for solutions of \eqref{BDBO}.

\item[(b)] The results in Theorem \ref{Theowellpos} (i) and (iii) are still true for regularity $s>\max\{1,r\alpha\}$ provided that there exists a local (or global) theory for the IVP \eqref{BDBO} in $H^s(\mathbb{R})$, for which the solutions generated are in the class \eqref{classofsolu}. This last condition is fundamental to apply techniques based on energy methods.
\item[(c)] The proof of Theorem \ref{Theowellpos} (ii) for $r=1/2$ requires the assumption $\mathcal{H}u_0 \in L^2(|x|\, dx)$. Since this is not true in general (as the weight $|x|$ does not satisfy the $A_2(\mathbb{R})$ condition), we have introduced the spaces $HZ_{s,1/2}(\mathbb{R})$. Also, notice that for $u_0 \in Z_{s,1/2}(\mathbb{R})$ the condition $\widehat{u_0}(0)=0$ does not make sense in general too.
\item[(d)] For the case $0<\alpha<1$, our conclusion in Theorem \ref{Theowellpos} (i) satisfies the condition $s\geq \alpha r$ when
\begin{equation}\label{condiregdec1}
\begin{aligned}
\frac{\sqrt{609}-15}{16} < \alpha<1, \quad \text{ and } \quad 3/2\alpha-3/8< r <3/2+\alpha,
\end{aligned}
\end{equation}
and Theorem \ref{Theowellpos} (iii) for those cases where 
\begin{equation}\label{condiregdec2}
\begin{aligned}
\frac{\sqrt{913}-23}{16} < \alpha<1, \quad \text{ and } \quad 3/2\alpha-3/8< r <5/2+\alpha.
\end{aligned}
\end{equation}
As a consequence, there still remains to determine whether the persistence (or well-posedness) in $Z_{s,r}(\mathbb{R})$ and $\dot{Z}_{s,r}(\mathbb{R})$ could be extended to match the regularity and decay condition $s\geq \alpha r$, for those parameters $0<\alpha<1$ and $0<r<5/2+\alpha$ not satisfying the restrictions \eqref{condiregdec1} and \eqref{condiregdec2}.

In the lower regularity case, one can try to adapt the functional setting employed by Molinet, Pilod, and Vento in \cite{MOVPIL} to deal with the inclusion of fractional weights. This topic exceeds the scope of this paper.
\end{itemize}
\end{remark}
Let us now state our unique continuation principles.
\begin{theorem}\label{Theotwotimcondi}
Let $s_{\alpha}=3/2$ for $-1\leq \alpha <0$, and $s_{\alpha}=3/2-3\alpha/8$ for $0<\alpha<1$.
\begin{itemize}
\item[(i)]  Consider $3/10<r<1/2$ and $s>s_{\alpha}$.  Let $u \in C([0,T];Z_{s,r}(\mathbb{R}))$ be a solution of the IVP \eqref{BDBO} for $\alpha=-1$. If there exist two different times $0\leq t_1<t_2 \leq T$ such that 
\begin{equation*}
u(\cdot,t_1) \in Z_{s,(1/2)^{+}}(\mathbb{R}), \, \text{ and } \,  u(\cdot,t_2) \in Z_{s,1/2}(\mathbb{R}),
\end{equation*}
then 
$$\widehat{u}(0,t)=\int u(x,t) \, dx=0 \quad \text{ for all } \quad t\in[t_1,T].$$ 
\item[(ii)] Assume $-1<\alpha<0$, $\max\{1/2,9/10+3\alpha/5\}<r<3/2+\alpha$ and $s>s_{\alpha}$.  Let $u \in C([0,T];Z_{s,r}(\mathbb{R}))$ be a solution of the IVP \eqref{BDBO}. If there exist two different times $t_1,t_2 \in [0,T]$ such that 
\begin{equation*}
u(\cdot,t_1), u(\cdot,t_2) \in Z_{s,3/2+\alpha}(\mathbb{R}), 
\end{equation*}
then
\begin{equation*}
\widehat{u}(0,t)=\int u(x,t) \, dx=\int u_0(x)\, dx=\widehat{u_0}(0)=0 \quad \text{ for all } \quad t\in[0,T].
\end{equation*}
\item[(iii)] Assume $0<\alpha<1$, $((3+2\alpha)(12-3\alpha))/(4(10-3\alpha))<r<3/2+\alpha$, and $s\geq \{s_{\alpha}^{+},\alpha r\}$. Let $u \in C([0,T];Z_{s,r}(\mathbb{R}))$ be a solution of the IVP \eqref{BDBO}. If there exist two different times $t_1,t_2 \in [0,T]$ such that 
\begin{equation*}
u(\cdot,t_1), u(\cdot,t_2) \in Z_{s,3/2+\alpha}(\mathbb{R}), 
\end{equation*}
then
\begin{equation*}
\widehat{u}(0,t)=\int u(x,t) \, dx=\int u_0(x)\, dx=\widehat{u_0}(0)=0 \quad \text{ for all } \quad t\in[0,T].
\end{equation*}
\end{itemize}
\end{theorem}
\begin{theorem}\label{Theothretimcondi}
Let $s_{\alpha}=3/2$ for $-1\leq \alpha <0$, and $s_{\alpha}=3/2-3\alpha/8$ for $0<\alpha<1$.
\begin{itemize}
\item[(i)]  Consider $9/10<r<3/2$ and $s>s_{\alpha}$.  Let $u \in C([0,T];Z_{s,r}(\mathbb{R}))$ be a solution of the IVP \eqref{BDBO} for $\alpha=-1$. If there exist two different times $0\leq t_1<t_2 \leq T$ such that 
\begin{equation*}
u(\cdot,t_1) \in \dot{Z}_{s,(3/2)^{+}}(\mathbb{R}) \, \text{ and } \,  u(\cdot,t_2) \in Z_{s,3/2}(\mathbb{R}),
\end{equation*}
then the following identity holds true
\begin{equation}\label{identnegatvalp}
\begin{aligned}
2\sin(t_2-t_1)\big(\int x u(x,t_1)\, dx\big)=(\cos(t_2-t_1)-1)\int |u_0(x)|^2\, dx.
\end{aligned}
\end{equation} 
\item[(ii)] Assume $-1<\alpha<0$, $3/2<r<5/2+\alpha$ and $s>s_{\alpha}$.  Let $u \in C([0,T];Z_{s,r}(\mathbb{R}))$ be a solution of the IVP \eqref{BDBO}. If there exist three different times $t_1,t_2,t_3 \in [0,T]$ such that 
\begin{equation*}
u(\cdot,t_j)  \in \dot{Z}_{s,5/2+\alpha}(\mathbb{R}), \, \, j=1,2,3 \, \, \text{ then } \, \, u \equiv0.
\end{equation*}
\item[(iii)] Assume $0<\alpha<1$, $((5+2\alpha)(12-3\alpha))/(4(10-3\alpha))<r<5/2+\alpha$, and $s\geq \{s_{\alpha}^{+},\alpha r\}$. Let $u \in C([0,T];Z_{s,r}(\mathbb{R}))$ be a solution of the IVP \eqref{BDBO}. If there exist three different times $t_1,t_2,t_2 \in [0,T]$ such that 
\begin{equation*}
u(\cdot,t_j)  \in \dot{Z}_{s,5/2+\alpha}(\mathbb{R}), \, \, j=1,2,3 \, \, \text{ then } \, \, u \equiv0.
\end{equation*}
\end{itemize}
\end{theorem}

The unique continuation principles established above are of independent interest, they are useful to determinate some qualitative properties of solutions of \eqref{BDBO} by knowing only the behavior of the solution for a finite number of times. In our case, these results determine the maximum decay limits allowed by \eqref{BDBO} for arbitrary initial data. As a consequence, we observe that regardless of whether the initial data has enough decay (e.g., Schwartz class) the solution emanating from it must decay of polynomial order determined by the dispersion. We detail this in (c)-(e) in Remark \ref{introdremark1} below. On the other hand, the proof of Theorems \ref{Theotwotimcondi} and \ref{Theothretimcondi} follow in spirit the arguments in \cite{FLinaPioncedGBO}. However, transferring decay to regularity in the frequency domain, our considerations involve the study of fractional derivatives and weights for the function $|\xi|^{m}$, $m<0$. We control such factors by using several properties of fractional derivatives such as Proposition \ref{steinderiweighbet2} below, which can be seen as an extension of the result in \cite[Proposition 2.9]{FLinaPonceWeBO}. We have also included an extra weight in our arguments to consider solutions with lower regularity in the hypothesis of Theorems \ref{Theotwotimcondi} and \ref{Theothretimcondi}. We apply a similar approach in \cite{OscarWHBO}.

\begin{remark}\label{introdremark1}
\begin{itemize}
\item[(a)] Recently,  Kenig, Pilod, Ponce, and Vega \cite{KPVUNI1,KDPVUNI2} deduced some local unique continuation principles for \eqref{BDBO} with $-\alpha \in [-1,2)\setminus\{0\}$. More precisely, they proved that if $u_1(x,t),u_2(x,t)$ are two suitable solutions of the IVP \eqref{BDBO} on $(x,t)\in \mathbb{R}\times [0,T]$, for which there exists some non-empty open set $\Omega \subset \mathbb{R}\times [0,T]$ where $u_1(x,t)=u_2(x,t)$, $(x,t)\in \Omega,$ then it follows that $u_1(x,t)=u_2(x,t)$ for all $(x,t)\in \mathbb{R}\times[0,T]$. A key argument to prove this principle is a global uniqueness result for fractional Schr\"odinger equation, see \cite[Theorem 1.2]{ghosh2016caldern}.

In contrast to the previous result, notice that our conclusion in Theorems \ref{Theotwotimcondi} and \ref{Theothretimcondi} are in a manner global as we require information on the whole spatial variable, but only a finite number of times. Additionally, our results apply to a single solution of \eqref{BDBO}. This can be attributed to the fact that our prove relies on some symmetries of the equation.   

\item[(b)] The unique continuation principles obtained for the Burgers-Hilbert equation are more restrictive than those derived for the case $-1<\alpha$.  This can be justified by the lack of some symmetries for solutions of this equation, including the quantity $I_1(u)$ in \eqref{conservequanti}. 

\item[(c)] Theorem \ref{Theotwotimcondi} (i) establishes that for arbitrary initial data in $u_0 \in Z_{s,r}(\mathbb{R})$ with $s>3/2$, and $r\neq 1/2$, the decay $r=(1/2)^{-}$ is the largest possible for solutions of the Burgers-Hilbert equation. More precisely, for a given initial data $u_0\in Z_{s,(1/2)^{+}}(\mathbb{R})$, the corresponding solution $u(x,t)$ of the IVP \eqref{BDBO} with $\alpha=-1$ satisfies
\begin{equation*}
|x|^{(1/2)^{-}}u \in L^{\infty}([0,T];L^2(\mathbb{R})).
\end{equation*}
Whereas, there does not exist a non-trivial solution $u$ with initial data $u_0$ such that
\begin{equation*}
|x|^{1/2}u \in L^{\infty}([0,T_1];L^2(\mathbb{R})) \, \text{ for some time } T_1>0.
\end{equation*} 
It is worth pointing out that our conclusion does not include the case $r=1/2$.
\item[(d)] Theorem \ref{Theothretimcondi} (i) provides some unique continuation principles for solutions of the Burgers-Hilbert equation $\alpha=-1$. Indeed, if $t_1-t_2=k\pi$ for some positive odd integer number $k$, then $u \equiv 0$. Besides, given $s>3/2$, if there exist three times $t_1<t_2<t_3$ such that $u(\cdot,t_1)\in Z_{s,(3/2)^{+}}(\mathbb{R})$, $u(\cdot,t_j)\in Z_{s,3/2}(\mathbb{R})$, $j=2,3$ and 
\begin{equation*}
\sin(t_2-t_1)\big(1-\cos(t_3-t_1) \big)\neq \sin(t_3-t_1)\big(1-\cos(t_2-t_1) \big),
\end{equation*}
then $u \equiv 0$. Accordingly, Theorem \ref{Theothretimcondi} (i) establishes  that for any initial data $u_0\in Z_{s,r}(\mathbb{R})$, $s>3/2$, $r\neq 3/2$, it follows that $(3/2)^{-}$ is the largest possible decay solutions of the IVP \eqref{BDBO} with $\alpha=-1$. In other words, if $u_0\in Z_{s,(3/2)^{+}}(\mathbb{R})$, then the corresponding solution $u(x,t)$ of the Burgers-Hilbert equation satisfies
\begin{equation*}
|x|^{(3/2)^{-}}u \in L^{\infty}([0,T];L^2(\mathbb{R})).
\end{equation*}
Although, there does not exist a non-trivial solution $u$ corresponding to data $u_0$ such that
\begin{equation*}
|x|^{3/2}u \in L^{\infty}([0,T_1];L^2(\mathbb{R})) \, \text{ for some time } T_1>0.
\end{equation*}
Note that our technique does not cover the case $r=3/2$.
\item[(e)] For the range $-1<\alpha<1,$ $\alpha \neq 0$, Theorem \ref{Theotwotimcondi} (ii) and (iii) shows that the decay $r=(3/2+\alpha)^{-}$ is the largest possible for arbitrary initial data (i.e., without the zero mean assumption). More precisely, Theorem \ref{Theotwotimcondi} (ii) and (iii) show that if $u_0 \in \mathbb{Z}_{s,r}(\mathbb{R})$ with $r\geq 3/2+\alpha$, $s\geq \max\{s_{\alpha}^{+},\alpha r\}$ and $\widehat{u_0}(0)\neq 0$, then the corresponding solution $u=u(x,t)$ verifies
$$|x|^{(3/2+\alpha)^{-}}u\in L^{\infty}([0,T];L^2(\mathbb{R})), \hspace{0.2cm} T>0.$$
However, there does not exists a non-trivial solution $u$ with initial data $u_0$ with $\widehat{u_0}(0) \neq 0$ such that 
$$|x|^{3/2+\alpha}u\in L^{\infty}([0,T'];L^2(\mathbb{R})), \hspace{0.1cm} \text{ for some } T'>0.$$
In this regard, Theorem \ref{Theowellpos} (i) is sharp. In the same manner, Theorem \ref{Theothretimcondi} (ii) and (iii) determine that the decay $r=(5/2+\alpha)^{-}$ is the largest possible $L^2$-polynomial spatial decay rate for solutions of \eqref{BDBO} with $\alpha \in (-1,1)\setminus\{0\}$. 
\end{itemize} 
\end{remark}

The results in Theorem \ref{Theowellpos} and (c)-(e) in Remark \ref{introdremark1} rigorously characterize the $L^2$-spatial behavior of solutions for the equation \eqref{BDBO}. It is confirmed that dispersive effects rather than nonlinearity determine that solutions of \eqref{BDBO} must decay at most of polynomial order whose maximum decay rate is controlled by the dispersion, even under the influence of constant dispersive effects such as $\alpha=-1$. Note that at least for $\alpha>1/3$, we expected this conclusion as the decay rate for the ground state solution in \eqref{decayGroundS} is of polynomial order with dependence on $\alpha$.

As it was stated for the generalized Benjamin-Ono equation in \cite{FLinaPioncedGBO}, one may ask whereas the assumption in Theorem \ref{Theothretimcondi} can be reduced to two different times. In other words, we seek to characterize the $L^2$ spatial behavior of solutions of \eqref{BDBO} by using the minimum number of time information. 
\begin{theorem}\label{reductwotimes}
Let $s_{\alpha}=3/2$ and $3/2<r_{\alpha}<5/2+\alpha$ for $-1< \alpha <0$, and $s_{\alpha}=3/2-3\alpha/8$ and $((5+2\alpha)(12-3\alpha))/(4(10-3\alpha))<r_{\alpha}<5/2+\alpha$ for $0<\alpha<1$. Assume $s\geq \{s_{\alpha}^{+},\alpha r_{\alpha}\}$, and let $u\in C([0,T];\dot{Z}_{s,r_{\alpha}}(\mathbb{R}))$ be a solution of the IVP \eqref{BDBO}. If there exist $t_1,t_2\in [0,T]$, $t_1\neq t_2,$ such that
\begin{equation*}
u(\cdot,t_j) \in Z_{s,5/2+\alpha}(\mathbb{R}), \, j=1,2, 
\end{equation*}
and
\begin{equation}\label{suppreduc}
\int xu(x,t_1)\, dx=0 \hspace{0.5cm} \text{ or } \hspace{0.5cm} \int xu(x,t_2)\, dx=0,
\end{equation}
then
\begin{equation*}
u\equiv 0.
\end{equation*}
\end{theorem}

\begin{theorem}\label{timesharp}
Assume $\alpha \in (-1,1)\setminus\{0\}$. Let $s>5/2$ for $-1<\alpha<0$, and $s>\max\{(3+2\alpha(5+2\alpha))/(1+\alpha(5+2\alpha)),(5/2+\alpha)\alpha+1\}$ for $0<\alpha<1$. Let $u\in C([0,T];\dot{Z}_{s,(5/2+\alpha)^{-}}(\mathbb{R}))$ be a nontrivial solution of \eqref{BDBO} such that 
\begin{equation*}
u_0\in \dot{Z}_{s,5/2+\alpha}(\mathbb{R}) \hspace{0.5cm} \text{ and } \hspace{0.5cm} \int xu_0(x)\, dx \neq 0.
\end{equation*}
Let
\begin{equation*}
t^{\ast}:=-\frac{4}{\left\|u_0\right\|_{L^2}^2}\int x u_0(x)\, dx.
\end{equation*}
If $t^{\ast}\in(0,T]$, then
\begin{equation*}
u(t^{\ast})\in \dot{Z}_{s,5/2+\alpha}(\mathbb{R}).
\end{equation*}
\end{theorem}

\begin{remark}\label{remark3}
\begin{itemize}
\item[(a)] For the case $-1<\alpha<1$, $\alpha \neq 0$, Theorem \ref{reductwotimes} shows that the three times condition can be reduced to two times $t_1 \neq t_2$ assuming \eqref{suppreduc}.  
\item[(b)] For the case $-1<\alpha<1$, $\alpha \neq 0$, Theorem \ref{timesharp} asserts that in general the three times condition cannot be reduced to two times. Thus, the conclusion in Theorem \ref{Theothretimcondi} is optimal. Note that this conclusion combined with \cite{FLinaPioncedGBO,FonPO} shows that the temporal information required to establish the maximum decay rate in $L^2$ for solution of \eqref{BDBO} is independent of the dispersion at least when $-1<\alpha<2$ with $ \alpha \neq 0$.
\\ \\
On the other hand, one question still unanswered is to determinate optimality conditions on the number of times for the Burgers-Hilbert equation $\alpha=-1$. For instance, the result in Theorem \ref{Theothretimcondi} shows that it may be the case that only the trivial solution satisfies certain spatial decay at two different times (see Remark \ref {introdremark1} (c)).

\item[(c)] We consider the IVP associated to the generalized fractional KdV
equation
\begin{equation}\label{gBDBO}
    \left\{\begin{aligned}
    &\partial_t u -\partial_xD^{\alpha} u+u^k\partial_x u =0,\quad (x,t)\in \R^{2}, \, -1\leq  \alpha<1,\, \alpha\neq 0, \, \, k \in \mathbb{Z}^{+}, \\
    &u(x,0)=u_0(x).
    \end{aligned}\right.
\end{equation}
The case $k=2$ has been studied recently, we refer to \cite{2020Sautwang,2020Sautwang1} and the review in \cite{2020KleinSautwang}. Since our estimates for the dispersive term are independent of the nonlinearity, the results in Theorems \ref{Theowellpos} and \ref{Theotwotimcondi} extend to solutions of \eqref{gBDBO} in the spaces $Z_{s,r}(\mathbb{R})$, or $\dot{Z}_{s,r}(\mathbb{R})$ where $s\geq \alpha r$, and such that the solutions generated are also in the class $L^{1}([0,T];L^{\infty}(\mathbb{R}))$ (for instance, one can consider the local theory in $H^s(\mathbb{R})$, $s>3/2$ for solutions of \eqref{gBDBO}). 

 This discussion partly confirms that the spatial behavior of \eqref{gBDBO} is dominated by dispersive effects rather that by nonlinear effects. On the other hand, the equivalent result of Theorems \ref{Theothretimcondi} and \ref{timesharp} for \eqref{gBDBO} depend on a more delicate analysis, for the sake of brevity, we will not carry out this discussion here.
\end{itemize}
\end{remark}

This manuscript is organized as follows. Section \ref{sectionprelim} aims to introduce the notation and preliminaries required to develop our arguments.  In Section \ref{main1}, we prove Theorem \ref{Theowellpos} (i), while parts (ii) and (iii) are established in Section \ref{main2}. Theorems \ref{Theotwotimcondi}, \ref{Theothretimcondi}, \ref{reductwotimes} and \ref{timesharp} are deduced in the following Sections \ref{main3}, \ref{main4}, \ref{main5} and \ref{main6} respectively. We conclude the paper with an appendix where we deduce the commutator estimates in Proposition \ref{fractfirstcaldcomm} and Lemma \ref{lemmacomm3}.

\section{Preliminaries}\label{sectionprelim}

This section is intended to introduce the notation and preliminary results. 

\subsection{Notation}

Given two positive quantities $a$ and $b$,  $a\lesssim b$ means that there exists a positive constant $c>0$ such that $a\leq c b$. We write $a\sim b$ to symbolize that $a\lesssim b$ and $b\lesssim a$.  $[A,B]$ stands for the commutator between two operators $A,B$, that is
\begin{equation*}
[A,B]=AB-BA.
\end{equation*}  
For $x\in \mathbb{R}$, we denote $\langle x\rangle=(1+x^2)^{1/2}$.

Given $1\leq p\leq \infty $, we use $\|\cdot\|_{L^{p}}$ to denote the $L^{p}$-norm. $C_c^{\infty}(\mathbb{R})$ denotes the space of smooth functions of compact support and $S(\mathbb{R})$ stands for the Schwartz space. We denote by  $C^{0,\gamma}(\mathbb{R})$ the space of H\"older continuous functions of order $\gamma \in (0,1]$ equipped with the norm
\begin{equation}\label{eqnot1}
\|f\|_{C^{0,\gamma}}=\|f\|_{L^{\infty}}+[f]_{C^{0,\gamma}}:=\|f\|_{L^{\infty}}+\sup_{\substack{x,y \in \mathbb{R}\\ x\neq y}}\frac{|f(x)-f(y)|}{|x-y|^{\gamma}}.
\end{equation}
For any integer $k\geq 1$ fixed, the H\"older space $C^{k,\gamma}(\mathbb{R})$ is defined by the norm $\|u\|_{C^{k,\gamma}}=\sum_{j=1}^k \|\partial_x^j u\|_{L^{\infty}}+[\partial_x^k u]_{C^{0,\gamma}}$.  The Fourier transform and the inverse Fourier transform of a function $f$ are denoted by $\widehat{f}$ and $f^{\vee}$, respectively. For a given number $s\in \mathbb{R}$, the Bessel operator $J^s$ is defined via the Fourier transform according to $\widehat{J^s \phi}(\xi)=\langle \xi \rangle^s \widehat{f}(\xi)$. The Sobolev spaces $H^s(\mathbb{R})$ consist of all tempered distributions such that $\left\|f\right\|_{H^s}=\left\|J^{s}f \right\|_{L^2}<\infty$. Now, if $A$ denotes a functional space (as those introduced above), we define the spaces $L^p_TA$ as follows
\begin{equation*}
\|f\|_{L^{p}_T A}=\|\|f(\cdot,t)\|_{A}\|_{L^{p}([0,T])},
\end{equation*}
for all $1\leq p\leq \infty$. We consider $\psi \in C^{\infty}_c(\mathbb{R})$ supported on $1/2\leq |\xi|<2$ with $0\leq \psi \leq 1$ and such that $\sum_{j\in \mathbb{Z}} \psi(2^{-j}\cdot)=1$. We define the projector operators by
\begin{equation}\label{projectors}
P_jf=\big(\psi(2^{-j}\cdot)\widehat{f}\big)^{\vee}.
\end{equation}
A function $\phi\in C^{\infty}_c(\mathbb{R})$ with $\phi(\xi)=1$ for $|\xi|\leq 1$ and $\supp(\phi) \subset [-2,2]$ will appear several times in our arguments. We denote by $P^{\phi}$ the operator defined by the Fourier multiplier by the function $\phi(\xi)$ namely
\begin{equation}\label{projectorphi}
P^{\phi}f=\big(\phi\widehat{f}\big)^{\vee}.
\end{equation}

\subsection{Leibniz rule and commutator estimates}

In this part, we introduce the key commutators estimates needed to deal with the dispersive and nonlinear terms present in the equation in \eqref{BDBO}. We first recall the following Leibniz rule for fractional derivatives.

\begin{lemma}\label{leibnizhomog}
For $s>0$, $p\in(1,\infty)$,
\begin{equation}
    \|D^s(fg)\|_{L^p} \lesssim \|f\|_{L^{p_1}}\|D^sg\|_{L^{p_2}} + \|g\|_{L^{p_3}}\|D^sf\|_{L^{p_4}}
\end{equation}
and
\begin{equation}
    \|J^s(fg)\|_{L^p} \lesssim \|f\|_{L^{p_1}}\|J^sg\|_{L^{p_2}} + \|g\|_{L^{p_3}}\|J^sf\|_{L^{p_4}}
\end{equation}
with
\begin{equation*}
    \frac{1}{p} = \frac{1}{p_1} + \frac{1}{p_2} = \frac{1}{p_3} + \frac{1}{p_4}, \, \, \, p_j \in (1,\infty], \, j = 1,\dots, 4. 
\end{equation*}
\end{lemma}
Lemma \ref{leibnizhomog} can be seen as a Kato-Ponce type inequalities (see \cite{KP}), we refer to \cite{FRAC} for a prove of this results. We require the following expansion for fractional derivatives.

\begin{lemma}\label{extendenlemma}
Let $1<p<\infty$ and $1<s<2$, then
\begin{equation}\label{eqextenden1}
\|D^{s}(fg)-fD^{s}g-s\partial_x f \mathcal{H}D^{s-1} g\|_{L^p(\mathbb{R})}\lesssim \|D^{s}f\|_{L^{\infty}(\mathbb{R})}\|g\|_{L^p(\mathbb{R})}.
\end{equation}
If $2\leq s <3$, it holds
\begin{equation}\label{eqextenden2}
\begin{aligned}
\|D^{s}(fg)-fD^{s}g-s\partial_x f \mathcal{H}D^{s-1}g+\frac{s(s-1)}{2}\partial_x^2 fD^{s-2} g\|_{L^p(\mathbb{R})}\lesssim \|D^{s}f\|_{L^{\infty}(\mathbb{R})}\|g\|_{L^p(\mathbb{R})}.
\end{aligned}
\end{equation}
\end{lemma}
The previous lemma is a direct consequence of (1.6) in \cite[Theorem 1.2]{Dli}. Now, motivated by the decomposition $\partial_x D^{\alpha}=-\mathcal{H}D^{1+\alpha}$ for the dispersion in \eqref{BDBO}, our analysis requires the following commutator estimates for the Hilbert transform.

\begin{prop}\label{CalderonComGU}  
Let $1<p<\infty$. Assume $l,m \in \mathbb{Z}^{+}\cup \{0\}$, $l+m\geq 1$ then
\begin{equation}\label{Comwellprel1}
    \|\partial_x^l[\mathcal{H}_x,g]\partial_x^{m}f\|_{L^p(\mathbb{R})} \lesssim_{p,l,m} \|\partial_x^{l+m} g\|_{L^{\infty}(\mathbb{R})}\|f\|_{L^p(\mathbb{R})}.
\end{equation}
Moreover, if $0\leq \alpha, \beta \leq 1$, $\beta>0$ with $\alpha+\beta= 1$, then
\begin{equation}\label{fracComwell}
    \|D_x^{\alpha}[\mathcal{H}_x,g]D_x^{\beta}f\|_{L^p(\mathbb{R})} \lesssim_{p,\alpha.\beta} \|\partial_x g\|_{L^{\infty}(\mathbb{R})}\|f\|_{L^p(\mathbb{R})}.
\end{equation}
\end{prop}

The inequality \eqref{Comwellprel1} was established in \cite[Lemma 3.1]{DawsonMCPON}. We refer to \cite{OscarCapil} for a proof of \eqref{fracComwell}. We also require the following fractional commutator estimate for the Hilbert transform.

\begin{prop}\label{fractfirstcaldcomm} 
Let $1<p<\infty$ and $\beta>0$, then
\begin{equation}\label{Comwellprel12}
\|[\mathcal{H}_x,g]D^{\beta}f\|_{L^p(\mathbb{R})}\lesssim \|D^{\beta}g\|_{L^{\infty}}\|f\|_{L^p}.
\end{equation}
\end{prop}
We prove Proposition \ref{fractfirstcaldcomm} in the Appendix \ref{appendA}. 

\begin{lemma}\label{lemmacomm1} 
Let $0<\beta\leq 1$ and $1<p<\infty$. Then it holds 
\begin{equation}\label{eqlemmacomm1}
    \|[D^{\beta},f]g\|_{L^p}\lesssim \|D^{\beta}f\|_{L^{p}}\|g\|_{L^{\infty}},
\end{equation}
\begin{equation}\label{eq1lemmacomm1} 
    \|[D^{\beta},f]g\|_{L^p}\lesssim \|D^{\beta}f\|_{L^{\infty}}\|g\|_{L^{p}}.
\end{equation}
\end{lemma}

The inequality \eqref{eqlemmacomm1} for $0<\beta<1$ was proved by Kenig, Ponce and Vega in \cite{KPV1993}, we refer to D. Li \cite{Dli} for the case $\beta=1$ and \eqref{eq1lemmacomm1}.

\begin{lemma}\label{lemmacomm2} 
For any $0\leq \beta<1$, $0<\gamma \leq 1-\beta$, $1<p<\infty$, we have
\begin{equation}
    \|D^{\beta}[D^{\gamma},f]D^{1-(\beta+\gamma)}g\|_{L^p}\lesssim \|\partial_x f\|_{L^{\infty}}\|g\|_{L^p}.
\end{equation}
\end{lemma}

The proof of the Lemma \ref{lemmacomm2} can be consulted in \cite[Proposition 3.10]{Dli} (see also \cite{DawsonMCPON}). To apply energy estimated for the dispersions $-1<\alpha\leq 1/2$ in \eqref{BDBO}, we require the following commutator relation for projector operators.

\begin{lemma}\label{lemmacomm3}
Let $\phi\in C^{\infty}_c(\mathbb{R})$ with $\phi(\xi)=1$ for $|\xi|\leq 1$, and $P^{\phi}$ the projector defined in \eqref{projectorphi}. Then for any $\beta\geq 0$, $\gamma >0$, $1<p<\infty$, we have
\begin{equation}
    \|D^{\beta}[P^{\phi},f]D^{\gamma}g\|_{L^p}\lesssim \big(\|D^{\beta+\gamma}f\|_{L^{\infty}}+\|\partial_x f\|_{L^{\infty}}\big)\|g\|_{L^p}.
\end{equation}
Moreover, if $\beta+\gamma=1$, $\|D^{\beta+\gamma}f\|_{L^{\infty}}$ can be replaced by $\|\partial_xf\|_{L^{\infty}}$ in the above inequality.
\end{lemma}
The proof of the above lemma is given in Appendix \ref{appendB} below.

\subsection{Weighted estimates and fractional derivatives}

In this part, we introduce smooth approximations for the weighted function $\langle x\rangle$. We also recall some equivalence for fractional derivatives presented in \cite{SteinThe}, as well as some useful consequences of this definition.  

Following the ideas in \cite{FLinaPioncedGBO} (see also \cite{FLinaPonceWeBO,FonPO}), for any $N\in \mathbb{Z}^{+}$ and $0< \theta \leq 1$, we define the truncated weights $\langle \cdot \rangle_{N}^{\theta} : \mathbb{R} \rightarrow \mathbb{R}$ according to
 \begin{equation*}
\langle x \rangle_{N}^{\theta} =\left\{\begin{aligned} 
 & (1+ |x|^2)^{\theta/2}, \text{ if } |x|\leq N, \\
 &(2N)^{\theta}, \hspace{1,1cm}\text{if } |x|\geq 3N
 \end{aligned}\right.
 \end{equation*}
in such a way that $\langle x \rangle_{N}$ is smooth and non-decreasing in $|x|$ with $\partial_x\langle x \rangle_{N} \leq 1$,  and there exist constants $c_{l}$ independent of $N$ such that $|\partial_x^l\langle x \rangle_{N}| \leq c\partial_x^l\langle x \rangle$, for each $l=2,3$. 

To treat the case $\alpha=-1$ in equation \eqref{BDBO}, we must assure that the computations involving the Hilbert transform and the approximations $\langle x\rangle_N^{\theta}$ are independent of the parameter $N$. In this direction, we have (see \cite{FonPO}):

\begin{prop}\label{propapcond}
For any $\theta \in (-1,1)$ and any $N\in \mathbb{Z}^{+}$, the Hilbert transform is bounded in $L^2(\langle x \rangle_N^{\theta} (x)\, dx)$ with a constant depending on $\theta$ but independent of $N$.
\end{prop}

We recall the following characterization of the spaces $L^p_s(\mathbb{R}^d)=J^{-s}L^p(\mathbb{R}^d)$.
\begin{theorem}(\cite{SteinThe})\label{TheoSteDer}
Let $b\in (0,1)$ and $2d/(d+2b)<p<\infty$. Then $f\in L_b^p(\mathbb{R}^d)$ if and only if
\begin{itemize}
\item[(i)]  $f\in L^p(\mathbb{R}^d)$, 
\item[(ii)]$\mathcal{D}^bf(x)=\left(\int_{\mathbb{R}^d}\frac{|f(x)-f(y)|^2}{|x-y|^{d+2b}}\, dy\right)^{1/2}\in L^{p}(\mathbb{R}^d),$
\end{itemize}
with 
\begin{equation*}
\left\|J^b f\right\|_{L^p}=\left\|(1-\Delta)^{b/2} f\right\|_{L^p} \sim \left\|f\right\|_{L^p}+\left\|\mathcal{D}^b f\right\|_{L^p} \sim \left\|f\right\|_{L^p}+\left\|D^b f\right\|_{L^p}.
\end{equation*} 
\end{theorem}
Next, we present some useful consequences of Theorem \ref{TheoSteDer}. We have the following point-wise inequality   
\begin{equation}\label{prelimneq0}
\mathcal{D}^b(fg)(x) \lesssim \mathcal{D}^b(f)(x)|g(x)|+\|f\|_{L^{\infty}}\mathcal{D}^b(g)(x).
\end{equation}
Whenever $p=2$ and $b\in (0,1)$,
\begin{equation} \label{prelimneq}  
\left\|\mathcal{D}^b(fg)\right\|_{L^2} \lesssim \left\|f\mathcal{D}^b g\right\|_{L^2}+\left\|g\mathcal{D}^bf \right\|_{L^2},
\end{equation}
and it holds
\begin{equation} \label{prelimneq1} 
\left\|\mathcal{D}^{b} h\right\|_{L^{\infty}} \lesssim \big(\left\|h\right\|_{L^{\infty}}+\left\|\nabla h\right\|_{L^{\infty}} \big).
\end{equation} 

As a further application of Theorem \ref{TheoSteDer}, we have 
\begin{lemma}\label{derivexp}   
Let $b\in (0,1)$. Assume that $0<\alpha<1$ and $t>0$, then
\begin{equation}\label{prelimneq2} 
\begin{aligned}
\mathcal{D}^{b}\big(e^{i\xi|\xi|^{\alpha}t}\big)(x) \lesssim \big(|t|^{b/(1+\alpha)}+|t|^{b}|x|^{b\alpha}\big).
\end{aligned}
\end{equation}
For all $t>0$,
\begin{equation}\label{prelimneq2.0}
  \mathcal{D}^b (e^{i\sign(\xi)t})(x)\lesssim |x|^{-b}, \hspace{0.5cm} x\neq 0. 
\end{equation}
If $-1<\alpha<0$ and $t>0$, it follows 
\begin{equation}\label{prelimneq2.1} 
\begin{aligned}
\mathcal{D}^{b}\big(e^{i\xi|\xi|^{\alpha}t}\big)(x) \lesssim F_{b,\alpha,t}(x), \hspace{0.4cm} x \neq 0,
\end{aligned}
\end{equation}
where
\begin{equation*}
\begin{aligned}
F_{b,\alpha,t}(x)=&\left\{\begin{aligned}
&|t|^{b/(1+\alpha)}+t|x|^{1+\alpha-b}, \hspace{0.3cm} \text{if } 1+\alpha-b \neq 0, \text{ or } 1+\alpha-b = 0 \text{ and } |x|\gtrsim 1,   \\
&|t|^{b/(1+\alpha)}-|t|^{b/(1+\alpha)}(-\ln|t^{b/(1+\alpha)}x|\big)^{1/2},  \hspace{0.3cm} \text{if } 1+\alpha-b= 0 \text{ and } |x|\ll 1.
\end{aligned} \right.
\end{aligned}
\end{equation*}
\end{lemma}

The proof of inequality \eqref{prelimneq2} in Lemma \ref{derivexp} is obtained following the arguments in \cite[Proposition 2]{NahPonc}. Similarly, we can adapt these ideas to control the power of the exponential term $\xi|\xi|^{\alpha}$ for $-1<\alpha<0$, deducing \eqref{prelimneq2.1}. The proof of \eqref{prelimneq2.0} is a direct consequence of the definition of the derivative $\mathcal{D}^{b}$.

We need the following interpolation inequality which is proved arguing as in \cite[Lemma 1]{FonPO}:

\begin{lemma}\label{complexinterpo}
Let $a,b>0$. Assume that $J^af=(1-\Delta)^{a/2}f \in L^{2}(\mathbb{R}^d)$ and $\langle x \rangle^b f=(1+|x|^2)^{b/2}f\in L^{2}(\mathbb{R}^d)$, $|x|=\sqrt{x_1^2+\dots x_d^2}$. Then for any $\theta_1 \in (0,1)$, 
\begin{equation}\label{prelimneq3}
\left\|J^{\theta_1 a}(\langle x \rangle^{(1-\theta_1)b}f)\right\|_{L^2}\lesssim \left\|\langle x \rangle^{b} f\right\|_{L^2}^{1-\theta_1}\left\|J^a f\right\|_{L^2}^{\theta_1}.
\end{equation}
Moreover, the inequality \eqref{prelimneq3} is still valid with $\langle x\rangle_N$ instead of $\langle x \rangle$ with a constant $c$ independent of $N$.
\end{lemma}

We shall employ the following result established in \cite[Proposition 2.9]{FLinaPioncedGBO}

\begin{prop}\label{steinderiweighbet}
Let $\phi \in C^{\infty}_c(\mathbb{R})$ with $0\leq \phi \leq 1$, $\phi(\xi)=1$ for $|\xi|\leq 1$. For any $\theta \in (0,1)$ and $\beta>0$,
\begin{equation*}
\mathcal{D}^{\theta}(|\xi|^{\beta}\phi(\xi))(\eta)\sim \left\{\begin{aligned}
&c|\eta|^{\beta-\theta}+c_1 \hspace{0.5cm} \beta \neq \theta, \, |\eta|\ll 1, \\
&c(-\ln|\eta|)^{1/2},  \hspace{0.3cm} \beta = \theta, \, |\eta|\ll 1,\\
&c|\eta|^{-1/2-\theta},  \hspace{0.5cm} \, |\eta|\gg 1,
\end{aligned} 
\right.
\end{equation*}
with $\mathcal{D}^{\theta}(|\xi|^{\beta}\phi(\xi))(\cdot)$ continuous in $\eta \in \mathbb{R}\setminus\{0\}$.
In particular, one has that
\begin{equation*}
\mathcal{D}^{\theta}(|\xi|^{\beta}\phi(\xi))\in L^2(\mathbb{R})\, \text{ if and only if } \, \theta<\beta+1/2.
\end{equation*}
Similar conclusion holds for $\mathcal{D}^{\theta}(\sign(\xi)|\xi|^{\beta}\phi(\xi))(\eta)$.
\end{prop}

We also require the following extension of Proposition \ref{steinderiweighbet} dealing with negative derivatives.
\begin{prop}\label{steinderiweighbet2} 
Let $\psi \in L^{\infty}(\mathbb{R})$ such that $\langle \xi \rangle \partial_{\xi}\psi\in L^{\infty}(\mathbb{R})$. Then for any $\theta \in (0,1)$ and $0<\beta<1/2$ 
\begin{equation}\label{eqsteinderiweig0}
\mathcal{D}^{\theta}(|\xi|^{-\beta}\psi(\xi))(\eta)\lesssim \big(\|\psi\|_{L^{\infty}}+\|\langle \xi \rangle \partial_{\xi} \psi \|_{L^{\infty}}\big)  |\eta|^{-\beta-\theta}, 
\end{equation}
for all $\eta \neq 0$. 

Moreover, let $\theta \in (0,1)$, $\theta<\gamma \leq 1$ and  $\psi \in C^{0,\gamma}(\mathbb{R})$ such that $\psi(0)=0$. Then for all $0<\beta <\gamma+1/2$
\begin{equation}\label{eqsteinderiweig1}
\mathcal{D}^{\theta}(|\xi|^{-\beta}\psi(\xi))(\eta)\lesssim \big(\|\psi\|_{L^{\infty}}+[\psi]_{C^{0,\gamma}})G_{\beta,\gamma,\theta}(\eta),
\end{equation}
for each $\eta \neq 0$, where
\begin{equation*}
G_{\beta,\gamma,\theta}(\eta)=\left\{\begin{aligned}
&1+|\eta|^{\gamma-\beta-\theta}, \hspace{0.3cm} \text{if }\gamma-\beta-\theta > 0, \\
&|\eta|^{\gamma-\beta-\theta}, \hspace{0.7cm} \text{if }\gamma-\beta-\theta <0, \text{ or } \gamma-\beta-\theta =0 \text{ and } |\eta|\gg 1, \\
&(-\ln|\eta|)^{1/2},  \hspace{0.3cm} \text{if }\gamma-\beta-\theta = 0, \, |\eta|\ll 1. 
\end{aligned} 
\right. 
\end{equation*}
Similar conclusions hold for $\mathcal{D}^{\theta}(\sign(\xi)|\xi|^{-\beta}\psi(\xi))$.
\end{prop}

\begin{proof}
To deduce \eqref{eqsteinderiweig0}, we divide our estimates according to
\begin{equation}\label{eqsteinderiweig1.1}
\begin{aligned}
\big(\mathcal{D}^{\theta}(|\xi|^{-\beta}\psi(\xi))(\eta)\big)^2=\int \frac{||\eta|^{-\beta}\psi(\eta)-|y|^{-\beta}\psi(y)|^2}{|\eta-y|^{1+2\theta}}\, dy &= \int_{|y|\leq 2|\eta|}(\cdots) \, dy+\int_{|y|\geq 2|\eta|}(\cdots) \, dy  \\
&=:\mathcal{I}_1+\mathcal{I}_2.
\end{aligned}
\end{equation}
To deal with $\mathcal{I}_1$, we consider a number $m>0$ to write
\begin{equation}\label{eqsteinderiweig2} 
\begin{aligned}
\Big|\frac{\psi(\eta)}{|\eta|^{\beta}}-\frac{\psi(y)}{|y|^{\beta}}\Big|=\Big|\frac{|\eta|^{m}\psi(\eta)}{|\eta|^{\beta+m}}-\frac{|y|^{m}\psi(y)}{|y|^{\beta+m}}\Big|\leq  & \Big|\frac{|\eta|^{m}\psi(\eta)-|y|^{m}\psi(y)}{|\eta|^{\beta+m}}\Big|\\
&+\Big|\frac{1}{|\eta|^{\beta+m}}-\frac{1}{|y|^{\beta+m}}\Big|||y|^{m}\psi(y)|.
\end{aligned}
\end{equation}
Then, by setting $m>1$ in the above display and employing the mean value inequality, it is seen that the right-hand side (r.h.s) of \eqref{eqsteinderiweig2} satisfies
\begin{equation*}
\begin{aligned}
\text{r.h.s}\,\eqref{eqsteinderiweig2}\lesssim &  \big(\|\psi\|_{L^{\infty}}+\|\langle \cdot \rangle\partial_{\eta}\psi\|_{L^{\infty}}\big)\max\{|\eta|^{m-1},|y|^{m-1}\}|\eta|^{-\beta-m}|\eta-y| \\
&+\|\psi\|_{L^{\infty}}\max\{|\eta|^{\beta+m-1},|y|^{\beta+m-1}\}|\eta-y||\eta|^{-\beta-m}|y|^{-\beta}.
\end{aligned}
\end{equation*}
Hence, the preceding inequality yields
\begin{equation*}
\begin{aligned}
\mathcal{I}_1 &\lesssim |\eta|^{-2(\beta+1)}\int_{|y|\leq 2|\eta|}|\eta-y|^{1-2\theta}\, dy+|\eta|^{-2}\int_{|y|\leq 2|\eta|} |\eta-y|^{1-2\theta}|y|^{-2\beta}\, dy\\
&\sim |\eta|^{-2(\beta+\theta)},
\end{aligned}
\end{equation*}
where to assure the integrability of the above expression, we have used that $\theta \in(0,1)$ and $\beta\in (0,1/2)$. On the other hand, since $|\eta-y|\sim |y|$ on the region $|y|\geq 2|\eta|$, we find
\begin{equation*}
\mathcal{I}_2 \lesssim |\eta|^{-2\beta}\|\psi\|_{L^{\infty}}^2\int_{|y|\geq 2|\eta|} |y|^{-1-2\theta}\, dy \sim |\eta|^{-2(\beta+\theta)}.
\end{equation*}
This completes the deduction of \eqref{eqsteinderiweig0}.

Next, we assume that $\psi \in C^{0,\gamma}(\mathbb{R})$ is such that $\psi(0)=0$.  To deduce \eqref{eqsteinderiweig1}, we proceed as in \eqref{eqsteinderiweig1.1}. To treat the first term, we choose $m=\gamma$ in \eqref{eqsteinderiweig2} to obtain
\begin{equation}
\begin{aligned}
\text{r.h.s}\,\eqref{eqsteinderiweig2} 
\lesssim & \, \big( |\eta|^{-\beta}|\eta-y|^{\gamma}+|\eta|^{-\beta-\gamma}|y|^{\gamma}|\eta-y|^{\gamma} \big)[\psi]_{C^{0,\gamma}}\\
&+|y|^{\gamma-\beta}|\eta|^{-\beta-\gamma}||\eta|^{\beta+\gamma}-|y|^{\beta+\gamma}|[\psi]_{C^{0,\gamma}},
\end{aligned}
\end{equation} 
where $[\cdot]_{C^{0,\gamma}}$ is given by \eqref{eqnot1}. Then, we compute
\begin{equation}
\begin{aligned}
\mathcal{I}_1\lesssim  & \, \, |\eta|^{-2\beta}\int_{|y|\leq 2|\eta|}|\eta-y|^{2\gamma-1-2\theta}\, dy+|\eta|^{-2(\beta+\gamma)}\int_{|y|\leq 2 |\eta|} \frac{||\eta|^{\beta+\gamma}-|y|^{\beta+\gamma}|^2|y|^{2(\gamma-\beta)}}{|\eta-y|^{1+2\theta}}\, dy\\
\lesssim & \, \, |\eta|^{-2\beta-2\theta+2\gamma},
\end{aligned}
\end{equation}
where we have used that $\gamma>\theta$ and $\beta<\gamma+1/2$. Now, since $|\eta-y| \sim |y|$ in the support of the integral in $\mathcal{I}_2$, we divide our arguments as follow
\begin{equation}
\begin{aligned}
\mathcal{I}_2 &\lesssim \int_{|y|\geq 2|\eta|} \frac{|\eta|^{-2\beta}|\psi(\eta)|^2}{|y|^{1+2\theta}} \, dy+\int_{|y|\geq 2|\eta|} \frac{|y|^{-2\beta}|\psi(y)|^2}{|y|^{1+2\theta}} \, dy \\
&=: \mathcal{I}_{2,1}+\mathcal{I}_{2,2}.
\end{aligned}
\end{equation}
At once, we get from the assumption $\psi(0)=0$,
\begin{equation}
\begin{aligned}
\mathcal{I}_{2,1}\lesssim |\eta|^{-2\beta+2\gamma}\big(\int_{|y|\geq 2|\eta|}|y|^{-1-2\theta}\, dy\big)[\psi]_{C^{0,\gamma}}^2 \sim |\eta|^{2\gamma-2\beta-2\theta}[\psi]_{C^{0,\gamma}}^2.
\end{aligned}
\end{equation}
We turn to $\mathcal{I}_{2,2}$. We divide this estimate according to the sign of $\gamma-\beta-\theta$. Indeed, if $\gamma-\beta-\theta<0$, by using $|\psi(y)|\leq |y|^{\gamma}[\psi]_{C^{0,\gamma}}$, we deduce
\begin{equation}
\begin{aligned}
\mathcal{I}_{2,2}\lesssim \int_{|y|\geq 2|\eta|} |y|^{2\gamma-2\beta-2\theta-1} \,dy \sim |\eta|^{2\gamma-2\beta-2\theta}.
\end{aligned}
\end{equation}
If $\gamma-\beta-\theta \geq 0$ and $|\eta| \gtrsim 1$,
\begin{equation}
\begin{aligned}
\mathcal{I}_{2,2}\lesssim \|\psi\|_{L^{\infty}}\int_{|y|\geq 2|\eta|}|y|^{-2\beta-2\theta-1} \, dy  \lesssim |\eta|^{2\gamma-2\beta-2\theta}.
\end{aligned}
\end{equation}
Finally, if $\gamma-\beta-\theta \geq 0$ and $0<|\eta| \ll 1$,
\begin{equation}
\begin{aligned}
\mathcal{I}_{2,2} &\lesssim \big(\int_{2|\eta|}^1 |y|^{2\gamma-2\beta-1-2\theta} dy \big)[\psi]_{C^{0,\gamma}}^2+\big(\int_{1}^{\infty} |y|^{-2\beta-2\theta-1} \, dy\big)\|\psi\|_{L^{\infty}}^2 \\
&\sim \int_{2|\eta|}^1 |y|^{2\gamma-2\beta-1-2\theta} dy+1.
\end{aligned}
\end{equation}
This completes the estimate for $\mathcal{I}_{2,2}$ and in turn the proof of the proposition.
\end{proof}

We conclude this section studying bounds for fractional derivatives of $\langle x \rangle_{N}^{\theta}$ and $\langle x \rangle^{\theta}$, whenever $0<\theta<1$.

\begin{prop} \label{propfracweighapp} 
Let $0<\theta<1$ and $\theta<\beta<3$, then $\|D^{\beta}\langle x \rangle_{N}^{\theta}\|_{L^{\infty}}\lesssim 1$ and $\|\mathcal{H}D^{\beta}\langle x \rangle_{N}^{\theta}\|_{L^{\infty}}\lesssim 1$, where the implicit constants are independent of $N$. Similar conclusion holds for $\langle x \rangle^{\theta}$ instead of $\langle x \rangle_{N}^{\theta}$.
\end{prop}
\begin{proof}
We will only control the $L^{\infty}$-norm of $D^{\beta}\langle x \rangle_{N}^{\theta}$, while the  study of $\mathcal{H}D^{\beta}\langle x \rangle_{N}^{\theta}$ follows from the same reasoning. Let us first assume that $0<\beta<1$. Let $0<\delta<\min\{\beta-\theta,1-\beta\}$ and $0<\frac{1}{q}<\delta$ fixed. By Sobolev embedding
\begin{equation}\label{eqlowdisp0} 
\begin{aligned}
\|D^{\beta}\langle x \rangle^{\theta}_N\|_{L^{\infty}} \lesssim \|D^{\beta}\langle x \rangle^{\theta}_N\|_{L^{q}}+\|D^{\beta+\delta}\langle x \rangle^{\theta}_N\|_{L^{q}} \lesssim \|D^{\beta-1}\partial_x\langle x \rangle^{\theta}_N\|_{L^{q}}+\|D^{\beta-1+\delta}\partial_x \langle x \rangle^{\theta}_N\|_{L^{q}}, 
\end{aligned}
\end{equation}
where we have used the decomposition $D=\mathcal{H}\partial_x$, and the fact that $\mathcal{H}$ is a bounded operator on $L^q(\mathbb{R})$. Now, for each $l\in \{0,1\}$, Hardy-Littlewood-Sobolev inequality yields 
\begin{equation*}
\begin{aligned}
\|D^{\beta-1+l\delta}\partial_x\langle x \rangle^{\theta}_N\|_{L^{q}}\sim\|\frac{1}{|\cdot|^{1-(1-\beta-l\delta)}}\ast \partial_x\langle x \rangle^{\theta}_N\|_{L^q}\lesssim \|\partial_x\langle x \rangle^{\theta}_N\|_{L^{p_l}},
\end{aligned}
\end{equation*}
where
\begin{equation*}
\frac{1}{p_l}=\frac{1}{q}+1-\beta-l\delta, \hspace{0.5cm} l=0,1.
\end{equation*}
Consequently, we have $\|\partial_x\langle x \rangle^{\theta}_N\|_{L^{p_l}}\lesssim 1$ with implicit constant independent of $N$, whenever $0<1/p_l<1-\theta$ for each $l=0,1$. However, this last condition can be easily verified from our choice of $\delta$ and $q$.

On the other hand, if $k<\beta<k+1$ for some $k=1,2$, writing $D^{\beta}\langle x \rangle^{\theta}_N=D^{\beta-k-1}\mathcal{H}^{k+1}\partial_x^{k+1}\langle x \rangle^{\theta}_N$, we can adapt the previous argument to obtain the desired conclusion for these restrictions on $\beta$. Finally, for $\beta=1$ or $\beta=2$, setting $D^{\beta}=\mathcal{H}^{\beta}\partial^{\beta}_x$ and using Sobolev embedding $W^{1,q}(\mathbb{R})\hookrightarrow L^{\infty}(\mathbb{R})$ with $0<\frac{1}{q}< 1-\theta$, we complete the proof of Proposition \ref{propfracweighapp}. 
\end{proof}

\section{Proof of Theorem \ref{Theowellpos} (i)} \label{main1}

Let $u_0 \in Z_{s,r}(\mathbb{R})$ with decay $r>0$ and regularity $s \geq \max\{s_{\alpha}^{+},\alpha r\}$, where $s_{\alpha}=3/2$ for $-1\leq \alpha <0$, and $s_{\alpha}=3/2-3\alpha/8$ for $0<\alpha<1$.  In virtue of Theorem \ref{localtheo}, there exist $T>0$ and $u\in C([0,T];H^s(\mathbb{R}))\cap L^1([0,T];W^{1,\infty}(\mathbb{R}))$ solution of the IVP \eqref{BDBO} with initial data $u_0$. Accordingly, to prove Theorem \ref{Theowellpos}, we will only establish the persistence result
\begin{equation}\label{persisteq1}
u \in L^{\infty}([0,T];L^2(|x|^{2r}\, dx)).
\end{equation}
Once we have proved \eqref{persisteq1}, it is not difficult to obtain $u \in C([0,T];L^2(|x|^{2r}\, dx))$ and the continuous dependence on the initial data (see for instance \cite{FLinaPioncedGBO,FonPO,OscarWHBO}). Moreover, by approximating with smooth solutions and taking the limit in our estimates, we will assume that the above solution $u$ of the IVP \eqref{BDBO} is sufficiently regular to perform all the computations developed throughout the proof of Theorem \ref{Theowellpos} below. 

Since our arguments are determined by the size of the parameter $\alpha \in [-1,1)\setminus\{0\}$ in \eqref{BDBO}, we divide our considerations into four main cases: $\alpha=-1$, $-1<\alpha \leq -1/2$, $-1/2<\alpha<0$ and $0<\alpha<1$.

\subsection{Case \texorpdfstring{$\alpha=-1$}{}}

In this case, the equation in \eqref{BDBO} takes the form \eqref{BHeqution}. Then, multiplying \eqref{BHeqution} by $\langle x \rangle^{2r}_N u$ with $0<r<1/2$ and integrating in space we deduce
\begin{equation*}
\frac{1}{2}\frac{d}{dt}\int (\langle x \rangle^{r}_N u)^2\, dx=-\int\langle x \rangle^{r}_N\mathcal{H}u(\langle x \rangle^{r}_N u)\, dx-\int \langle x \rangle^{r}_N(uu_x)(\langle x \rangle^{r}_N u)\, dx.
\end{equation*} 
On one hand, the term provided by the nonlinearity is bounded as follows 
\begin{equation*}
\Big| \int \langle x \rangle^{r}_N(uu_x)(\langle x \rangle^{r}_N u)\, dx\Big| \lesssim \|u_x\|_{L^{\infty}}\|\langle x \rangle^{r}_N u\|^2_{L^2},
\end{equation*}
while on the other hand, in virtue of Proposition \ref{propapcond}, the factor concerning the dispersion is estimate in the following manner
\begin{equation*}
\begin{aligned}
\Big|\int\langle x \rangle^{r}_N\mathcal{H}u(\langle x \rangle^{r}_N u)\, dx \Big|\lesssim \|\langle x \rangle^{r}_N\mathcal{H}u\|_{L^2}\|\langle x \rangle^{r}_N u\|_{L^2}\lesssim \|\langle x \rangle^{r}_N u\|_{L^2}^2.
\end{aligned}
\end{equation*}
Gathering the previous estimates, after applying Gronwall's inequality and taking the limit $N\to \infty$, we get
\begin{equation*}
\sup_{t\in [0,T]}\|\langle x \rangle^{r} u(t)\|_{L^2} \leq c(\|\langle x \rangle^{r}u_0\|_{L^2},\|u\|_{L^1_T W^{1,\infty}}),
\end{equation*}
provided that $0<r<1/2$. This verifies the persistence property in $Z_{s,r}(\mathbb{R})=H^s(\mathbb{R})\cap L^2(|x|^{2r}\, dx)$, $r \in (0,1/2)$, $s>3/2$, for the case $\alpha=-1$.


\subsection*{Preliminary considerations cases \texorpdfstring{$\alpha \in (-1,1)\setminus \{0\}$}{}}

The proof of Theorem \ref{Theowellpos} is obtained from weighted energy estimates applied to the equation in \eqref{BDBO}. More precisely, we set $m\geq 0$ integer and $\theta \in (0,1]$, then after multiplying  equation \eqref{BDBO} by $x^{2m}\langle x\rangle^{2\theta}_Nu$, and integrating in space, we deduce the differential equation 
\begin{equation}\label{grondiffereequ}
\begin{aligned}
\frac{1}{2}\frac{d}{dt}\int \big(x^m \langle x \rangle_N^{\theta} u \big)^2\, dx &- \underbrace{\int (x^m\langle x \rangle_N^{\theta}\partial_xD^{\alpha} u) (x^m\langle x \rangle_N^{\theta} u )\, dx}_{\mathcal{A}_1}+\underbrace{\int uu_x(x^{2m}\langle x \rangle_N^{2\theta}u)\, dx}_{\mathcal{A}_2}=0.  
\end{aligned}
\end{equation}
Roughly speaking, the idea is to write the decay parameter $r=m+\theta$ and provide estimates for \eqref{grondiffereequ}, increasing the values of $m$ and $\theta$ until we cover all the admissible possibilities $0<r<3/2+\alpha$. Certainly, we have 
\begin{equation*}
|\mathcal{A}_2|\lesssim \|\partial_x u\|_{L^{\infty}}\|x^m\langle x \rangle^{\theta}_N u\|_{L^2}^2.
\end{equation*}
Hence, our efforts are reduced to estimate $\mathcal{A}_1$. Since the parameters $m$ and $\theta$ will be precisely stated in each context, for the sake of brevity, we will employ the same notation $\mathcal{A}_1$ for several different cases. Now, we proceed to present our estimates for this factor. 


\subsection{Case \texorpdfstring{$-1<\alpha\leq -1/2$}{}}

Under the present restrictions, we notice that the spatial decay contemplated in Theorem \ref{Theowellpos} satisfies $0<r<3/2+\alpha \leq 1$, then this case is obtained from \eqref{grondiffereequ} setting $m=0$ and $r=\theta$. Thus, we further divide the analysis of $\mathcal{A}_1$ into two additional subcases:
\begin{itemize}
\item[(I)] $0<\theta<1+\alpha$,
\item[(II)] $1+\alpha\leq \theta < 3/2+\alpha \leq 1$.
\end{itemize}


\subsubsection{Case (I): $0< \theta <1+\alpha$}\label{subs1}

To deal with $\mathcal{A}_1$, writing $\partial_xD^{\alpha}=-\mathcal{H}D^{1+\alpha}$, we have 
\begin{equation}\label{claimnegade} 
\begin{aligned}
\begin{aligned}
\langle x \rangle_N^{\theta}\mathcal{H}D^{1+\alpha} u &=[\langle x \rangle^{\theta}_N,\mathcal{H}]D^{1+\alpha}u+\mathcal{H}\big(\langle x \rangle_N^{\theta}D^{1+\alpha} u\big)\\
&=[\langle x \rangle^{\theta}_N,\mathcal{H}]D^{1+\alpha}u+\mathcal{H}\big([\langle x \rangle_N^{\theta},D^{1+\alpha}] u\big)+\mathcal{H}D^{1+\alpha}\big(\langle x \rangle_N^{\theta}u\big) \\
&=\mathcal{A}_{1,1}+\mathcal{A}_{1,2}+\mathcal{H}D^{1+\alpha}\big(\langle x \rangle_N^{\theta}u\big).
\end{aligned}
\end{aligned}
\end{equation}
Due to the fact that $\mathcal{H}D^{1+\alpha}$ determines a skew-symmetric operator, by going back to the integral defining $\mathcal{A}_1$, we have that the last factor on the r.h.s of the preceding equality does not contribute to the estimate. On the other hand, we employ Proposition \ref{fractfirstcaldcomm} to determine
\begin{equation*}
\begin{aligned}
\|\mathcal{A}_{1,1}\|_{L^2}=\|[\langle x \rangle^{\theta}_N,\mathcal{H}]D^{1+\alpha}u\|_{L^2}\lesssim \|D^{1+\alpha}\langle x \rangle^{\theta}_N\|_{L^{\infty}}\|u\|_{L^2},
\end{aligned}
\end{equation*}
and Lemma \ref{lemmacomm1} allows us to conclude
\begin{equation*}
\begin{aligned}
\|\mathcal{A}_{1,2}\|_{L^2}=\|[\langle x \rangle_N^{\theta},D^{1+\alpha}] u\|_{L^2}\lesssim \|D^{1+\alpha}\langle x \rangle^{\theta}_N\|_{L^{\infty}}\|u\|_{L^2}.
\end{aligned}
\end{equation*}
Hence, Proposition \ref{propfracweighapp}  and the preceding bounds complete the proof of Case (I).

\subsubsection{Case (II): $1+\alpha\leq \theta <3/2+\alpha$}\label{subsec(II)}\label{subcaseII}

In contrast with part (I) above, the estimate for $\mathcal{A}_1$ is more involved and requires several modifications. Mainly, we need to increase the number of derivatives in \eqref{BDBO} in a controlled fashion to compensate the fact that $\theta\geq 1+\alpha$. It is worth noticing that  if we were to argue as in Case (I) above, applying Lemma \ref{lemmacomm2} to handle the commutator $[\langle x \rangle_N^{\theta},D^{1+\alpha}] u$, we would need to assure $D^{\alpha}u \in L^2(\mathbb{R})$. However, since $|\xi|^{\alpha}\notin L^{2}_{loc}(\mathbb{R})$ for $-1<\alpha\leq -1/2$, $D^{\alpha}u \in L^2(\mathbb{R})$ is not expected to hold in general. 

Instead, we divide the analysis of $\mathcal{A}_1$ according to low and high frequencies as well as the magnitude of $1+\alpha$. For this purpose, let us consider $\phi \in C^{\infty}_c(\mathbb{R})$ with $\phi(\xi)=1$ if $|\xi|\leq 1$. We will employ the associated operator $P^{\phi}$ defined in \eqref{projectorphi}. We note
\begin{equation}\label{eqlowdisp1}
Id=P^{\phi}+P^{1-\phi},
\end{equation}
where $Id$ denotes the identity operator and $P^{1-\phi}$ is defined by the Fourier multiplier with symbol $1-\phi$. Writing $\partial_xD^{\alpha}=-\mathcal{H}D^{1+\alpha}$, we divide the estimate for $\mathcal{A}_1$ according to
\begin{equation*}
\begin{aligned}
\langle x \rangle_N^{\theta}\mathcal{H}D^{1+\alpha} u =&[\langle x \rangle_N^{\theta},\mathcal{H}]D^{1+\alpha} u +\mathcal{H}\big(\langle x \rangle_N^{\theta}D^{1+\alpha} P^{1-\phi} u \big)+\mathcal{H}\big(\langle x \rangle_N^{\theta}D^{1+\alpha} P^{\phi} u \big)\\
=&[\langle x \rangle_N^{\theta},\mathcal{H}]D^{1+\alpha} u +\mathcal{H}\big([\langle x \rangle_N^{\theta},D^{1+\alpha}] P^{1-\phi} u \big)+\mathcal{H}D^{1+\alpha}\big(\langle x \rangle_N^{\theta} P^{1-\phi} u \big)+\mathcal{H}\big(\langle x \rangle_N^{\theta}D^{1+\alpha} P^{\phi} u \big)\\
=:& \widetilde{\mathcal{A}}_{1,1}+\widetilde{\mathcal{A}}_{1,2}+\widetilde{\mathcal{A}}_{1,3}+\widetilde{\mathcal{A}}_{1,4}.
\end{aligned}
\end{equation*}
By Proposition \ref{fractfirstcaldcomm}, we find 
\begin{equation*}
\begin{aligned}
\|\widetilde{\mathcal{A}}_{1,1}\|_{L^2}=\|[\langle x \rangle_N^{\theta},\mathcal{H}]D^{\theta+\epsilon}D^{1+\alpha-\theta-\epsilon} u\|_{L^2}\lesssim \|D^{\theta+\epsilon}\langle x \rangle_N^{\theta}\|_{L^{\infty}}\|D^{1+\alpha-\theta-\epsilon} u\|_{L^2},
\end{aligned}
\end{equation*}
whenever $0<\epsilon<1-\theta$. Notice that under these restrictions, Proposition \ref{propfracweighapp} assures that $\|D^{\theta+\epsilon}\langle x \rangle_N^{\theta}\|_{L^{\infty}}$ is controlled by a constant independent of $N$. Consequently, to complete the bound of the above inequality, we require to estimate the $L^2(\mathbb{R})$-norm of $D^{1+\alpha-\theta-\epsilon} u$, whether $u$ solves \eqref{BDBO}.

\begin{lemma}\label{claimnegade0.1} 
Assume $-1<\alpha\leq -1/2$. Let $1+\alpha \leq \theta<3/2+\alpha$ and $s>1$. Then there exists $0<\epsilon \ll 1$ such that
\begin{equation}\label{eqlowdisp0.1} 
\sup_{t\in[0,T]}\|D^{1+\alpha-\theta-\epsilon}u(t)\|_{L^2}\leq c(\|\langle x \rangle^{\theta}u_0\|_{L^2},\|u\|_{L^{\infty}_T H^s},\|u\|_{L^{1}_T W^{1,\infty}}).
\end{equation}
\end{lemma}
\begin{proof}
Let $\theta-(1+\alpha)<\beta<\min\{1/2,\theta\}$ and $0<\epsilon<\min\{1-\theta,1+\alpha-\theta+\beta\}$. We emphasize that such $\beta$ exists since $1+\alpha \leq \theta <3/2+\alpha$. Fixing these parameters, we begin by verifying \eqref{eqlowdisp0.1} for $t=0$. By Plancherel's identity, H\"older's inequality and Sobolev embedding, we get
\begin{equation}\label{eqlowdisp0.2}
\begin{aligned}
\|D^{1+\alpha-\theta-\epsilon} u_0\|_{L^2}&\lesssim \||\xi|^{1+\alpha-\theta-\epsilon}\phi \widehat{u_0}\|_{L^2}+\||\xi|^{1+\alpha-\theta-\epsilon}(1-\phi) \widehat{u_0}\|_{L^2} \\
&\lesssim \||\xi|^{1+\alpha-\theta-\epsilon}\phi\|_{L^{1/\beta}}\|\widehat{u_0}\|_{L^{2/(1-2\beta)}}+\|\widehat{u_0}\|_{L^2}\\
&\lesssim \|J^{\beta}\widehat{u_0}\|_{L^{2}}+\|\widehat{u_0}\|_{L^2} \lesssim \|\langle x \rangle^{\theta}u_0\|_{L^2},
\end{aligned} 
\end{equation}
which holds due to the fact that $|\xi|^{1+\alpha-\theta-\epsilon}\phi\in L^{1/\beta}(\mathbb{R})$. Now,  to obtain \eqref{eqlowdisp0.1} for arbitrary time $t\in(0,T]$, we proceed to perform energy estimates. Indeed, by applying $D^{1+\alpha-\theta-\epsilon}$ to the equation in \eqref{BDBO}, multiplying by $D^{1+\alpha-\theta-\epsilon} u$ and integrating in space, we arrive at
\begin{equation}\label{eqlowdisp1.1}
\frac{1}{2}\frac{d}{dt}\int (D^{1+\alpha-\theta-\epsilon} u)^2\, dx =-\int D^{1+\alpha-\theta-\epsilon}(u_xu)D^{1+\alpha-\theta-\epsilon}u \, dx,
\end{equation}
where to cancel the factor provided by the dispersion, we have used that $\mathcal{H}D^{1+\alpha}$ is skew-symmetric. By writing $u \partial_x u=\frac{1}{2}\partial_x(u^2)$ and applying Lemma \ref{leibnizhomog}, we find
\begin{equation*}
\begin{aligned}
\Big|\int D^{1+\alpha-\theta-\epsilon}(uu_x)D^{1+\alpha-\theta-\epsilon}u \, dx\Big| &\lesssim \|D^{2+\alpha-\theta-\epsilon}(u^2)\|_{L^2}\|D^{1+\alpha-\theta-\epsilon}u \|_{L^2} \\
&\lesssim \|u\|_{L^{\infty}}\|D^{2+\alpha-\theta-\epsilon}u\|_{L^2}\|D^{1+\alpha-\theta-\epsilon}u \|_{L^2}.
\end{aligned}
\end{equation*}
We observe that $\|D^{2+\alpha-\theta-\epsilon}u\|_{L^2}\leq \|J^s u\|_{L^2}$, whenever $s>1$. Therefore, plugging the above estimate in \eqref{eqlowdisp1.1} and applying Gronwall's inequality the proof is complete.
\end{proof}
We turn to $\widetilde{\mathcal{A}}_{1,2}$. An application of Lemma \ref{lemmacomm2} reveals
\begin{equation*}
\begin{aligned}
\|\widetilde{\mathcal{A}}_{1,2}\|_{L^2}&\lesssim \|[\langle x \rangle_N^{\theta},D^{1+\alpha}] D^{-\alpha}D^{\alpha }P^{1-\phi} u \|_{L^2} \lesssim \|D^{\alpha}P^{1-\phi}u\|_{L^2} \lesssim \|u\|_{L^2},
\end{aligned}
\end{equation*}
where given that $\alpha<0$ and $1-\phi$ is supported outside of the origin, we have used that $\|D^{\alpha}P^{1-\phi}u\|_{L^2} \lesssim_{\alpha} \|u\|_{L^2}$. Now, we consider $\widetilde{\mathcal{A}}_{1,3}$. Returning to the integral defining $\mathcal{A}_2$, we employ the fact that $\mathcal{H}D^{1+\alpha}$ is skew-symmetric and \eqref{eqlowdisp1} to get
\begin{equation*}
\begin{aligned}
\int \widetilde{\mathcal{A}}_{1,3}\langle x \rangle_N^{\theta} u\, dx &= \int \mathcal{H}D^{1+\alpha}\big(\langle x \rangle_N^{\theta} P^{1-\phi} u \big) \langle x \rangle_N^{\theta} u \, dx\\
&=\int \mathcal{H}D^{1+\alpha}\big(\langle x \rangle_N^{\theta}  u \big) \langle x \rangle_N^{\theta} P^{\phi}u \, dx\\
&=\int \mathcal{H}\big(\langle x \rangle_N^{\theta}  u \big)D^{1+\alpha}[\langle x \rangle_N^{\theta}, P^{\phi}]u \, dx+\mathcal{H}\big(\langle x \rangle_N^{\theta}  u \big)D^{1+\alpha}P^{\phi}\big( \langle x \rangle_N^{\theta} u\big) \, dx.
\end{aligned}
\end{equation*}
An application of Lemma \ref{lemmacomm3} then shows
\begin{equation*}
\begin{aligned}
\Big|\int \widetilde{\mathcal{A}}_{1,3}\langle x \rangle_N^{\theta} u\, dx \Big|&\lesssim\big( \|D^{1+\alpha}[\langle x \rangle_N^{\theta}, P^{\phi}]u\|_{L^2}+\|D^{1+\alpha}P^{\phi}\big( \langle x \rangle_N^{\theta} u\big)\|_{L^2} \big)\|\langle x \rangle_N^{\theta} u\|_{L^2} \\
&\lesssim \|D^{1+\alpha}[\langle x \rangle_N^{\theta}, P^{\phi}]D^{\theta+\epsilon-1-\alpha}D^{1+\alpha-\theta-\epsilon}u\|_{L^2}\|\langle x \rangle_N^{\theta} u\|_{L^2}+\|\langle x \rangle_N^{\theta} u\|_{L^2}^{2}\\
&\lesssim \big(\|D^{\theta+\epsilon}\langle x \rangle^{\theta}\|_{L^{\infty}}+\|\partial_x \langle x \rangle^{\theta}\|_{L^{\infty}}\big)\|D^{1+\alpha-\theta-\epsilon}u\|_{L^2}\|\langle x \rangle_N^{\theta} u\|_{L^2}+\|\langle x \rangle_N^{\theta} u\|_{L^2}^{2},
\end{aligned}
\end{equation*}
where we have also used that $\|D^{\beta}P^{\phi}f\|_{L^2}\lesssim_{\beta}\|f\|_{L^2}$. Notice that the above estimate and Lemma \ref{claimnegade0.1} complete the analysis of $\widetilde{\mathcal{A}}_{1,3}$.

To estimate  $\widetilde{\mathcal{A}}_{1,4}$, we shall employ the equation \eqref{BDBO} to obtain new deferential inequalities which provide a closed differential form that ultimately control this factor. Roughly, the idea is to increase the number of derivatives considered in the equation until we can control the commutator $[\langle x \rangle_N^{\theta},D^{1+\alpha}]$. In doing so, given $\alpha \in (-1,1/2]$, we choose an integer $\kappa\geq 2$ fixed such that 
\begin{equation*}
0<\kappa(1+\alpha)+\alpha \,  \text{ and } \, \kappa(1+\alpha)\leq 3/2, \, \text{ in other words, } \, \alpha\in (-\kappa/(\kappa+1),-(\kappa-3/2)/\kappa].
\end{equation*}
 Thus, rewriting $\partial_x D^{\alpha}=-\mathcal{H}D^{1+\alpha}$, we apply $D^{l(1+\alpha)}P^{\phi}$ to the equation in \eqref{BDBO}, multiplying the resulting expression by $\langle x \rangle_N^{2\theta}D^{l(1+\alpha)}P^{\phi}u $, we get
\begin{equation}\label{diferenequalderivlocal}
\begin{aligned}
\frac{1}{2}\frac{d}{dt}\int \big(\langle x \rangle_N^{\theta}D^{l(1+\alpha)}P^{\phi} u \big)^2\, dx &+ \underbrace{\int (\langle x \rangle_N^{\theta}\mathcal{H}D^{1+\alpha}D^{l(1+\alpha)}P^{\phi} u) (\langle x \rangle_N^{\theta}D^{l(1+\alpha)}P^{\phi}u)\, dx}_{\mathcal{B}_{1,l}}\\
&+\underbrace{\int \langle x \rangle_N^{\theta} D^{l(1+\alpha)}P^{\phi}(uu_x)(\langle x \rangle_N^{\theta}D^{l(1+\alpha)}P^{\phi}u)\, dx}_{\mathcal{B}_{2,l}}=0,
\end{aligned}
\end{equation}
for each $l=1,\dots,\kappa$. Before analyzing $\mathcal{B}_{1,l}$ and $\mathcal{B}_{2,l}$, let us show that
\begin{equation}\label{eqlowdisp2}
\|\langle x \rangle^{\theta}D^{l(1+\alpha)}P^{\phi} u_0\|_{L^2}\lesssim \|\langle x \rangle^{\theta}u_0\|_{L^2},
\end{equation}
for all $l=1,\dots,\kappa$. This in turn verifies the required estimate in \eqref{diferenequalderivlocal} for the initial data. By Plancherel's identity and Theorem \ref{TheoSteDer}, it is seen that
\begin{equation*}
\begin{aligned}
\|\langle x \rangle^{\theta}D^{l(1+\alpha)}P^{\phi} u_0\|_{L^2}&=\|J^{\theta}_{\xi}\big(|\xi|^{l(1+\alpha)}\phi \widehat{u_0}\big)\|_{L^2}\\
& \lesssim \||\xi|^{l(1+\alpha)}\phi \widehat{u_0}\|_{L^2}+\|\mathcal{D}^{\theta}\big(|\xi|^{l(1+\alpha)}\phi\big) \widehat{u_0}\|_{L^2}+\||\xi|^{l(1+\alpha)}\phi\|_{L^{\infty}}\|\mathcal{D}^{\theta}\widehat{u_0}\|_{L^2} \\
&\lesssim \|\langle x \rangle^{\theta} u_0\|_{L^2}+\|\mathcal{D}^{\theta}\big(|\xi|^{l(1+\alpha)}\phi\big) \widehat{u_0}\|_{L^2}.
\end{aligned}
\end{equation*}
Now, if $l(1+\alpha)\geq 1$, \eqref{prelimneq1} yields
\begin{equation*}
\begin{aligned}
\|\mathcal{D}^{\theta}\big(|\xi|^{l(1+\alpha)}\phi\big) \widehat{u_0}\|_{L^2}&\leq \|\mathcal{D}^{\theta}\big(|\xi|^{l(1+\alpha)}\phi\big)\|_{L^{\infty}}\|u_0\|_{L^2} \\
&\lesssim \big(\||\xi|^{l(1+\alpha)}\phi\|_{L^{\infty}}+\|\partial_x \big(|\xi|^{l(1+\alpha)}\phi\big)\|_{L^{\infty}}\big)\|u_0\|_{L^2}\\
&\lesssim \|u_0\|_{L^2}.
\end{aligned}
\end{equation*}
We suppose that $l(1+\alpha)<1$, then we consider $\max\{0,\theta-l(1+\alpha)\}<\beta<\min\{\theta,1/2\}$, so that by H\"older inequality and Sobolev embedding,
\begin{equation*}
\begin{aligned}
\|\mathcal{D}^{\theta}\big(|\xi|^{l(1+\alpha)}\phi\big) \widehat{u_0}\|_{L^2}\leq \|\mathcal{D}^{\theta}\big(|\xi|^{l(1+\alpha)}\phi\big)\|_{L^{1/\beta}}\|\widehat{u_0}\|_{L^{2/(1-2\beta)}} &\leq \|\mathcal{D}^{\theta}\big(|\xi|^{l(1+\alpha)}\phi\big)\|_{L^{1/\beta}}\|J^{\beta}\widehat{u_0}\|_{L^2}\\
&\lesssim \|\langle x \rangle^{\theta}u_0\|_{L^2}. 
\end{aligned}
\end{equation*}
By our choice of $\beta$ and Proposition \ref{steinderiweighbet}, we employed above that $\mathcal{D}^{\theta}\big(|\xi|^{l(1+\alpha)}\phi\big)\in L^{1/\beta}(\mathbb{R})$. This completes the deduction of \eqref{eqlowdisp2}. Now, we turn to the estimates for $\mathcal{B}_{1,l}$ and $\mathcal{B}_{2,l}$ in \eqref{diferenequalderivlocal}. 

We first deal with the factors $\mathcal{B}_{2,l}$ for those integer numbers $l=1,\dots,\kappa$ for which $l(1+\alpha)<1$. We write
\begin{equation}\label{eqlowdisp2.1}
\begin{aligned}
\langle x \rangle^{\theta}_N D^{l(1+\alpha)}P^{\phi}(uu_x)=&[\langle x \rangle^{\theta}_N,D^{l(1+\alpha)}]P^{\phi}(uu_x)+D^{l(1+\alpha)}[\langle
x \rangle_N^{\theta},P^{\phi}](uu_x)\\
&+D^{l(1+\alpha)}P^{\phi}\big(\langle x \rangle^{\theta}_N(uu_x)\big).
\end{aligned}
\end{equation}
Hence, by Lemmas \ref{lemmacomm2} and \ref{lemmacomm3},
\begin{equation}\label{eqlowdisp3}
\begin{aligned}
\|\langle x \rangle^{\theta}_N D^{l(1+\alpha)}&P^{\phi}(uu_x)\|_{L^2}\\
\lesssim &\|[\langle x \rangle^{\theta}_N,D^{l(1+\alpha)}]D^{1-l(1+\alpha)}D^{l(1+\alpha)-1}P^{\phi}(uu_x)\|_{L^2}\\
&+\|D^{l(1+\alpha)}[\langle
x \rangle_N^{\theta},P^{\phi}]D^{1-l(1+\alpha)}D^{l(1+\alpha)-1}(uu_x)\|_{L^2}+\|D^{l(1+\alpha)}P^{\phi}\big(\langle x \rangle^{\theta}_N(uu_x)\big)\|_{L^2}\\
\lesssim & \|\partial_x\langle x \rangle^{\theta}_N \|_{L^{\infty}}\big(\|D^{l(1+\alpha)-1}\partial_xP^{\phi}(u^2)\|_{L^2}+\|D^{l(1+\alpha)-1}\partial_x(u^2)\|_{L^2}\big)+\|\langle x \rangle^{\theta}_N(uu_x)\|_{L^2}\\
\lesssim & \|u\|_{L^{\infty}}\big(\|u\|_{L^2}+\|D^{l(1+\alpha)}u\|_{L^2}\big)+\|u_x\|_{L^{\infty}}\|\langle x \rangle^{\theta}_Nu\|_{L^2},
\end{aligned}
\end{equation}
where we have used the Leibniz rule of Lemma \ref{leibnizhomog}. Notice that our choice of $\kappa$ yields $\|D^{l(1+\alpha)}u\|_{L^2} \leq \|u\|_{H^{(3/2)^{+}}}$, for all $l=1,\dots, \kappa$. We find that the last term on the r.h.s of the above inequality is the term to be estimated using \eqref{grondiffereequ}. Formally, this suggests that the estimates for \eqref{diferenequalderivlocal} for all $l=1,\dots,\kappa$ should be combined with those of \eqref{grondiffereequ} to obtain a closed Gronwall's inequality.

Now, we consider $\mathcal{B}_{2,l}$ for those integers $l=1,\dots,\kappa$ for which $l(1+\alpha)>1$. Notice that the definition of $\kappa$ implies $0<l(1+\alpha)-1<1$. Hence, since $D=\mathcal{H}\partial_x$, we have
\begin{equation}\label{eqlowdisp3.1}
\begin{aligned}
\langle x \rangle^{\theta}_N D^{l(1+\alpha)}P^{\phi}(uu_x)=&[\langle x \rangle^{\theta}_N,\mathcal{H}]D^{l(1+\alpha)-1}\partial_xP^{\phi}(uu_x)+\mathcal{H}\big([\langle x \rangle^{\theta}_N,D^{l(1+\alpha)-1}]\partial_xP^{\phi}(uu_x)\big) \\
&+\mathcal{H}D^{l(1+\alpha)-1}\big([\langle x \rangle^{\theta}_N,P^{\phi}]\partial_x(uu_x)\big)-\mathcal{H}D^{l(1+\alpha)-1}P^{\phi}\big(\partial_x\langle x \rangle^{\theta}_N(uu_x)\big)\\
&+\mathcal{H}D^{l(1+\alpha)-1}\partial_xP^{\phi}\big(\langle x \rangle^{\theta}_N(uu_x)\big).
\end{aligned}
\end{equation}
By similar reasoning as in \eqref{eqlowdisp3}, applying Proposition \ref{CalderonComGU}, Lemmas \ref{leibnizhomog}, \ref{lemmacomm2} and \ref{lemmacomm3}, we deduce
\begin{equation*}
\begin{aligned}
\|\langle x \rangle^{\theta}_N  D^{l(1+\alpha)}P^{\phi}(uu_x)\|_{L^2}\lesssim & \|\partial_x \langle x \rangle^{\theta}_N\|_{L^{\infty}}\big(\|uu_x\|_{L^2}+\|u^2\|_{L^2}+\|D^{l(1+\alpha)-1}(uu_x)\|_{L^2}\big)+\|\langle x \rangle^{\theta}_N(uu_x)\|_{L^2}\\
\lesssim & \big(\|u\|_{L^{\infty}}+\|u_x\|_{L^{\infty}}\big)\big(\|u\|_{L^{2}}+\|D^{l(1+\alpha)}u\|_{L^{2}}+\|\langle x \rangle^{\theta}_Nu\|_{L^2}\big).
\end{aligned}
\end{equation*}
The previous result completes the study of $\mathcal{B}_{2,l}$ for all $l=1,\dots,\kappa$. We proceed to control $\mathcal{B}_{1,l}$. We first treat the cases $l=1,\dots,\kappa-1$. We consider the following decomposition
\begin{equation}\label{eqlowdisp4.1} 
\begin{aligned}
 \langle x \rangle_N^{\theta}\mathcal{H}D^{(l+1)(1+\alpha)}P^{\phi} u  &= [\langle x \rangle_N^{\theta},\mathcal{H}]D^{(l+1)(1+\alpha)}P^{\phi} u+\mathcal{H}\big(\langle x \rangle_N^{\theta}D^{(l+1)(1+\alpha)}P^{\phi}u\big).
\end{aligned}
\end{equation}
The last term on the r.h.s of the above identity is the estimate to be controlled by \eqref{diferenequalderivlocal} with $l+1$, while the first term on the r.h.s is controlled by Proposition \ref{fractfirstcaldcomm} in the following manner
\begin{equation*}
\begin{aligned}
\|[\langle x \rangle_N^{\theta},\mathcal{H}]D^{(l+1)(1+\alpha)}P^{\phi} u\|_{L^2}&\lesssim \|D^{\theta+\epsilon}\langle x \rangle_N^{\theta}\|_{L^{\infty}}\|D^{l(1+\alpha)}P^{\phi}D^{1+\alpha-\theta-\epsilon}u\|_{L^2}\\
&\lesssim \|D^{1+\alpha-\theta-\epsilon}u\|_{L^2},
\end{aligned}
\end{equation*}
where $0<\epsilon \ll 1$ is taken according to Proposition \ref{propfracweighapp} and Lemma \ref{claimnegade0.1}. Therefore, gathering the previous estimates, it remains to study $\mathcal{B}_{1,\kappa}$. To control this factor, we employ the following decomposition
\begin{equation}\label{eqlowdisp5}
\begin{aligned}
 \langle x \rangle_N^{\theta}\mathcal{H}D^{1+\alpha}D^{\kappa(1+\alpha)}P^{\phi} u =& [\langle x \rangle_N^{\theta},\mathcal{H}]D^{(\kappa+1)(1+\alpha)}P^{\phi} u +\mathcal{H}\big([\langle x \rangle_N^{\theta},D^{1+\alpha}]D^{\kappa(1+\alpha)}P^{\phi}u\big)\\
&+\mathcal{H}D^{1+\alpha}\big(\langle x \rangle_N^{\theta} D^{\kappa(1+\alpha)}P^{\phi}u\big).
\end{aligned}
\end{equation}
Going back to $\mathcal{B}_{1,\kappa}$, we notice that the last term on the r.h.s of the above inequality does not contribute to the differential inequality. By applying Proposition \ref{fractfirstcaldcomm}, the first term on the r.h.s of \eqref{eqlowdisp5} is bounded as follows
\begin{equation*}
\begin{aligned}
\| [\langle x \rangle_N^{\theta},\mathcal{H}]D^{(\kappa+1)(1+\alpha)}P^{\phi} u \|_{L^2}&=\| [\langle x \rangle_N^{\theta},\mathcal{H}]D^{\theta+\epsilon}D^{(\kappa+1)(1+\alpha)-\theta-\epsilon}P^{\phi} u \|_{L^2}\\
&\lesssim \|D^{\theta+\epsilon}\langle x \rangle_N^{\theta}\|_{L^{\infty}}\|D^{\kappa(1+\alpha)}P^{\phi}D^{1+\alpha-\theta-\epsilon} u \|_{L^2} \\&
\lesssim\|D^{1+\alpha-\theta-\epsilon} u \|_{L^2} 
\end{aligned}
\end{equation*}
which is controlled provided that $0<\epsilon \ll 1$ is given by Proposition \ref{propfracweighapp} and Lemma \ref{claimnegade0.1}. Next, we apply Lemma \ref{lemmacomm2} and the fact that $\kappa(1+\alpha)+\alpha>0$ to get
\begin{equation*}
\begin{aligned}
\|\mathcal{H}\big([\langle x \rangle_N^{\theta},D^{1+\alpha}]D^{k(1+\alpha)}P^{\phi}u\big)\|_{L^2}\lesssim \|\partial_x \langle x \rangle_N^{\theta} \|_{L^{\infty}}\|D^{\kappa(1+\alpha)+\alpha}P^{\phi}u\|_{L^2} \lesssim \|u\|_{L^2}.
\end{aligned}
\end{equation*}
This completes the estimate for \eqref{eqlowdisp5} and in consequence the analysis of $\mathcal{B}_{1,k}$ for all $l=1,\dots,\kappa$.

Recalling the integer $\kappa\geq 2$  such that $0<\kappa(1+\alpha)+\alpha$ and $\kappa(1+\alpha)\leq 3/2$, i.e., $\alpha\in (-\kappa/(\kappa+1),-(\kappa-3/2)/\kappa]$, we let
\begin{equation*}
\mathfrak{g}(t)= \|\langle x \rangle^{\theta}_N u\|_{L^2}^2+\sum_{l=1}^{\kappa}\|\langle x \rangle^{\theta}_N D^{l(1+\alpha)}P^{\phi} u\|_{L^2}^2.
\end{equation*}
Therefore, by collecting all the preceding estimates obtained through the Case (II), we arrive at the differential inequality
\begin{equation}\label{resuldiff}
\begin{aligned}
\frac{1}{2}\frac{d}{dt}\mathfrak{g}(t) \lesssim &\big(\|D^{1+\alpha-\theta-\epsilon}u\|_{L^2}+(1+\|u\|_{L^{\infty}}+\|u_x\|_{L^{\infty}})\|u\|_{H^s}\big)\mathfrak{g}^{1/2}(t)\\
&+(1+\|u\|_{L^{\infty}}+\|u_x\|_{L^{\infty}})\mathfrak{g}(t).
\end{aligned}
\end{equation}
Since \eqref{eqlowdisp2} establishes that $\mathfrak{g}(0)$ makes sense, by applying Gronwall's inequality to the preceding equation, followed by letting $N\to \infty$, we complete the proof of Theorem \ref{Theowellpos} (i) whenever $-1<\alpha\leq-1/2$.


\subsection*{Case \texorpdfstring{$-1/2<\alpha<0$}{}}

In this case, the maximum decay rate satisfies $1<3/2+\alpha<2$, so to treat weights $0< r=\theta \leq 1$, we will employ the differential identity \eqref{grondiffereequ} first with $m=0$ and $\theta \in (0,1]$, and then with $m=1$ and $0<\theta<1/2+\alpha$, we cover the range $1<r=1+\theta<3/2+\alpha$.

We begin with the considerations for $m=0$, $\theta=r\in (0,1]$ in \eqref{grondiffereequ}. We split the analysis of $\mathcal{A}_1$ exactly as in \eqref{claimnegade}.  Noticing that $\theta<1+\alpha$, whenever $0< \theta \leq 1/2$, we have that the same arguments dealing with Case (I) in Subsection \ref{subs1} provide the desired conclusion for this restriction on $\theta$. Thereby, we shall assume that $1/2< \theta \leq 1$. We proceed to study $\mathcal{A}_{1,1}$ and $\mathcal{A}_{1,2}$ defined in \eqref{claimnegade}.

By Proposition \ref{CalderonComGU}, we observe
\begin{equation}\label{eqmainlow0}
\begin{aligned}
\|\mathcal{A}_{1,1}\|_{L^2}&=\|[\mathcal{H},\langle x \rangle^{\theta}_N]\partial_x \mathcal{H}D^{\alpha}u\|_{L^2}\lesssim \|\partial_x \langle x \rangle^{\theta}_N \|_{L^{\infty}}\|D^{\alpha}u\|_{L^2},
\end{aligned}
\end{equation}
and by Lemma \ref{lemmacomm2},
\begin{equation}\label{eqmainlow0.1}
\begin{aligned}
\|\mathcal{A}_{1,2}\|_{L^2}&=\|[D^{1+\alpha},\langle x \rangle^{\theta}_N]D^{-\alpha}D^{\alpha}u\|_{L^2}\lesssim \|\partial_x \langle x \rangle^{\theta}_N \|_{L^{\infty}}\|D^{\alpha}u\|_{L^2}. 
\end{aligned}
\end{equation}
By construction $|\partial_x \langle x \rangle^{\theta}_N| \lesssim 1$ with implicit constant independent of $N$. Then, it only remains to analyze $\|D^{\alpha}u\|_{L^2}$.  

\begin{lemma}\label{claimnegade1} 
Let $-1/2<\alpha<0$ and $s>1+\alpha$. For each $\theta>1/2$, it holds
\begin{equation*}
\sup_{t\in[0,T]}\|D^{\alpha}u(t)\|_{L^2} \leq c\big(\|\langle x \rangle^{\theta}u_0\|_{L^2},\|u\|_{L^{\infty}_TH^s},\|u\|_{L^{1}_TW^{1,\infty}}\big).
\end{equation*}
\end{lemma}

\begin{proof}
The proof follows from energy estimates applied to the equation in \eqref{BDBO}. We first control the $L^2$-norm of $D^{\alpha}u_0$. Since $\theta>1/2$, by Sobolev embedding and Plancherel's identity we get
\begin{equation}\label{eqmainlow1}
|\widehat{u_0}(\xi)|\lesssim \|J_{\xi}^{\theta}\widehat{u_0}\|_{L^2}=\|\langle x \rangle^{\theta}u_0\|_{L^2}, \, \text{ for all } \, \xi \in \mathbb{R}.
\end{equation}
Let $\phi \in C^{\infty}_c(\mathbb{R})$ such that $0\leq \phi \leq 1$ and $\phi(\xi)=1$ for $|\xi|\leq 1$. Given that $-1/2<\alpha<0$, it follows that $|\xi|^{\alpha}\phi \in L^2_{loc}(\mathbb{R})$, then this fact and \eqref{eqmainlow1} show 
\begin{equation*}
\begin{aligned}
\|D^{\alpha}u_0\|_{L^2} 
&\lesssim \||\xi|^{\alpha}\phi\|_{L^2}\|\widehat{u_0}\|_{L^{\infty}}+\||\xi|^{\alpha}(1-\phi)\|_{L^{\infty}}\| \widehat{u_0}\|_{L^2} \\
&\lesssim \|\langle x \rangle^{\theta}u_0\|_{L^2}.
\end{aligned}
\end{equation*}
This proves $D^{\alpha}u_0\in L^2(\mathbb{R})$. We continue by applying $D^{\alpha}$ to the equation in \eqref{BDBO} and multiplying by $D^{\alpha}u$ to get 
\begin{equation}\label{eqmainlow2}
\begin{aligned}
\frac{1}{2}\frac{d}{dt}&\int \big(D^{\alpha}u \big)^2\, dx=-\int D^{\alpha}(uu_x)D^{\alpha}u\, dx.
\end{aligned}
\end{equation} 
Writing $u\partial_x u=\frac{1}{2}\partial_x(u^2)$, the factor involving the nonlinearity in \eqref{eqmainlow1} is controlled by Lemma \ref{leibnizhomog} as follows 
\begin{equation}\label{eqmainlow3}
\begin{aligned}
\big|\int D^{\alpha}(uu_x)D^{\alpha}u\, dx \big|\lesssim \|D^{1+\alpha}(u^2)\|_{L^2}\|D^{\alpha}u\|_{L^2} \lesssim \|u\|_{L^{\infty}}\|D^{1+\alpha}u\|_{L^2}\|D^{\alpha}u\|_{L^2}.
\end{aligned}
\end{equation}
Noticing that our local theory determines $\|D^{1+\alpha}u\|_{L^{\infty}_TL^2}\lesssim \|u\|_{L^{\infty}_TH^s}$, we gather \eqref{eqmainlow2}, \eqref{eqmainlow3} and Gronwall's inequality to find
\begin{equation*}
\|D^{\alpha}u(t)\|_{L^2}\leq \|D^{\alpha}u_0\|_{L^2}+c\|u\|_{L^{\infty}_TH^s}\|u\|_{L^{1}_T W^{1,\infty}},
\end{equation*}
for all $t\in (0,T]$. This concludes the proof of the lemma.
\end{proof}

In virtue of Lemma \ref{claimnegade1} and the preceding estimates, we complete the analysis of $\mathcal{A}_1$ in \eqref{grondiffereequ} with parameters $m=0$ and $1/2\leq \theta <1$. Consequently, we verify the persistent property in the space $Z_{s,r}(\mathbb{R})$ for all $0<r \leq 1$, $s>3/2$ and each $-1/2< \alpha <0$.

Next, we deal with weights satisfying $1<r<3/2+\alpha$, that is, we set $m=1$ and $0<\theta<1/2+\alpha$ in \eqref{grondiffereequ}. As discussed above, we are reduced to control $\mathcal{A}_1$ under the present restrictions. In doing so, we employ the identity
\begin{equation}\label{identcomm}
\begin{aligned}
\left[x^m,\partial_x D^{\alpha}\right]f=\sum_{\substack{1\leq k \leq m \\ k \text{ even }}} c_{\alpha,k} \mathcal{H}D^{\alpha+1-k}(x^{m-k}f)+\sum_{\substack{1\leq k \leq m \\ k \text{ odd }}} c_{\alpha,k}D^{\alpha+1-k}(x^{m-k}f),
\end{aligned}
\end{equation}
valid for $m\geq 1$ integer, some constants $c_{\alpha,k}$ depending on $\alpha$ and $f$ sufficiently regular with enough decay. In particular, 
\begin{equation}
\begin{aligned}
&\left[x,\partial_x D^{\alpha}\right]f=-(1+\alpha) D^{\alpha}f,\\
&\left[x^2,\partial_x D^{\alpha}\right]f=\alpha(1+\alpha)\mathcal{H}D^{\alpha-1}f-2(1+\alpha)D^{\alpha}(xf).
\end{aligned}
\end{equation}
Then, setting $\partial_xD^{\alpha}=-\mathcal{H}D^{1+\alpha}$, we write 
\begin{equation}\label{decomwithweighx}
\begin{aligned}
\langle x \rangle_N^{\theta}x\mathcal{H}D^{1+\alpha} u=&(1+\alpha)\langle x \rangle_N^{\theta}D^{\alpha}u+[\langle x \rangle_N^{\theta},\mathcal{H}]D^{1+\alpha}(xu)+\mathcal{H}\big([\langle x \rangle_N^{\theta},D^{1+\alpha}](xu)\big)\\
&+\mathcal{H}D^{1+\alpha}\big(\langle x \rangle_N^{\theta}xu \big).
\end{aligned}
\end{equation}
As before, the fact that $\mathcal{H}D^{1+\alpha}$ is skew-symmetric leads us to control the first, second and third terms on the r.h.s of the above decomposition. By Proposition \ref{fractfirstcaldcomm} and Lemma \ref{lemmacomm1}, we get 
\begin{equation}\label{eqmainlow3.1}
\begin{aligned}
\|[\langle x \rangle_N^{\theta},\mathcal{H}]D^{1+\alpha}(xu)\|_{L^2}+
\|\mathcal{H}\big([\langle x \rangle_N^{\theta},D^{1+\alpha}](xu)\big)\|_{L^2} \lesssim \|D^{1+\alpha}\langle x \rangle^{\theta}_N\|_{L^{\infty}}\|xu\|_{L^2}.
\end{aligned}
\end{equation}
Since $\theta<1+\alpha$ the above expression is bounded by Proposition \ref{propfracweighapp} and the persistence results in $L^{\infty}([0,T];L^2(|x|^2 \, dx))$. Therefore, it remains to prove $\langle x \rangle^{\theta}D^{\alpha}u\in L^{\infty}([0,T];L^{2}(\mathbb{R}))$. By employing Theorem \ref{TheoSteDer}, we have
\begin{equation}\label{eqlowdisp4}
\begin{aligned}
\|\langle x \rangle^{\theta}D^{\alpha}u\|_{L^2} \lesssim & \|J_{\xi}^{\theta}(|\xi|^{\alpha}\phi\widehat{u})\|_{L^2}+\|J_{\xi}^{\theta}(|\xi|^{\alpha}(1-\phi)\widehat{u})\|_{L^2} \\
\lesssim & \||\xi|^{\alpha}\phi\widehat{u}\|_{L^2}+\|\mathcal{D}^{\theta}_{\xi}(|\xi|^{\alpha}\phi)\widehat{u}\|_{L^2}+\||\xi|^{\alpha}\phi\mathcal{D}^{\theta}_{\xi}\widehat{u}\|_{L^2} \\
&+\||\xi|^{\alpha}(1-\phi)\widehat{u}\|_{L^2}+\|\mathcal{D}_{\xi}^{\theta}(|\xi|^{\alpha}(1-\phi))\widehat{u}\|_{L^2}+\|\xi|^{\alpha}(1-\phi)\mathcal{D}_{\xi}^{\theta}\widehat{u}\|_{L^2}.
\end{aligned}
\end{equation}
We proceed to bound each factor on the r.h.s \eqref{eqlowdisp4}. By \eqref{eqsteinderiweig0} in Proposition \ref{steinderiweighbet2} and \eqref{eqmainlow1}, we find
\begin{equation*}
\begin{aligned}
\||\xi|^{\alpha}\phi\widehat{u}\|_{L^2}+\|\mathcal{D}^{\theta}_{\xi}(|\xi|^{\alpha}\phi)\widehat{u}\|_{L^2} &\lesssim \big(\||\xi|^{\alpha}\phi\|_{L^2}+\||\xi|^{\alpha-\theta}\phi\|_{L^2}\big)\|\widehat{u}\|_{L^{\infty}}+\|\widehat{u}\|_{L^2} \\
&\lesssim \|J_{\xi}\widehat{u}\|_{L^2}\lesssim \|\langle x \rangle u\|_{L^2},
\end{aligned}
\end{equation*}
where we have used $|\xi|^{\alpha}\phi, |\xi|^{\alpha-\theta}\phi \in L^2(\mathbb{R})$ provided that $-1/2<\alpha<0$ and $\theta<1/2+\alpha$. Thus, the above inequality is controlled by the persistence $u \in L^{\infty}([0,T]; L^2(|x|^2\, dx))$. Next, we choose $-\alpha<\beta<1/2$ fixed, so that by applying H\"older inequality, Theorem \ref{TheoSteDer} and Sobolev embedding,
\begin{equation*}
\begin{aligned}
\||\xi|^{\alpha}\phi\mathcal{D}^{\theta}\widehat{u}\|_{L^2} \lesssim \||\xi|^{\alpha}\phi\|_{L^{1/\beta}}\|\mathcal{D}^{\theta}\widehat{u}\|_{L^{2/(1-2\beta)}}&\lesssim \||\xi|^{\alpha}\phi\|_{L^{1/\beta}}\|J^{\theta}\widehat{u}\|_{L^{2/(1-2\beta)}} \\
&\lesssim \|J^{\beta+\theta}_{\xi}\widehat{u}\|\lesssim \|\langle x \rangle u\|_{L^2}. 
\end{aligned}
\end{equation*}
Next, by applying \eqref{prelimneq1},
\begin{equation*}
\begin{aligned}
\||\xi|^{\alpha}(1-\phi)\widehat{u}\|_{L^2}+\|\mathcal{D}_{\xi}^{\theta}(|\xi|^{\alpha}(1-\phi))\widehat{u}&\|_{L^2}+\|\xi|^{\alpha}(1-\phi)\mathcal{D}_{\xi}^{\theta}\widehat{u}\|_{L^2} \\
&\lesssim \|\widehat{u}\|_{L^2}+\|\mathcal{D}_{\xi}^{\theta}\widehat{u}\|_{L^2} \lesssim \|\langle x \rangle^{\theta}u\|_{L^2}.
\end{aligned}
\end{equation*}
This completes the study of \eqref{eqlowdisp4}. Consequently, we conclude the analysis of $\mathcal{A}_1$. Gathering the above results, we deduce the persistent property in the space $Z_{s,r}(\mathbb{R})$ for $0<r <3/2+\alpha$, $s>3/2$ and $-1/2\leq \alpha <0$. The proof of Theorem \ref{Theowellpos} (i) for parameters $-1<\alpha<0$ is complete.


\subsection{Case \texorpdfstring{$0<\alpha<1$}{}}

Here the dispersive part involves the operator $D^{1+\alpha}$, where the power $1+\alpha>1$. Thus, formally, our analysis involves one additional derivative in contrast with the case $\alpha \in [-1,0)$.  Moreover, assuming that $0<\alpha<1$, we have that the weight limit $3/2+\alpha$ ranges between $2$ and $5/2$, so that to cover all the admissible cases, we require to examine additional energy estimates. 

Let us first obtain some general considerations dealing with the estimate for $\mathcal{A}_1$ in \eqref{grondiffereequ}. We write $\partial_x=-\mathcal{H}D$ to perform the following decomposition 
\begin{equation}\label{eqgreatdisp0}
\begin{aligned}
x^m \langle x  \rangle^{\theta}_N \partial_x D^{\alpha}u=&\langle x \rangle^{\theta}_N[x^m,\partial_x D^{\alpha}]u-\langle x \rangle_N^{\theta}\mathcal{H}D^{1+\alpha}(x^m u)\\
=& \langle x \rangle^{\theta}_N[x^m,\partial_x D^{\alpha}]u-[\langle x \rangle_N^{\theta},\mathcal{H}]D^{1+\alpha}(x^m u)-\mathcal{H}([\langle x \rangle_N^{\theta},D^{1+\alpha}](x^m u))\\
&+\partial_x D^{\alpha}( x^m \langle x \rangle_N^{\theta}u)\\
=& \mathpzc{A}_{1,1}+\mathpzc{A}_{1,2}+\mathpzc{A}_{1,3}+\partial_x D^{\alpha}(x^{m}\langle x \rangle^{\theta}_{N} u).
\end{aligned}
\end{equation}
Therefore, to handle with $\mathcal{A}_1$, we only need to bound $\mathpzc{A}_{1,j}$, $j=1,2,3$. Indeed, by identity \eqref{identcomm}, we bound the first term above in the following manner
\begin{equation}\label{eqgreatdisp1}
\begin{aligned}
\|\mathpzc{A}_{1,1}\|_{L^{2}}\lesssim  \sum_{\substack{1\leq k \leq m \\ k \text{ even }}} \|\langle x \rangle_N^{\theta} \mathcal{H}D^{\alpha+1-k}(x^{m-k}u)\|_{L^2}+\sum_{\substack{1\leq k \leq m \\ k \text{ odd }}} \|\langle x \rangle_N^{\theta} D^{\alpha+1-k}(x^{m-k}u)\|_{L^2}.
\end{aligned}
\end{equation}
Now, we employ Proposition \ref{fractfirstcaldcomm} to get
\begin{equation}
\begin{aligned}
\|\mathpzc{A}_{1,2}\|_{L^2}\lesssim \|D^{1+\alpha}\langle x \rangle^{\theta}_{N}\|_{L^{\infty}}\|x^{m}u\|_{L^{2}}.
\end{aligned}
\end{equation}
Recalling Proposition \ref{propfracweighapp}, we have $ \|D^{1+\alpha}\langle x \rangle^{\theta}_{N}\|_{L^{\infty}} \lesssim 1$ with implicit constant independent of $N$. Thus, assuming from previous steps that $x^{m}u \in L^{\infty}([0,T];L^2(\mathbb{R}))$ (with $m=0$ provided by the local theory), we complete the analysis of $\mathpzc{A}_{1,2}$.   

To estimate $\mathpzc{A}_{1,3}$, we write
\begin{equation}\label{eqgreatdisp2} 
\begin{aligned}
\left[\langle x \rangle_N^{\theta},D^{1+\alpha}\right](x^m u) =&[\langle x \rangle_N^{\theta},D^{1+\alpha}](x^m u)+(1+\alpha)\partial_x\langle x \rangle_N^{\theta} \mathcal{H}D^{\alpha}(x^m u)\\
&-(1+\alpha)\partial_x\langle x \rangle_N^{\theta} \mathcal{H}D^{\alpha}(x^m u)\\
=& [\langle x \rangle_N^{\theta},D^{1+\alpha}](x^m u)+(1+\alpha)\partial_x\langle x \rangle_N^{\theta} \mathcal{H}D^{\alpha}(x^m u) \\
&   -(1+\alpha)[\partial_x\langle x \rangle^{\theta}_N,\mathcal{H}]D^{\alpha}(x^m u)-(1+\alpha)\mathcal{H}[\partial_x \langle x \rangle_N^{\theta},D^{\alpha}](x^m u)\\
& -(1+\alpha)\mathcal{H}D^{\alpha}(\partial_x \langle x \rangle_N^{\theta}x^m u).
\end{aligned}
\end{equation} 
We go ahead with the study of each term on the r.h.s of the above identity. Firstly, Lemma \ref{extendenlemma} yields
\begin{equation}\label{eqgreatdisp2.1} 
\begin{aligned}
\|[\langle x \rangle_N^{\theta},D^{1+\alpha}](x^m u)+(1+\alpha)\partial_x\langle x \rangle_N^{\theta} \mathcal{H}D^{\alpha}(x^m u) \|_{L^2} \lesssim \|D^{1+\alpha}\langle x \rangle_N^{\theta}\|_{L^{\infty}}\|x^m u\|_{L^2}.
\end{aligned}
\end{equation}
Proposition \ref{fractfirstcaldcomm} and Lemma \ref{lemmacomm1} show 
\begin{equation}\label{eqgreatdisp2.2}
\begin{aligned}
\|[\partial_x\langle x \rangle^{\theta}_N,\mathcal{H}]D^{\alpha}(x^m u)\|_{L^2}+\|[\partial_x \langle x \rangle_N^{\theta},D^{\alpha}](x^m u)\|_{L^2}\lesssim \|\partial_xD^{\alpha}\langle x \rangle_{N}^{\theta}\|_{L^{\infty}}\|x^m u\|_{L^2}.
\end{aligned}
\end{equation}
Notice that $\partial_xD^{\alpha}=-\mathcal{H}D^{1+\alpha}$, Proposition \ref{propfracweighapp} establishes that $\|\partial_xD^{\alpha}\langle x \rangle_{N}^{\theta}\|_{L^{\infty}} \lesssim 1$ with a constant independent of $N$.  We follow by applying Theorem \ref{TheoSteDer} and the properties \eqref{prelimneq} and \eqref{prelimneq1} to find 
\begin{equation}\label{eqgreatdisp3}
\begin{aligned}
\|D^{\alpha}(\partial_x \langle x \rangle_N^{\theta} x^m u)\|_{L^2} =&\big\|D^{\alpha}\big(\frac{ \partial_x \langle x \rangle^{\theta}_N x^m}{\langle x \rangle^{m+\theta-1}_N}\langle x \rangle^{m+\theta-1}_N u\big)\big\|_{L^2} \\
 \lesssim & \big\|\frac{ \partial_x \langle x \rangle^{\theta}_N x^m}{\langle x \rangle^{m+\theta-1}_N}\langle x \rangle^{m+\theta-1}_N u \big\|_{L^2}+\big\|\mathcal{D}^{\alpha}\big(\frac{ \partial_x \langle x \rangle^{\theta}_N x^m}{\langle x \rangle^{m+\theta-1}_N}\big)\langle x \rangle^{m+\theta-1}_N u\big\|_{L^2} \\
 & +\big\|\frac{ \partial_x \langle x \rangle^{\theta}_N x^m}{\langle x \rangle^{m+\theta-1}_N}\mathcal{D}^{\alpha}(\langle x \rangle^{m+\theta-1}_N u )\big\|_{L^2}\\
\lesssim & \|\langle x \rangle^{m+\theta-1}_N u\|_{L^2}+\|J^{\alpha}\big(\langle x \rangle^{m+\theta-1}_N u \big)\|_{L^2}.
\end{aligned}
\end{equation}
The above inequality is controlled by complex interpolation Lemma \ref{complexinterpo} as follows
\begin{equation}\label{eqgreatdisp4}  
\begin{aligned}
\|J^{\alpha}\big(\langle x \rangle^{m+\theta-1}_N u \big)\|_{L^2} \lesssim \|J^{(m+\theta)\alpha}u\|_{L^2}^{\frac{m+\theta-1}{m+\theta}}\|\langle x \rangle_N^{m+\theta}u\|_{L^2}^{\frac{1}{m+\theta}}.
\end{aligned}
\end{equation}
Since we will impose on each step that $s\geq (m+\theta)\alpha$, and by recurrence $\langle x \rangle^{m+\theta-1} u \in L^{\infty}([0,T];L^2(\mathbb{R}))$, the above inequality is sufficient to deal with \eqref{eqgreatdisp3}. We remark that setting $r=m+\theta$, the above comments precisely impose the condition $s\geq \max\{s_{\alpha}^{+},\alpha r\}$ stated in Theorem \ref{Theowellpos}. Collecting \eqref{eqgreatdisp2}-\eqref{eqgreatdisp4}, we complete the analysis of $\mathpzc{A}_{1,3}$. Thus, matters are reduced to handle $\mathpzc{A}_{1,1}$, namely, we will justify the validity of \eqref{eqgreatdisp1}. We divide this analyze according to the size of $r=m+\theta \in (0,3/2+\alpha)$.

\subsubsection*{\bf Case weight $0<r \leq 1$}  This case is obtained by setting $m=0$ and $r=\theta \in (0,1]$ in the differential equation \eqref{grondiffereequ}. Recalling the decomposition \eqref{eqgreatdisp1}, we observe that $\mathpzc{A}_{1,3}=0$. Thus according to our previous considerations, we conclude the persistence property in $Z_{s,r}(\mathbb{R})$, $0<r\leq 1$ and $s\geq \max\{s_\alpha^{+},\alpha r\}$, where $s_\alpha=3/2-3\alpha/8$.

\subsubsection*{\bf Case weight $1<r <3/2+\alpha$ if $\alpha \in (0,1/2]$, and $1<r \leq 2$ if $\alpha \in (1/2,1)$.}

Letting $r=1+\theta$, this case follows from \eqref{grondiffereequ} with $m=1$, $\theta \in (0,1)$ if $\alpha \in (0,1/2]$, and  including $\theta=1$ if $\alpha>1/2$. Then, in line with \eqref{eqgreatdisp1},  it is enough to establish the following lemma.
\begin{lemma}\label{lemmaClaim1} 
Assume that $0<\theta < 1/2+\alpha$ if $\alpha \in (0,1/2]$, or $\theta \in [0,1]$ if $1/2<\alpha<1$. If $s\geq \max\{s_{\alpha}^{+},\alpha (1+\theta)\}$, then it follows
\begin{equation}\label{eqmain2.0.1}
\sup_{t\in [0,T]}\|\langle x\rangle^{\theta} D^{\alpha}u(t)\|_{L^2} \leq c\big(T,\|u\|_{L^{\infty}_TH^s},\|\langle x \rangle^{1+\theta}u_0\|_{L^2},\|u\|_{L^{1}_TW^{1,\infty}}\big).
\end{equation}
\end{lemma}

\begin{proof}

We begin by showing that $\langle x\rangle^{\theta} D^{\alpha}u_0\in L^2(\mathbb{R})$. Let $\phi\in C^{\infty}_c(\mathbb{R})$ with $0\leq \phi\leq 1$ and $\phi(\xi)=1$ if $|\xi|\leq 1$. We first  deal with the case $0<\theta<\min\{1,\frac{1}{2}+\alpha\}$, whenever $0<\alpha<1$. Combining Plancherel's identity, Theorem \ref{TheoSteDer}, \eqref{prelimneq} and \eqref{prelimneq1}, we deduce
\begin{equation}\label{eqmain2.3}
\begin{aligned}
\|\langle x\rangle^{\theta} D^{\alpha}u_0\|_{L^2}&\leq \|J^{\theta}(|\xi|^{\alpha}\phi\widehat{u_0})\|_{L^2}+\|J^{\theta}\big((|\xi|^{\alpha}(1-\phi))\langle \xi \rangle^{-\alpha}\langle \xi \rangle^{\alpha}\widehat{u_0}\big)\|_{L^2}\\
&\lesssim (1+\|\mathcal{D}^{\theta}\big(|\xi|^{\alpha}(1-\phi)\langle \xi \rangle^{-\alpha}\big)\|_{L^{\infty}}) \|\langle\xi \rangle^{\alpha}\widehat{u_0}\|_{L^2}+\|\mathcal{D}^{\theta}\widehat{u_0}\|_{L^2}\\
&\hspace{1cm}+\|\mathcal{D}^{\theta}(|\xi|^{\alpha}\phi)\|_{L^2}\|\widehat{u_0}\|_{L^{\infty}}+\|\mathcal{D}^{\theta}(\langle\xi \rangle^{\alpha}\widehat{u_0}\big)\|_{L^2} \\
&\lesssim \|J^{\alpha}u_0\|_{L^2}+ \|\langle x \rangle^{\theta}u_0\|_{L^2}+\|\langle x \rangle^{1/2^{+}}u_0\|_{L^2}+\|J^{\theta}_{\xi}(\langle\xi \rangle^{\alpha}\widehat{u_0}\big)\|_{L^2},
\end{aligned}
\end{equation}
where given that $\theta<\alpha+1/2$, Proposition \ref{steinderiweighbet} assures that $\|\mathcal{D}^{\theta}(|\xi|^{\alpha}\phi)\|_{L^2}<\infty$. We employ Lemma \ref{complexinterpo} with $\theta_1=\theta/(1+\theta)$ and $b=\alpha (1+\theta)$ to get
\begin{equation*}
\begin{aligned}
\|J^{\theta}_{\xi}(\langle\xi \rangle^{\alpha}\widehat{u_0}\big)\|_{L^2}\lesssim \|\langle\xi \rangle^{\alpha (1+\theta)}\widehat{u_0}\|_{L^2}^{1/(1+\theta)}\|J^{1+\theta}_{\xi}\widehat{u_0}\|_{L^{2}}^{\theta/(1+\theta)} = \|J^{\alpha(1+\theta)}u_0\|_{L^2}^{1/(1+\theta)}\|\langle x \rangle^{1+\theta}u_0 \|_{L^{2}}^{\theta/(1+\theta)}.
\end{aligned} 
\end{equation*}
The above inequality and \eqref{eqmain2.3} allow us to conclude that  $\langle x\rangle^{\theta} D^{\alpha}u_0\in L^2(\mathbb{R})$ whenever $0<\theta < \min\{1,1/2+\alpha\}$, $\alpha \in (0,1)$. The remaining case $\theta=1$ is obtained arguing as above, employing the local derivative $\partial_{\xi}$ instead of the fractional derivative, and noticing that $|\xi|^{\alpha-1}\phi \in L^2(\mathbb{R})$ for $\alpha>1/2$. We omit these details.

We proceed to deduce \eqref{eqmain2.0.1} for arbitrary time. We apply $D^{\alpha}$ to the equation in \eqref{BDBO}, multiplying the resulting expression by $\langle x \rangle_N^{2\theta}D^{\alpha}u $ and integrating in space show 
\begin{equation}\label{grondiffereequ3}
\begin{aligned}
\frac{1}{2}\frac{d}{dt}\int \big( \langle x \rangle_N^{\theta} D^{\alpha}u \big)^2\, dx- \underbrace{\int \big(\langle x \rangle_N^{\theta} D^{\alpha}\partial_x D^{\alpha}u \big)(\langle x \rangle_N^{\theta} D^{\alpha}u )\, dx}_{\mathpzc{B}_1} +\underbrace{\int \langle x \rangle_N^{\theta} D^{\alpha}(uu_x)(\langle x \rangle_N^{\theta} D^{\alpha}u)\, dx}_{\mathpzc{B}_2}=0. 
\end{aligned}
\end{equation}
By replacing $u$ by $D^{\alpha}u$ in the estimate dealing with $\mathcal{A}_1$ in \eqref{grondiffereequ} for weights in the range $0<r\leq 1$ (i.e., with $m=0$ following the arguments \eqref{eqgreatdisp2}-\eqref{eqgreatdisp4}), the term involving the dispersion is estimated as follows
\begin{equation}\label{eqmain2.1}
|\mathpzc{B}_1| \lesssim \big( \|J^{(1+\theta )\alpha}u\|_{L^2}+\|\langle x \rangle_N^{\theta}D^{\alpha}u\|_{L^2}\big)\|\langle x \rangle_N^{\theta}D^{\alpha}u\|_{L^2}.  
\end{equation}
We move to the analysis of the most difficult factor $\mathpzc{B}_2$. We remark that this part requires extra considerations due to the presence of the fractional derivative $D^{\alpha}$ for the values $0<\alpha<1$ (cf. \cite[Theorem 1.1]{FLinaPioncedGBO} for the case $1< \alpha<2$). Thus, writing $u_x=-\mathcal{H}Du$ yields
\begin{equation*}
\begin{aligned}
\langle x \rangle_N^{\theta}  D^{\alpha}(uu_x)=&\frac{1}{2}[\langle x \rangle_N^{\theta}, D^{\alpha}]\partial_x(u^2)+D^{\alpha}\big(\langle x \rangle_N^{\theta} uu_x\big)\\
=&\frac{1}{2}[\langle x \rangle_N^{\theta}, D^{\alpha}]\partial_x(u^2)-D^{\alpha}\big( u[\langle x \rangle_N^{\theta},\mathcal{H}]Du\big)-D^{\alpha}\big( u\mathcal{H}(\langle x \rangle_N^{\theta} Du)\big)\\
=&\frac{1}{2}[\langle x \rangle_N^{\theta}, D^{\alpha}]\partial_x(u^2)-D^{\alpha}\big( u[\langle x \rangle_N^{\theta},\mathcal{H}]Du\big)-D^{\alpha}\big( u\mathcal{H}([\langle x \rangle_N^{\theta}, D^{1-\alpha}]D^{\alpha}u)\big)\\
&-D^{\alpha}\big( u\mathcal{H}D^{1-\alpha}(\langle x \rangle_N^{\theta}D^{\alpha}u)\big) \\
=:& \mathpzc{B}_{2,1}+\mathpzc{B}_{2,2}+\mathpzc{B}_{2,3}+\mathpzc{B}_{2,4}.
\end{aligned}
\end{equation*}
By Lemmas \ref{leibnizhomog} and \ref{lemmacomm2}, the estimate for $\mathpzc{B}_{2,1}$ follows by previous considerations. Now, since
\begin{equation*}
\mathpzc{B}_{2,2}=-[D^{\alpha},u]\big([\langle x \rangle_N^{\theta},\mathcal{H}]Du\big)-uD^{\alpha}[\langle x \rangle_N^{\theta},\mathcal{H}]Du,
\end{equation*}
an application of Lemma \ref{lemmacomm1} and Proposition \ref{CalderonComGU} allow us to deduce
\begin{equation*}
\begin{aligned}
\|\mathpzc{B}_{2,2}\|_{L^2} &\lesssim \|[D^{\alpha},u]\big([\langle x \rangle_N^{\theta},\mathcal{H}]Du\big)\|_{L^2}+\|uD^{\alpha}[\langle x \rangle_N^{\theta},\mathcal{H}]Du\|_{L^2} \\
&\lesssim \big( \|D^{\alpha}u\|_{L^{\infty}}+\|u\|_{L^{\infty}}\big)\|[\langle x \rangle_N^{\theta},\mathcal{H}]\partial_x\mathcal{H}u\|_{H^1} \\
&\lesssim \big( \|D^{\alpha}u\|_{L^{\infty}}+\|u\|_{L^{\infty}}\big)\big(\|\partial_x \langle x \rangle_N^{\theta}\|_{L^{\infty}}+\|\partial_x^2 \langle x \rangle_N^{\theta}\|_{L^{\infty}}\big)\|u\|_{L^2}.
\end{aligned}
\end{equation*}
The subsequent estimate is obtained by writing
\begin{equation*}
\mathpzc{B}_{2,3}=-[D^{\alpha},u]\big(\mathcal{H}[\langle x \rangle_N^{\theta}, D^{1-\alpha}]D^{\alpha}u\big)-u\mathcal{H}D^{\alpha}[\langle x \rangle_N^{\theta}, D^{1-\alpha}]D^{\alpha}u ,
\end{equation*}
so that Lemmas \ref{lemmacomm1} and \ref{lemmacomm2} yield
\begin{equation}\label{eqmain3}
\begin{aligned}
\|\mathpzc{B}_{2,3}\|_{L^2} &\lesssim\big( \|u\|_{L^{\infty}}+\|D^{\alpha}u\|_{L^{\infty}}\big)\big(\|[\langle x \rangle_N^{\theta}, D^{1-\alpha}]D^{\alpha}u\|_{L^2}+\|D^{\alpha}[\langle x \rangle_N^{\theta}, D^{1-\alpha}]D^{\alpha}u\|_{L^2}\big)\\
&\lesssim\big( \|u\|_{L^{\infty}}+\|D^{\alpha}u\|_{L^{\infty}}\big)\|\partial_x\langle x \rangle_N^{\theta}\|_{L^{\infty}}\big(\|u\|_{L^2}+\|D^{\alpha}u\|_{L^2} \big).
\end{aligned}
\end{equation}
To control the norm $\|D^{\alpha}u\|_{L^{\infty}}$, we employ that
\begin{equation}\label{inecontralpde}
\begin{aligned}
\|D^{\alpha}u\|_{L^{\infty}}\lesssim \|u\|_{L^2}+\|\partial_x u\|_{L^{\infty}}.
\end{aligned}
\end{equation}
Actually, recalling the notation \eqref{projectors}, \eqref{inecontralpde} can be deduced as follows
\begin{equation}\label{inecontralpde1}
\begin{aligned}
\|D^{\alpha}u\|_{L^{\infty}} &\leq \|P_{\leq 0}(D^{\alpha}u)\|_{L^{\infty}}+ \sum_{j\geq 1} \|P_j(D^{\alpha}u)\|_{L^{\infty}} \\
&\lesssim \||\xi|^{\alpha}\psi_{\leq 0}\|_{L^2}\|u\|_{L^2}+\sum_{j\geq 1} 2^{-j(1-\alpha)}\|\widetilde{P}_j(\partial_x u)\|_{L^{\infty}}\\
&\lesssim \|u\|_{L^2}+\|\partial_x u\|_{L^{\infty}},
\end{aligned}
\end{equation}  
where $\widetilde{P}_j$ denotes a modified projector supported in frequency in $|\xi|\sim 2^{j}$, $j\geq 1$. Finally,  we have
\begin{equation}\label{eqmain4}
\begin{aligned}
\mathpzc{B}_{2,4}=-[D^{\alpha},u]D^{1-\alpha}\mathcal{H}\big(\langle x \rangle_N^{\theta} D^{\alpha}u\big)+u\partial_x(\langle x \rangle_N^{\theta} D^{\alpha}u).
\end{aligned}
\end{equation}
The first term on the r.h.s of the above identity is bounded as in \eqref{eqmain3}, i.e., applying Lemma \ref{lemmacomm2}. Going back to the integral defining $\mathpzc{B}_2$ and integrating by parts, we estimate the second term on the r.h.s of \eqref{eqmain4} in the following manner
\begin{equation*}
\begin{aligned}
\Big|\int u\partial_x(\langle x \rangle_N^{\theta} D^{\alpha}u)(\langle x \rangle_N^{\theta} D^{\alpha}u) \, dx \Big|\lesssim \|\partial_x u\|_{L^{\infty}}\|\langle x \rangle_N^{\theta} D^{\alpha}u\|_{L^2}^2.
\end{aligned}
\end{equation*}
Consequently, gathering together the estimates for $\mathpzc{B}_{2,1},\dots, \mathpzc{B}_{2,4}$, we find
\begin{equation}\label{eqmain5}
\begin{aligned}
|\mathpzc{B}_2|\lesssim & \big((\|u\|_{L^{\infty}}+\|\partial_x u\|_{L^{\infty}})\|u\|_{H^{\alpha}}+\|u\|_{H^{\alpha}}^2\big)\|\langle x\rangle_N^{\theta} D^{\alpha}u\|_{L^2}+(1+\|\partial_x u\|_{L^{\infty}})\|\langle x \rangle_N^{\theta} D^{\alpha}u\|_{L^2}^2.
\end{aligned}
\end{equation}
Plugging \eqref{eqmain2.1} and \eqref{eqmain5} in \eqref{grondiffereequ3}, we obtain a differential inequality which after applying Gronwall's lemma and letting $N \to \infty$ implies \eqref{eqmain2.0.1}. 
\end{proof}

\subsubsection*{\bf Case weight $2<r < 3/2+\alpha$, $1/2<\alpha<1$.} 

We consider $m=2$ and $0<\theta<\alpha-1/2$ in \eqref{grondiffereequ}, so that $r=2+\theta$. As showed in \eqref{eqgreatdisp1}, we just need to prove the following lemma.
\begin{lemma}\label{lemmaClaim2} 
Assume $1/2<\alpha<1$. Let $r=2+\theta$ with $0<\theta<\alpha-1/2$, and $s\geq \max\{s_{\alpha}^{+},\alpha r\}$. It holds
\begin{equation}\label{eqlemmaClaim2} 
\begin{aligned}
\sup_{t\in[0,T]}\|\langle x \rangle^{\theta}&D^{\alpha}(x u(t))\|_{L^2} \leq c\big(T,\|u\|_{L^{\infty}_TH^s},\|\langle x \rangle^ru_0\|_{L^2},\|u\|_{L^{1}_TW^{1,\infty}}\big),
\end{aligned}
\end{equation}
and 
\begin{equation}\label{eqlemmaClaim3} 
\sup_{t\in[0,T]}\|\langle x \rangle^{\theta}\mathcal{H}D^{\alpha-1}u(t)\|_{L^2} \lesssim \sup_{t\in[0,T]}\|\langle x \rangle^{(3/2)^{+}}u(t)\|_{L^2}.
\end{equation}
\end{lemma}

\begin{proof}
By recurrent arguments, employing Theorem \ref{TheoSteDer} and Lemma \ref{complexinterpo}, we observe
\begin{equation}\label{eqlemmaClaim3.1} 
\begin{aligned}
\|\langle x \rangle^{\theta}D^{\alpha}(x u_0) \|_{L^2}\lesssim &\|J^{\theta}_{\xi}(|\xi|^{\alpha}\phi \partial_{\xi}\widehat{u_0})\|_{L^2}+\|J^{\theta}_{\xi}(|\xi|^{\alpha}(1-\phi)\partial_{\xi}\widehat{u_0})\|_{L^2}\\
\lesssim & \|J^{1+\theta}\widehat{u_0}\|_{L^2}+\|\mathcal{D}^{\theta}_{\xi}(|\xi|^{\alpha}\phi)\partial_{\xi}\widehat{u_0}\|_{L^2}+\|J^{1+\theta}_{\xi}(\langle \xi\rangle^{\alpha}\widehat{u_0})\|_{L^2} \\
\lesssim & \|\langle x \rangle^{2}u_0\|_{L^2}+\|J^{\alpha(2+\theta)}u_0\|_{L^2}^{\frac{1}{2+\theta}}\|\langle x \rangle^{2+\theta}u_0\|_{L^2}^{\frac{1+\theta}{2+\theta}},
\end{aligned}
\end{equation}
where we have also used Proposition \ref{steinderiweighbet} and Sobolev embedding, $\|\partial_{\xi}\widehat{u_0}\|_{L^{\infty}}\lesssim \|\langle x \rangle^{(3/2)^{+}}u_0\|_{L^2}$. The deduction of \eqref{eqlemmaClaim2} for arbitrary time follows from weighted energy estimates in the equation in \eqref{BDBO}, for the sake of brevity we omit these computations. Finally, \eqref{eqlemmaClaim3} can be obtained by Plancherel's identity, Theorem \ref{TheoSteDer}, \eqref{prelimneq} and using \eqref{eqsteinderiweig0} in Proposition \ref{steinderiweighbet2}. This estimate follows by rather similar consideration as in \eqref{eqlowdisp4}. The proof of the lemma is complete. 

\end{proof}

Consequently, we have completed the analysis of all the admissible cases for the dispersion parameter $\alpha \in [-1,1)\setminus \{0\}$ in the equation in \eqref{BDBO}. The proof of Theorem \ref{Theowellpos} (i) is complete.

\section{Proof of Theorem \ref{Theowellpos} (ii) and (iii)} \label{main2}

This part aims to establish the persistence property in the space $\dot{Z}_{s,r}(\mathbb{R})$. Our arguments are based on weighted energy estimates as in the proof of Theorem \ref{Theowellpos} (i) above. In particular, for the case $\alpha \in (-1,1) \setminus \{0\}$, we will employ the differential equation \eqref{grondiffereequ} to reach our conclusions, that is, we shall bound the term $\mathcal{A}_1$. We divide our considerations according to the dispersive parameter $\alpha$.

\subsection{Proof of Theorem \ref{Theowellpos} (ii): Case \texorpdfstring{$\alpha=-1$}{}}

By applying $\mathcal{H}$ to \eqref{BHeqution}, we get the equation
\begin{equation}\label{burghilbe}
\partial_t \mathcal{H}u-u+\mathcal{H}(uu_x)=0.
\end{equation}
Setting $m=0,1$, we multiply \eqref{BHeqution} by $x^{2m}\langle x \rangle_{N}^{2\theta}u$ and \eqref{burghilbe} by $x^{2m}\langle x \rangle_{N}^{2\theta}\mathcal{H}u$ to obtain, after summing these equations and integrating in space
\begin{equation}\label{eqburghilbe}
\begin{aligned}
\frac{d}{dt} \big(\int (x^{m}\langle x \rangle_{N}^{\theta}u)^2 +\int (x^{m}\langle x \rangle_{N}^{\theta}\mathcal{H}u)^2 \, dx \big)=&-\int (x^{m}\langle x \rangle_{N}^{\theta}u \partial_x u)(x^{m}\langle x \rangle_{N}^{\theta}u)\, dx \\
&-\int (x^{m}\langle x \rangle_{N}^{\theta}\mathcal{H}(u \partial_x u))(x^{m}\langle x \rangle_{N}^{\theta}\mathcal{H}u)\, dx.
\end{aligned}
\end{equation} 
Let us  perform some estimates on the above differential equation. Firstly, to control the right-hand side of \eqref{eqburghilbe} when $m=0$, we write
\begin{equation*}
\begin{aligned}
\langle x  \rangle_{N}^{\theta}\mathcal{H}(uu_x)=&\frac{1}{2}[\langle x \rangle_{N}^{\theta},\mathcal{H}]\partial_x(u^2)-\mathcal{H}(\langle x\rangle_{N}^{\theta} u D \mathcal{H}u)\\
=&\frac{1}{2}[\langle x \rangle_{N}^{\theta},\mathcal{H}]\partial_x(u^2)-\mathcal{H}(u[\langle x\rangle_{N}^{\theta},D^{1/2}]D^{1/2}\mathcal{H}u)\\
&-\mathcal{H}(u D^{1/2}[\langle x \rangle_N^{\theta}, D^{1/2}]\mathcal{H}u)-[\mathcal{H},u] D(\langle x \rangle_N^{\theta}\mathcal{H}u))+u\partial_x(\langle x \rangle_N^{\theta}\mathcal{H}u).
\end{aligned}
\end{equation*}
As we have done before, the last term of the above expression is controlled by going back to the integral involving the nonlinearity and integrating by parts, while Proposition \ref{CalderonComGU} and Lemma \ref{lemmacomm2} yield
\begin{equation*}
\begin{aligned}
\|\langle x  \rangle_{N}^{\theta}\mathcal{H}(uu_x)-u\partial_x(\langle x \rangle_N^{\theta}\mathcal{H}u)\|_{L^2} \lesssim \|\partial_x \langle x \rangle_N^{\theta}\|_{L^{\infty}}\|u\|_{L^{\infty}}\|u\|_{L^2}+\|\partial_x u\|_{L^{\infty}}\|\langle x \rangle_N^{\theta}\mathcal{H}u\|_{L^2}.
\end{aligned}
\end{equation*}
Now, when $m=1$, in virtue of the identity
\begin{equation}\label{identHilcom}
    [H,x]f=0 \text{ if and only if } \int_{\mathbb{R}} f(x)\, dx=0, 
\end{equation}
we may write
\begin{equation*}
\begin{aligned}
x\langle x \rangle_N^{\theta} \mathcal{H}(u \partial_x u)=&\frac{1}{2}\langle x \rangle_N^{\theta}\mathcal{H}(\partial_x(x u^2))-\frac{1}{2}\langle x \rangle_N^{\theta}\mathcal{H}(u^2) \\
=&\frac{1}{2}[\langle x \rangle_N^{\theta},\mathcal{H}]\partial_x(xu^2)-\frac{1}{2}\langle x \rangle_N^{\theta}\mathcal{H}(u^2)+\frac{1}{2}\mathcal{H}(\langle x \rangle_N^{\theta} u^2)+\mathcal{H}(\langle x\rangle_N^{\theta} x u \partial_x u)\\
=&\frac{1}{2}[\langle x \rangle_N^{\theta},\mathcal{H}]\partial_x(xu^2)-\frac{1}{2}\langle x \rangle_N^{\theta}\mathcal{H}(u^2)+\frac{1}{2}\mathcal{H}(\langle x \rangle_N^{\theta} u^2)-\mathcal{H}(\langle x\rangle_N^{\theta} u [x,D] \mathcal{H}u)\\
&-\mathcal{H}(u[\langle x\rangle_N^{\theta},D^{1/2}]D^{1/2}(x\mathcal{H}u))-\mathcal{H}(uD^{1/2}[\langle x\rangle_N^{\theta},D^{1/2}](x\mathcal{H}u))\\
&-[\mathcal{H},u]D(x\langle x \rangle_N^{\theta}\mathcal{H}u)+u\partial_x(x \langle x \rangle_N^{\theta} \mathcal{H}u)\\
=:& \sum_{j=1}^{7} \mathcal{G}_{j}+u\partial_x(x \langle x \rangle_N^{\theta} \mathcal{H}u),
\end{aligned}
\end{equation*}
where we have used the decomposition $\partial_x=-\mathcal{H}D$. We proceed with the study of each component $\mathcal{G}_{j}$. By Proposition \ref{CalderonComGU},
\begin{equation*}
\begin{aligned}
\|\mathcal{G}_{1}\|_{L^2}+\|\mathcal{G}_{7}\|_{L^2} \lesssim \|\partial_x \langle x \rangle_N^{\theta}\|_{L^{\infty}}\|u\|_{L^{\infty}}\|u\|_{L^2}+\|\partial_x u\|_{L^{\infty}}\|x\langle x \rangle_N^{\theta}\mathcal{H}u\|_{L^{2}},
\end{aligned}
\end{equation*}
where the second term on the r.h.s of the above expression is the quantity to be estimated using Gronwall's inequality. Now, if $0<\theta<1/2$, from the identity $[D,x]=\mathcal{H}$ and Proposition \ref{propapcond}, we obtain
\begin{equation*}
\begin{aligned}
\|\mathcal{G}_{2}\|_{L^2}+\|\mathcal{G}_{3}\|_{L^2}+\|\mathcal{G}_{4}\|_{L^2} \lesssim \|u\|_{L^{\infty}}\|\langle x \rangle_N^{\theta} u\|_{L^2},
\end{aligned}
\end{equation*}
which is controlled by the persistence result $u\in L^{\infty}([0,T];L^{2}(|x|^{(1/2)^{-}}\, dx))$ obtained in the proof of Theorem \ref{Theowellpos} (i). Finally, Lemma \ref{lemmacomm2} provides 
\begin{equation*}
\begin{aligned}
\|\mathcal{G}_{5}\|_{L^2}+\|\mathcal{G}_{6}\|_{L^2}\lesssim \|\partial_x \langle x\rangle_N^{\theta}\|_{L^{\infty}}\|u\|_{L^{\infty}}\|\langle x\rangle_N^{\theta} x\mathcal{H}u\|_{L^2}.
\end{aligned}
\end{equation*}
This completes the estimate of the r.h.s of \eqref{eqburghilbe}.

The previous estimates almost complete the proof of the persistence result in the present case, however, since the equation \eqref{BHeqution} does not preserve the property $\widehat{u}_0(0)=0$, we require the following claim
\begin{claim}\label{claimmeanzero}
Let $u_0 \in \dot{Z}_{s,r}(\mathbb{R})$ with $s >3/2$, $r>1/2$ and $u \in C([0,T];Z_{s,(1/2)^{-}}(\mathbb{R}))$ solution of \eqref{BDBO} for $\alpha=-1$ with initial data $u_0$ provided by Theorem \ref{Theowellpos} (i). Then $\widehat{u}(0,t)=0$ for all $t\in [0,T]$.
\end{claim}
The proof of Claim \ref{claimmeanzero} follows from the fact that $u$ solves the integral equation associated to \eqref{BDBO} for $\alpha=-1$, and using complex interpolation Lemma \ref{complexinterpo} to assure that $u^2 \in L^{\infty}([0,T];L^{2}(|x|^{1^{+}}dx))$. We omit these details.

Now, to justify the validity of the differential identity \eqref{eqburghilbe}, we require to prove
\begin{equation}\label{eqburghilbe1}
\mathcal{H}u_0 \in L^2([0,T];L^2(|x|^{2r}\, dx)), 
\end{equation}
whenever $1/2\leq r <3/2$. Once we have proved \eqref{eqburghilbe1}, by setting $m=0$ and $1/2\leq \theta \leq 1$ in \eqref{burghilbe}, the previous estimates lead to the desired persistence result for weights in the range $1/2\leq r \leq 1$. The remaining cases, i.e., $1<r<3/2$ are obtained by letting $m=1$ and $0<\theta<1/2$ in \eqref{eqburghilbe1}, and employing previous consideration. Thus, it remains to prove \eqref{eqburghilbe1}. Indeed, we have:
\begin{itemize}
\item[(i)]  If $r=1/2$, \eqref{eqburghilbe1} does not hold in general, so we address the persistence result in the space
$$HZ_{s,1/2}(\mathbb{R}):=\{f\in Z_{s,1/2}(\mathbb{R}): \|f\|_{Z_{s,r}}+ \||x|^{1/2}\mathcal{H}f\|_{L^2}<\infty\}.$$
\item[(ii)] If $1/2<r\leq 1$, assuming $u_0 \in \dot{Z}_{s,r}(\mathbb{R})$, \eqref{eqburghilbe1} is consequence of the following result:

\begin{lemma}\label{interpo}
Let $1/2<s\leq 1$ and $f\in H^s(\mathbb{R})$ such that $f(0)=0$. Then, $\|\sign(\xi) f\|_{H^s}\lesssim \|f\|_{H^s}$. 
\end{lemma}
The previous lemma is proved in \cite[Lemma 5.1]{OscarCapil}, in fact, it is a consequence of \cite[Proposition 3.2]{Yafaev}.
\item[(iii)] If $1<r<3/2$, \eqref{eqburghilbe1} is obtained from identity \eqref{identHilcom} and Proposition \ref{propapcond}. 
\end{itemize}   

This completes the proof of Theorem \ref{Theowellpos} (ii).


\subsection{Case \texorpdfstring{$-1<\alpha\leq-1/2$}{}}

In contrast with the previous case, we remark that for $\alpha \in (-1,1)\setminus\{0\}$, suitable solutions of the equation in \eqref{BDBO} (for instance, with the hypothesis stated in Theorem \ref{Theowellpos} (ii) is enough) preserve the property $\widehat{u}(0,t)=0$ for all $t\in [0,T]$. We further divide our efforts into parts determined by the range $3/2+\alpha \leq r<5/4+\alpha$.

\subsubsection*{\bf Case weight $3/2+\alpha \leq r \leq 1$} This case is obtained from \eqref{grondiffereequ} with $m=0$ and $3/2+\alpha \leq r=\theta \leq 1$. By the preceding discussions, we are only concerned with the estimate for $\mathcal{A}_1$. We first state the following lemma.
\begin{lemma}\label{lemmazeromean1}
Assume $-1<\alpha \leq -1/2$ and $s>3/2$. If $-\alpha<\theta\leq 1$, then
\begin{equation}\label{eqzeromean1}
\sup_{t\in[0,T]}\|D^{\alpha}u(t)\|_{L^2}\leq c\big(\|\langle x \rangle^{\theta}u_0\|_{L^2},\|u\|_{L^{\infty}_TH^s},\|u\|_{L^{1}_TW^{1,\infty}}\big).
\end{equation}
If $3/2+\alpha \leq \theta \leq 1$, then there exists $0<\epsilon\ll 1$ such that
\begin{equation}\label{eqzeromean2}
\begin{aligned}
\sup_{t\in[0,T]} \|D^{1+\alpha-\theta-\epsilon}u(t)\|_{L^2} \leq c(\|\langle x \rangle^{\theta}u_0\|_{L^2},\|u\|_{L^{\infty}_T H^s},\|u\|_{L^{1}_T W^{1,\infty}}).
\end{aligned}
\end{equation}
Moreover, let $\kappa \geq 2$ integer, such that $\alpha \in (-\kappa/(\kappa+1),-(\kappa-3/2)/\kappa]$, then
\begin{equation}\label{eqzeromean3}
\begin{aligned}
\|\langle x \rangle^{\theta}D^{l(1+\alpha)}P^{\phi}u_0\|_{L^2} \lesssim \|\langle x \rangle^{\theta} u_0\|_{L^2},
\end{aligned}
\end{equation}
for each $l=1,\dots,\kappa$.
\end{lemma}

Before proving Lemma \ref{lemmazeromean1}, let us states its consequences leading to the required persistence result in weighted spaces. Indeed, by \eqref{eqzeromean1}, we can decompose the estimate for $\mathcal{A}_1$ as in \eqref{claimnegade} to employ \eqref{eqmainlow0} and \eqref{eqmainlow0.1}, concluding the subcase $-\alpha<\theta\leq 1$. On the other hand, the inequalities \eqref{eqzeromean2} and \eqref{eqzeromean3} enable us to repeat the same arguments in the Subsection \ref{subsec(II)} to extend the persistence for those cases where $3/2+\alpha \leq \theta \leq -\alpha$. This cover all the considerations in Theorem \ref{Theowellpos} (ii) for $-1<\alpha\leq 1/2$ and $3/2+\alpha\leq r \leq 1$. 

\begin{proof}[Proof of Lemma \ref{lemmazeromean1}]
We first suppose that $-\alpha<\theta \leq 1$. The deduction of \eqref{eqzeromean1} follows from the same reasoning in the proof of Lemma \ref{claimnegade1}. Thus, we shall only verify the validity of  this inequality for the initial data. In the first place, since $\theta>1/2$, by Sobolev embedding, $\widehat{u_0}\in C^{0,\gamma}(\mathbb{R})$ the space of H\"older continuous functions of order $\gamma\in(0,\theta-1/2]$, it holds then
\begin{equation}\label{Holdercond}
|\widehat{u_0}(\xi)-\widehat{u_0}(\eta)|\lesssim \|\widehat{u_0}\|_{H^{\theta}_{\xi}}|\xi-\eta|^{\gamma}, \, \, \text{ for all } \, \, \xi, \eta \in \mathbb{R}.
\end{equation}
Therefore, the zero mean assumption and \eqref{Holdercond} allow us to conclude
\begin{equation}\label{eqzeromean3.1}
\begin{aligned}
\|D^{\alpha}u_0\|_{L^2}=\||\xi|^{\alpha}\widehat{u_0}\|_{L^2}&\lesssim \||\xi|^{\alpha+\gamma}\phi\|_{L^2}\|\widehat{u_0}\|_{H^{\theta}_{\xi}}+\|(1-\phi)\widehat{u_0}\|_{L^2}\lesssim \|\langle x \rangle^{\theta}u_0\|_{L^2},
\end{aligned}
\end{equation}
where we have chosen $\gamma>-\alpha-1/2$. This completes the analysis of \eqref{eqzeromean1}.

Next, we set $3/2+\alpha \leq \theta \leq 1$. Once we have verified \eqref{eqzeromean2} for the initial data, the deduction of this result for arbitrary time follows the same reasoning in the proof of Lemma \ref{claimnegade0.1}.  Indeed, an application of \eqref{Holdercond} reveals
\begin{equation}
\begin{aligned}
\|D^{1+\alpha-\theta-\epsilon}u_0\|_{L^2}=\||\xi|^{1+\alpha-\theta-\epsilon}\widehat{u_0}\|_{L^2}&\lesssim \||\xi|^{1/2+\alpha-\epsilon}\phi\|_{L^2}\|\widehat{u_0}\|_{H^{\theta}_{\xi}}+\|\widehat{u_0}\|_{L^2} \\
&\lesssim \|\langle x \rangle^{\theta}u_0\|_{L^2},
\end{aligned}
\end{equation} 
where we employed that $|\xi|^{1/2+\alpha-\epsilon}\in L^2_{loc}(\mathbb{R})$, whenever $0<\epsilon<1+\alpha$. Finally, by using \eqref{Holdercond} to increase the magnitude of the weight considered, by simple modification to the deduction of \eqref{eqlowdisp2}, we obtain \eqref{eqzeromean3}. This completes the proof of the lemma.
\end{proof}

\subsubsection*{\bf Case weight $1<r<5/2+\alpha$} Here we employ \eqref{grondiffereequ} with $m=1$ and $0<\theta<3/2+\alpha$. Our arguments follow the strategy used in Case (II) \ref{subcaseII}. However, due to the incorporation of the weight $x$ in the differential identity, we need to introduce several modifications. We start by decomposing the estimate for $\mathcal{A}_1$ as in \eqref{decomwithweighx}, that is
\begin{equation}\label{eqzeromean3.2}
\begin{aligned}
\langle x \rangle_N^{\theta}x\mathcal{H}D^{1+\alpha} u=&(1+\alpha)\langle x \rangle_N^{\theta}D^{\alpha}u+[\langle x \rangle_N^{\theta},\mathcal{H}]D^{1+\alpha}(xu)+\mathcal{H}\big([\langle x \rangle_N^{\theta},D^{1+\alpha}](xu)\big)\\
&+\mathcal{H}D^{1+\alpha}\big(\langle x \rangle_N^{\theta}xu \big) \\
=& \widetilde{\mathcal{A}}_{1,1}+\widetilde{\mathcal{A}}_{1,2}+\widetilde{\mathcal{A}}_{1,3}+\mathcal{H}D^{1+\alpha}\big(\langle x \rangle_N^{\theta}xu \big),
\end{aligned}
\end{equation}
being the last term the quantity to be controlled. In this case, Proposition \ref{fractfirstcaldcomm} yields
\begin{equation}\label{eqzeromena3.3}
\begin{aligned}
\widetilde{\mathcal{A}}_{1,2}=\|[\langle x \rangle_N^{\theta},\mathcal{H}]D^{\theta+\epsilon}D^{1+\alpha-\theta-\epsilon}(xu)\|_{L^2} \lesssim \|D^{\theta+\epsilon}\langle x \rangle_N^{\theta}\|_{L^{\infty}}\|D^{1+\alpha-\theta-\epsilon}(xu)\|_{L^2},
\end{aligned}
\end{equation}
where $0<\epsilon \ll 1$ to be chosen later, and the $L^{\infty}(\mathbb{R})$ norm of $D^{\theta+\epsilon}\langle x \rangle_N^{\theta}$ is controlled by Proposition \ref{propfracweighapp}. On the other hand, to estimate $\widetilde{A}_{1,3}$, we write 
\begin{equation*}
\begin{aligned}
\left[\langle x \rangle_N^{\theta},D^{1+\alpha}\right](xu)=&[\langle x \rangle_N^{\theta},D^{1+\alpha}]P^{\phi}(xu)+[\langle x \rangle_N^{\theta},D^{1+\alpha}]P^{1-\phi}(xu).
\end{aligned}
\end{equation*}
We have from Lemma \ref{lemmacomm2} that
\begin{equation*}
\begin{aligned}
\|[\langle x \rangle_N^{\theta},D^{1+\alpha}]P^{1-\phi}(xu)\|_{L^2}=\|[\langle x \rangle_N^{\theta},D^{1+\alpha}]D^{-\alpha}D^{\alpha}P^{1-\phi}(xu)\|_{L^2}\lesssim \|\partial_x\langle x \rangle_N^{\theta}\|_{L^{\infty}}\|xu\|_{L^2},
\end{aligned}
\end{equation*}
the above estimate is justified by the persistence result $u\in L^{\infty}([0,T];L^2(|x|^2)\, dx)$  previously determined. Now, if $0<\theta<1+\alpha$, by Lemma \ref{lemmacomm1} we find
\begin{equation*}
\begin{aligned}
\|[\langle x \rangle_N^{\theta},D^{1+\alpha}]P^{\phi}(xu)\|_{L^2}\lesssim \|D^{1+\alpha}\langle x \rangle_N^{\theta}\|_{L^{\infty}}\|xu\|_{L^{2}}.
\end{aligned}
\end{equation*}
If $1+\alpha\leq \theta< 3/2+\alpha$, we open up the commutator to write
\begin{equation}\label{eqzeromean3.4}
\begin{aligned}
\left[\langle x \rangle_N^{\theta},D^{1+\alpha}\right]P^{\phi}(xu)=&\langle x \rangle_N^{\theta}D^{1+\alpha}P^{\phi}(xu)-D^{1+\alpha}(\langle x \rangle_N^{\theta}P^{\phi}(xu))\\
=& \langle x \rangle_N^{\theta}D^{1+\alpha}P^{\phi}(xu)-D^{1+\alpha}([\langle x \rangle_N^{\theta},P^{\phi}](xu))-D^{1+\alpha}P^{\phi}(\langle x \rangle_N^{\theta}xu).
\end{aligned}
\end{equation}
The last two terms of the above identity can be estimates by Lemma \ref{lemmacomm3} and the properties of the projector $P^{\phi}$ in the following manner
\begin{equation*}
\begin{aligned}
\|D^{1+\alpha}[\langle x \rangle_N^{\theta},P^{\phi}]D^{\theta+\epsilon-(1+\alpha)}D^{1+\alpha-\theta-\epsilon}(xu)\|_{L^2}&+\|D^{1+\alpha}P^{\phi}(\langle x \rangle_N^{\theta}xu)\|_{L^2}\\ &\lesssim \|D^{\theta+\epsilon}\langle x \rangle_N^{\theta}\|_{L^{\infty}}\|D^{1+\alpha-\theta-\epsilon}(xu)\|_{L^2}+\|\langle x \rangle_N^{\theta}xu\|_{L^2}.
\end{aligned}
\end{equation*}
Since the above inequality has the factor to be estimated by the energy estimate for this case, it is sufficient to control the r.h.s of \eqref{eqzeromena3.3} to bound the above expression.

Now, as we did in Case (II) \ref{subcaseII}, the estimate for $ \langle x \rangle_N^{\theta}D^{1+\alpha}P^{\phi}(xu)$ in \eqref{eqzeromean3.4} requires introducing new differential inequalities. To this end, we choose a fixed integer $\kappa\geq 2$ such that $\alpha\in (-\kappa/(\kappa+1),-(\kappa-3/2)/\kappa]$. Thus, the equation in \eqref{BDBO} provides the following identities
\begin{equation}\label{differenzerome}
\begin{aligned}
\frac{1}{2}\frac{d}{dt}\int \big(\langle x \rangle_N^{\theta}D^{l(1+\alpha)}P^{\phi} (xu) \big)^2\, dx &+ \underbrace{\int \langle x \rangle_N^{\theta}D^{l(1+\alpha)}P^{\phi}(x D^{\alpha}\partial_x u )(\langle x \rangle_N^{\theta}D^{l(1+\alpha)}P^{\phi}(xu))\, dx}_{\mathcal{B}_{1,l}}\\
&+\underbrace{\int \langle x \rangle_N^{\theta} D^{l(1+\alpha)}P^{\phi}(x uu_x)(\langle x \rangle_N^{\theta}D^{l(1+\alpha)}P^{\phi}u)\, dx}_{\mathcal{B}_{2,l}}=0,
\end{aligned}
\end{equation}
for each $l=1,\dots, \kappa$ and $1+\alpha \leq \theta<3/2+\alpha$. . We first split the study of $\mathcal{B}_{1,l}$ as follows
\begin{equation}\label{eqzeromean3.5}
\begin{aligned}
\langle x \rangle_N^{\theta}D^{l(1+\alpha)}P^{\phi}(xD^{\alpha}\partial_x u)=-(1+\alpha)\langle x \rangle_N^{\theta}D^{l(1+\alpha)}P^{\phi}D^{\alpha}u+\langle x \rangle_N^{\theta}D^{l(1+\alpha)}P^{\phi}D^{\alpha}\partial_x(xu).
\end{aligned}
\end{equation}
The first term of the above identity can be estimated employing energy estimates as we did for the differential equation \eqref{diferenequalderivlocal}. Thus, we obtain 
\begin{equation}\label{eqzeromean3.6}
\begin{aligned}
\sup_{t\in[0,T]}\|\langle x \rangle^{\theta}D^{l(1+\alpha)}P^{\phi}D^{\alpha}u(t)\|_{L^2}\lesssim c(\|\langle x \rangle^{1+\theta}u_0\|_{L^2},\|u\|_{L^{\infty}_TH^s},\|u\|_{L^1_TW^{1,\infty}}),
\end{aligned}
\end{equation}
for each $l=1,\dots, \kappa$, where we remark that the validity of the above expression for the initial data follows from Theorem \ref{TheoSteDer}, \eqref{Holdercond}, and the embedding $H^{1+\theta}_{\xi}(\mathbb{R})\hookrightarrow C^{0,\gamma}(\mathbb{R})$ with $\gamma < 1/2+\theta$ for $1+\theta\leq 3/2$, and $\gamma=1$ for $1+\theta>3/2$.

Following the same analysis in the study of \eqref{eqlowdisp4.1} and \eqref{eqlowdisp5}, replacing $u$ bu $xu$, we arrive at
\begin{equation}\label{eqzeromean3.6.1}
\begin{aligned}
\|\langle x \rangle_N^{\theta}D^{l(1+\alpha)}P^{\phi}D^{\alpha}\partial_x(xu)\|_{L^2}\lesssim \|D^{1+\alpha-\theta-\epsilon}(xu)\|_{L^2}+\|xu\|_{L^2},
\end{aligned}
\end{equation}
where $0<\epsilon \ll 1$. The study of $\mathcal{B}_{1,l}$ is complete. We divide the estimate for $\mathcal{B}_{2,l}$ as follows
\begin{equation*}
\begin{aligned}
\langle x \rangle_N^{\theta}D^{l(1+\alpha)}P^{\phi}(xu \partial_x u)=\frac{1}{2}\langle x \rangle_N^{\theta}D^{l(1+\alpha)}\partial_xP^{\phi}(xu^2)-\frac{1}{2}\langle x \rangle_N^{\theta}D^{l(1+\alpha)}P^{\phi}(u^2).
\end{aligned}
\end{equation*}
The persistence property in $L^2(|x|^2\, dx)$ and the fact that $u\in L^1([0,T];W^{1,\infty}(\mathbb{R}))$ allow us to apply Theorem \ref{TheoSteDer} and property \eqref{prelimneq} to bound the last term on the r.h.s of the above identity. We focus now on the first term on the r.h.s of the above display. By letting $\partial_x=-\mathcal{H}D$, we write
\begin{equation}\label{eqzeromean3.7}
\begin{aligned}
\langle x \rangle_N^{\theta}D^{l(1+\alpha)}\partial_xP^{\phi}(xu^2)=&-[\langle x \rangle_N^{\theta},\mathcal{H}]D^{l(1+\alpha)+1}P^{\phi}(xu^2)-\mathcal{H}\big([\langle x \rangle_N^{\theta},P^{\phi}]D^{l(1+\alpha)+1}(xu^2)\big)\\
&-\mathcal{H}P^{\phi}\big(\langle x \rangle_N^{\theta}D^{l(1+\alpha)+1}(xu^2)\big).
\end{aligned}
\end{equation}
From Proposition \ref{fractfirstcaldcomm} and Lemma \ref{lemmacomm3}, we deduce 
\begin{equation*}
\begin{aligned}
\|[\langle x \rangle_N^{\theta},\mathcal{H}]D^{l(1+\alpha)+1}P^{\phi}(xu^2)\|_{L^2}&+\|\mathcal{H}\big([\langle x \rangle_N^{\theta},P^{\phi}]D^{l(1+\alpha)+1}(xu^2)\big)\|_{L^2} \\
&\lesssim \big(\|D^{1+l(1+\alpha)}\langle x \rangle_N^{\theta}\|_{L^{\infty}}+\|\partial_x\langle x \rangle_N^{\theta}\|_{L^{\infty}}\big)\|u\|_{L^{\infty}}\|\langle x \rangle u\|_{L^2}.
\end{aligned}
\end{equation*}
By Proposition \ref{propfracweighapp}, we have that $\|D^{1+l(1+\alpha)}\langle x \rangle_N^{\theta}\|_{L^{\infty}}$ is bounded by a uniform constant independent of $N$. Next, we divide the study of the remaining factor on the r.h.s of \eqref{eqzeromean3.7} according to the value of $0<l(1+\alpha)+1 \leq 5/2$, $l=0,,\dots,\kappa$. Thus, if $l(1+\alpha)+1<2$, we write 
\begin{equation}\label{eqzeromean3.8}
\begin{aligned}
\langle x \rangle_N^{\theta}D^{l(1+\alpha)+1}(xu^2)=&D^{l(1+\alpha)+1}(\langle x \rangle_N^{\theta}xu^2)-(l(1+\alpha)+1)\partial_x\langle x \rangle_N^{\theta} \mathcal{H}D^{l(1+\alpha)}(xu^2)\\
&+R(\langle x \rangle_N^{\theta},xu^{2}) \\
=&D^{l(1+\alpha)+1}(\langle x \rangle_N^{\theta}xu^2)-(l(1+\alpha)+1)[\partial_x\langle x \rangle_N^{\theta} ,\mathcal{H}]D^{l(1+\alpha)}(xu^2)\\
&-(l(1+\alpha)+1)\mathcal{H}[\partial_x\langle x \rangle_N^{\theta},D^{l(1+\alpha)}](xu^2)\\
&-(l(1+\alpha)+1)\mathcal{H}D^{l(1+\alpha)}\big(\partial_x\langle x \rangle_N^{\theta}(xu^2)\big)+R(\langle x \rangle_N^{\theta},xu^{2}),
\end{aligned}
\end{equation}
where in virtue of Lemma \ref{extendenlemma},
\begin{equation}\label{eqzeromean3.9}
\|R(\langle x \rangle_N^{\theta},xu^{2})\|_{L^2} \lesssim \|D^{l(1+\alpha)+1}\langle x \rangle_N^{\theta}\|_{L^{\infty}}\|u\|_{L^{\infty}}\|\langle x \rangle u\|_{L^2}.
\end{equation}
Therefore, from \eqref{eqzeromean3.8}, \eqref{eqzeromean3.9}, Proposition \ref{fractfirstcaldcomm}, Lemma \ref{lemmacomm1} and the properties of the projector $P^{\phi}$, we get
\begin{equation*}
\begin{aligned}
\|\mathcal{H}P^{\phi}\big(\langle x \rangle_N^{\theta}D^{l(1+\alpha)+1}(xu^2)\big)\|_{L^2}\lesssim & \|u\|_{L^{\infty}}\|\langle x \rangle_N^{\theta}xu\|_{L^2}\\
&+\big(\|\partial_x\langle x\rangle_N^{\theta}\|_{L^{\infty}}+\|D^{l(1+\alpha)}\partial_x\langle x \rangle_N^{\theta}\|_{L^{\infty}} + \|D^{l(1+\alpha)+1}\langle x \rangle_N^{\theta}\|_{L^{\infty}}\big)\\
&\qquad \qquad \times\|u\|_{L^{\infty}}\|\langle x \rangle u\|_{L^2},
\end{aligned}
\end{equation*}
being $\|\langle x \rangle_N^{\theta}xu\|_{L^2}$ the term to be estimated after combining the differential identity \eqref{grondiffereequ} with \eqref{differenzerome} for each $l=1,\dots, \kappa$. We recall that Proposition \ref{propfracweighapp} implies that the $L^{\infty}$-norm of $D^{l(1+\alpha)}\partial_x\langle x \rangle_N^{\theta}$ and $D^{l(1+\alpha)+1}\langle x \rangle_N^{\theta}$ are bounded by a constant independent of $N$. Now, if $2\leq l(1+\alpha)+1<5/2$, by following a similar decomposition to \eqref{eqzeromean3.8} and applying \eqref{eqextenden2} in Lemma \ref{extendenlemma}, it is not difficult to deduce
\begin{equation*}
\begin{aligned}
\|\mathcal{H}P^{\phi}&\big(\langle x \rangle_N^{\theta}D^{l(1+\alpha)+1}(xu^2)\big)\|_{L^2} \\\lesssim & \|u\|_{L^{\infty}}\|\langle x \rangle_N^{\theta}xu\|_{L^2}+\big(\|\partial_x\langle x\rangle_N^{\theta}\|_{L^{\infty}}+\|\partial_x^2\langle x\rangle_N^{\theta}\|_{L^{\infty}}+\|D^{l(1+\alpha)-1}\partial_x^2\langle x \rangle_N^{\theta}\|_{L^{\infty}}\\
&\hspace{3cm}+\|D^{l(1+\alpha)}\partial_x\langle x \rangle_N^{\theta}\|_{L^{\infty}} + \|D^{l(1+\alpha)+1}\langle x \rangle_N^{\theta}\|_{L^{\infty}}\big)\times\|u\|_{L^{\infty}}\|\langle x \rangle u\|_{L^2},
\end{aligned}
\end{equation*}
which is controlled by Proposition \ref{propfracweighapp}. This completes the analysis of the nonlinear parts $\mathcal{B}_{2,l}$, $l=1,\dots, \kappa$ in \eqref{differenzerome}. Gathering the previous estimates, we get a closed differential inequality similar to that of \eqref{resuldiff} that yields the desired conclusion. Consequently, collecting the missing estimates between \eqref{eqzeromean3.2}-\eqref{eqzeromean3.6.1}, we infer that the deduction of the present case will be completed as soon as we have proved the following lemma.
\begin{lemma}\label{lemmazeromean2}
Consider $s>3/2$. Let $0<\theta<1/2$,
\begin{equation}\label{eqzeromean4}
\begin{aligned}
\sup_{t\in[0,T]}\|\langle x \rangle^{\theta} D^{\alpha}u(t)\|_{L^2} \leq c(\|\langle x\rangle^{1+\theta}u_0\|,\|u\|_{L^{\infty}_TH^s},\|u\|_{L^1_T W^{1,\infty}}).
\end{aligned}
\end{equation}
If $1/2 \leq \theta<3/2+\alpha$, then it follows
\begin{equation}\label{eqzeromean4.0}
\begin{aligned}
\sup_{t\in[0,T]}\|\langle x \rangle^{\theta} D^{\alpha}u(t)\|_{L^2} \lesssim \sup_{t\in[0,T]}\|\langle x \rangle^{(3/2)^{-}} D^{\alpha}u(t)\|_{L^2}.
\end{aligned}
\end{equation}
Moreover, assume $1+\alpha \leq \theta < 3/2+\alpha$. Then, there exists $0<\epsilon \ll 1$ such that
\begin{equation}\label{eqzeromean4.1}
\sup_{t \in [0,T]}\|D^{1+\alpha-\theta-\epsilon}(xu)(t)\|_{L^2}\leq c(\|\langle x\rangle^{1+\theta}u_0\|,\|u\|_{L^{\infty}_TH^s},\|u\|_{L^1_T W^{1,\infty}}).
\end{equation}
Additionally, let $\kappa\geq 2$ such that $\alpha\in (-\kappa/(\kappa+1),-(\kappa-3/2)/\kappa]$. Then it follows,
\begin{equation}\label{eqzeromean4.2}
\|\langle x \rangle^{\theta}D^{l(1+\alpha)}P^{\phi}(xu_0)\|_{L^2} \lesssim \|\langle x \rangle^{1+\theta} u_0\|_{L^2},
\end{equation}
for each $l=1,\dots, \kappa$.
\end{lemma}
There is a subtle difference between \eqref{eqzeromean4} and \eqref{eqzeromean4.0}. The former inequality can be obtained by energy estimates in the equation in \eqref{BDBO}, while the latter can be deduced directly from properties of fractional weights and derivatives, and the persistence result $u\in L^{\infty}([0,T];L^2(|x|^{3^{-}}))$. This imposes that for a given $\alpha \in (-1,-1/2]$, we first  establish $u\in L^{\infty}([0,T];L^2(|x|^{2r}))$ for weights $r=1+\theta$, $0<\theta<1/2$. Once we have proved this results, we employ \eqref{eqzeromean4.0} to complete the cases $1/2\leq \theta <3/2+\alpha$.

\begin{proof}[Proof of Lemma \ref{lemmazeromean2}]
We first verify \eqref{eqzeromean4} for the initial data. By recurrent arguments using Theorem \ref{TheoSteDer}, we find
\begin{equation}\label{eqzeromean5}
\begin{aligned}
\|\langle x \rangle_N^{\theta}D^{\alpha}u_0\|_{L^2}\lesssim & \||\xi|^{\alpha}\widehat{u_0}\|_{L^2}+\|\mathcal{D}^{\theta}_{\xi}(|\xi|^{\alpha}\widehat{u_0})\phi\|_{L^2}+\||\xi|^{\alpha}\widehat{u_0}\mathcal{D}_{\xi}^{\theta}\phi\|_{L^2} \\
&+\|\mathcal{D}_{\xi}^{\theta}(|\xi|^{\alpha}(1-\phi))\widehat{u_0}\|_{L^2}+\||\xi|^{\alpha}(1-\phi)\mathcal{D}^{\theta}_{\xi}\widehat{u_0}\|_{L^2}.
\end{aligned}
\end{equation}
To estimate the above inequality, following the ideas in \eqref{eqzeromean3.1}, we are left to prove $\mathcal{D}^{\theta}_{\xi}(|\xi|^{\alpha}\widehat{u_0})\phi \in L^2(\mathbb{R})$. We choose $\theta-\alpha-1/2<\gamma<\min\{1/2+\theta,1\}$, so that by applying \eqref{eqsteinderiweig1} in Proposition \ref{steinderiweighbet2}, and the embedding $H^{1/2+\gamma}(\mathbb{R})\hookrightarrow C^{0,\gamma}(\mathbb{R})$, we deduce
\begin{equation}\label{eqzeromean6}
\begin{aligned}
\|\mathcal{D}^{\theta}_{\xi}(|\xi|^{\alpha}\widehat{u_0})\phi\|_{L^2} \lesssim \|\widehat{u_0}\|_{C^{0,\gamma}}\lesssim \|\widehat{u_0}\|_{H^{1/2+\gamma}_{\xi}} \sim \|\langle x \rangle^{1/2+\gamma}u_0\|_{L^2},
\end{aligned}
\end{equation}
where we implicitly used that $|\xi|^{\gamma+\alpha-\theta}\in L^2_{loc}(\mathbb{R})$. In particular, if $\theta-\alpha-1/2<\gamma \leq 1/2$, \eqref{eqzeromean6} is controlled by $\langle x \rangle u_0\in L^2(\mathbb{R})$. Then, from the persistence result $u\in L^{2}(|x|^2 \, dx)$, and the assumption $\widehat{u}(0,t)=0$ for $t\in[0,T]$, we can replace $u_0$ by $u$ in \eqref{eqzeromean5} and \eqref{eqzeromean6}, obtaining \eqref{eqzeromean4} for all $0<\theta<1+\alpha$. Next, assuming the persistence result $D^{\alpha}u\in L^{\infty}([0,T];L^2(|x|^{3^{-}}\, dx))$ and replacing $u_0$ by $u$ in the previous considerations, we can take $\gamma$ arbitrarily close to $1^{-}$ to deduce \eqref{eqzeromean4.0}.

Hence, to complete the considerations for \eqref{eqzeromean4}, we shall assume that $1+\alpha \leq \theta <1/2$. This case is deduced by weighted energy estimates in the equation \eqref{grondiffereequ3} with the current restriction on $\theta$ and $\alpha$. In fact, the nonlinear part can be estimated by using the fractional derivative $\mathcal{D}_{\xi}^{\theta}$ in the frequency domain, together with its properties and the persistence result $u\in L^{\infty}([0,T];L^2(|x|^2\, dx))$. We omit these computations. The dispersive term in \eqref{grondiffereequ3} is estimated by replacing $u$ by $D^{\alpha}u$ in the arguments of Case (II) \ref{subcaseII}. Accordingly, recalling the validity of \eqref{eqzeromean3.6}, we just need to verify:
\begin{claim}\label{claimlowdisp1}
Let $0<\theta<3/2+\alpha$ and $s>3/2$. Then there exists some $0<\epsilon\ll 1$ such that
\begin{equation}\label{eqclaimlowdisp1}
\begin{aligned}
\sup_{t\in [0,T]}\|D^{1+2\alpha-\theta-\epsilon}u(t)\|_{L^2} \leq c(\|\langle x\rangle^{1+\theta}u_0\|,\|u\|_{L^{\infty}_TH^s},\|u\|_{L^1_T W^{1,\infty}}).
\end{aligned}
\end{equation}
\end{claim}
The proof of the previous claim follows by similar ideas as in the deduction of Lemma \ref{claimnegade0.1}. This completes the analysis for \eqref{eqzeromean4}. The deduction for \eqref{eqzeromean4.1} follows from weighted energy estimates and recurrent argument by now. Finally, replacing the role of $u_0$ by $xu_0$ in \eqref{eqlowdisp2}, we deduce \eqref{eqzeromean4.2}.

\end{proof} 


\subsection{Case \texorpdfstring{$-1/2<\alpha<0$}{}}

We first consider weights $3/2+\alpha \leq r \leq 2$, and then we turn to decay $2 \leq r <5/2+\alpha$. The former case is obtained from the differential equation \eqref{grondiffereequ} with $m=1$ and $1/2+\alpha \leq \theta \leq 1$.  Thus, since the estimate for $\mathcal{A}_1$ involves the commutator $[x,\mathcal{H}D^{1+\alpha}]$ as in \eqref{decomwithweighx}, we require to deduce the following lemma:
\begin{lemma}\label{lemmahomlowd}
Assume that $1/2+\alpha<\theta  \leq 1$ and $s>3/2$. Then it follows,
\begin{equation}
\sup_{t\in [0,T]}\|\langle x\rangle^{\theta} D^{\alpha}u(t)\|_{L^2} \leq c\big(T,\|u\|_{L^{\infty}_TH^s},\|\langle x \rangle^{1+\theta}u_0\|_{L^2},\|u\|_{L^{1}_TW^{1,\infty}}\big).
\end{equation}
\end{lemma}
It is not difficult to adapt previous consideration to prove Lemma \ref{lemmahomlowd}, for the sake of brevity we omit these details. Now, we can follow the same decomposition and arguments dealing with \eqref{decomwithweighx} to extend the persistence result for the cases $1/2+\alpha \leq \theta<1+\alpha$, i.e., it follows that $u\in L^{\infty}([0,T];L^2(|x|^{2r}\, dx))$ as long as $3/2+\alpha\leq r=1+\theta < 2+\alpha$. Consequently, we shall assume that $1+\alpha \leq \theta \leq 1$. This case is deduce by  applying the same decomposition \eqref{decomwithweighx}, \eqref{Comwellprel1} in Proposition \ref{CalderonComGU}, and Lemma \ref{lemmacomm2}, we omit this analysis.

Finally, we deal with the case $2<r<5/2+\alpha$, that is, we set $m=2$ and $0<\theta<1/2+\alpha$ in \eqref{grondiffereequ}. In virtue of the identity 
\begin{equation*}
\begin{aligned}
\left[x^2,\mathcal{H}D^{1+\alpha}\right]f=-(1+\alpha)\alpha \mathcal{H}D^{\alpha-1}f+2(1+\alpha)D^{\alpha}(xf),
\end{aligned}
\end{equation*}
we decompose the estimate for the factor $\mathcal{A}_1$ in \eqref{grondiffereequ} as follows 
\begin{equation}\label{eqlemmahomlowd}
\begin{aligned}
\langle x \rangle_{N}^{\theta}x^2 \mathcal{H}D^{1+\alpha}u=&\langle x\rangle_N^{\theta}[x^{2},\mathcal{H}D^{1+\alpha}]u+\langle x \rangle_N^{\theta}\mathcal{H}D^{1+\alpha}(x^2 u) \\
=& -(1+\alpha)\alpha \langle x \rangle_{N}^{\theta} \mathcal{H}D^{\alpha-1}u+2(1+\alpha)\langle x \rangle_{N}^{\theta}D^{\alpha}(xu)\\
&+[\langle x\rangle^{\theta}_N,\mathcal{H}]D^{1+\alpha}(x^2 u)+\mathcal{H}([\langle x\rangle^{\theta}_N,D^{1+\alpha}](x^2u))+\mathcal{H}D^{1+\alpha}(\langle x \rangle_N^{\theta}x^2 u).
\end{aligned}
\end{equation}
Going back to the integral defining $\mathcal{A}_1$, we find that the last term on the r.h.s of the above identity yields null contribution to the estimate. The second term is estimated in Lemma \ref{lemmahomlowd}. Whereas, by employing the persistence result $u\in L^{\infty}([0,T];L^2(|x|^{4}\, dx))$, and following similar considerations to \eqref{eqmainlow3.1}, we control the third and fourth term on the r.h.s of \eqref{eqlemmahomlowd}. On the other hand, by Plancherel's identity, Theorem \ref{TheoSteDer}, \eqref{prelimneq0}, \eqref{prelimneq} and \eqref{eqsteinderiweig1} in Proposition \ref{steinderiweighbet2} with $\gamma=1$, it is seen 
\begin{equation}\label{eqlemmahomlowd1}
\begin{aligned}
\|\langle x \rangle^{\theta}D^{\alpha-1}u\|_{L^2} 
\lesssim & \|\langle x \rangle^{(3/2)^{+}}u\|_{L^2},
\end{aligned}
\end{equation}
which completes the estimate for the remaining factor on the r.h.s of \eqref{eqlemmahomlowd}. The proof of Theorem \ref{Theowellpos} (ii) for the case $-1/2<\alpha<0$ is complete.


\subsection{Case \texorpdfstring{$0<\alpha<1$}{}}

We divide our arguments according to the size of the decay variable $r\geq 3/2+\alpha$. 

\subsubsection*{\bf Case weight $3/2+\alpha \leq r<5/2+\alpha$ for $0<\alpha \leq 1/2$, and $3/2+\alpha \leq r \leq 3$ for $1/2<\alpha \leq 1$}

The same arguments presented in the deduction of the persistence property in $Z_{s,r}(\mathbb{R})$ with $1<r \leq 2$ for  $1/2<\alpha<1$ can be extended to the space $\dot{Z}_{s,r}(\mathbb{R})$ assuming that $3/2+\alpha\leq r \leq 2$ for $0<\alpha \leq 1/2$. Thus, according to Lemma \ref{lemmaClaim1}, we only need to prove
\begin{equation}
\|\langle x \rangle^{\theta} D^{\alpha}u_0\|_{L^2} \lesssim \|u_0\|_{Z_{s,r}}, 
\end{equation}
where $1/2+\alpha \leq \theta \leq 1$, $\alpha\in(0,1/2]$. In fact, the hypothesis $\widehat{u}(0)=0$, the mean value inequality and Sobolev embedding yield
\begin{equation}\label{eqmain12}
|\widehat{u_0}(\xi)|\lesssim \|\partial_{\xi}\widehat{u_0}\|_{L^{\infty}}|\xi|\lesssim \|\langle x \rangle^{(3/2)^{+}}u_0\|_{L^2}|\xi|.
\end{equation}
Therefore, by employing \eqref{eqmain12} to increase the magnitude of the weight in the $\xi$-variable by one, we can repeat the argument in the deduction of \eqref{eqmain2.3} to obtain \eqref{eqmain12}. This provides the persistence property in  $\dot{Z}_{s,r}(\mathbb{R})$ with $3/2+\alpha\leq r \leq 2$, for each $0<\alpha \leq 1/2$.

Next, we treat the cases $2<r < 5/2+\alpha$ with $\alpha\in (0,1/2]$, and $3/2+\alpha \leq r\leq 3$ for $1/2< \alpha <1$. These conclusions are obtained by the same energy estimate employed in the analysis for the space $Z_{s,r}(\mathbb{R})$ with $2<r<3/2+\alpha$, $\alpha \in (1/2,1]$. Thus, from the arguments in the deduction of Lemma \ref{lemmaClaim2}, we notice that it is enough to show
\begin{equation}\label{eqmain9}
\|\langle x \rangle^{\theta}D^{\alpha}(xu_0)\|_{L^2}\lesssim \|u_0\|_{Z_{s,r}},
\end{equation}
and  
\begin{equation}\label{eqmain9.1}
\sup_{t\in [0,T]}\|\langle x \rangle^{\theta}\mathcal{H}D^{\alpha-1}u(t)\|_{L^2} \lesssim \sup_{t\in[0,T]}\|\langle x \rangle^{(3/2)^{+}}u(t)\|_{L^{2}},
\end{equation}
for all $0<\theta<1/2+\alpha$ with $0<\alpha\leq 1/2$, and each $\alpha-1/2\leq \theta \leq 1$ for the cases  $1/2<\alpha<1$. Anyway, as long as $\theta<1$, the estimate in \eqref{eqlemmaClaim3.1} provides the required bound for the first term on l.h.s of \eqref{eqmain9}, and the case $\theta=1$ is obtained by simple modifications. Now, bearing in mind \eqref{eqmain12}, we can modify the arguments in \eqref{eqlemmahomlowd1} to obtain \eqref{eqmain9.1}. 

Consequently, we have completed the deduction of the persistence property in the space $\dot{Z}_{s,r}(\mathbb{R})$ for $3/2+\alpha \leq r<5/2+\alpha$ if $0< \alpha\leq 1/2$. Additionally, we have proved persistence in $\dot{Z}_{s,r}(\mathbb{R})$ for $3/2+\alpha \leq r \leq 3$, if $1/2< \alpha<1$.

\subsubsection*{\bf Case weight $3<r<5/2+\alpha$ for $1/2<\alpha<1$}

Here we prove the persistence property in the space $\dot{Z}_{s,r}(\mathbb{R})$ with $3<r<5/2+\alpha$ and $1/2<\alpha<1$. Writing $r=3+\theta$ where $0<\theta<\alpha-1/2$, our considerations are obtained from \eqref{grondiffereequ} with $m=3$ and $0<\theta<\alpha-1/2$. Thus, by the estimate \eqref{eqgreatdisp1}, we are reduced to prove. 

\begin{lemma}
Let $r=3+\theta$ with $0<\theta<\alpha-1/2$, and $s\geq \max\{s_{\alpha}^{+},\alpha r\}$. Then it holds that
\begin{equation}\label{eqmain10}
\begin{aligned}
\sup_{t\in[0,T]}\|\langle x \rangle^{\theta}&D^{\alpha}(x^2 u(t))\|_{L^2} \leq c\big(T,\|u\|_{L^{\infty}_TH^s},\|\langle x \rangle^ru_0\|_{L^2},\|u\|_{L^{1}_TW^{1,\infty}}\big),
\end{aligned}
\end{equation}
and 
\begin{equation}\label{eqmain11} 
\begin{aligned}
\sup_{t\in[0,T]}\big(\|\langle x \rangle^{\theta}\mathcal{H}D^{\alpha-1}(xu(t))\|_{L^2}+\|\langle x \rangle^{\theta}D^{\alpha-2}u(t)\|_{L^2} \big)\lesssim \sup_{t\in[0,T]}\|\langle x \rangle^{2}u(t)\|_{L^2}.
\end{aligned}
\end{equation}
\end{lemma}
At this point, we have developed enough arguments to deduce the above lemma. We omit its deduction for the sake of brevity.


\section{Proof of Theorem \ref{Theotwotimcondi}} \label{main3}

We will prove Theorem \ref{Theotwotimcondi} only for two groups of dispersion, $-1\leq \alpha \leq -1/2$ and $0<\alpha \leq 1/2$. The remaining cases, $-1/2<\alpha<0$ and $1/2<\alpha<1$ are treated by simple modifications to our considerations below. We have chosen the previous restrictions to show our technique for negative and positive values of $\alpha$, as well as to exemplify those cases where the maximum decay rate $3/2+\alpha$ is at most one, and where it is at most of order two.

Let $s>0$ provided by the hypothesis of Theorem \ref{Theotwotimcondi} according to the values of $r$ and $\alpha \in [-1,1) \setminus\{0\}$. Without loss of generality, we may assume that $t_1=0$, that is,
\begin{equation}\label{twoeq1}
u_0, u(t_2)\in Z_{s,3/2+\alpha}(\mathbb{R}),
\end{equation}
with the extra assumption that $u_0 \in Z_{s,(1/2)^{+}}(\mathbb{R})$ for $\alpha=-1$. Our conclusions are obtained by estimating the integral equation associated to \eqref{BDBO} in the frequency domain, namely, 
\begin{equation}\label{twoeq2}
\widehat{u}(\xi,t)=e^{i\xi|\xi|^{\alpha}t}\widehat{u_0}(\xi)-\int_0^t e^{i\xi|\xi|^{\alpha}(t-\tau)}\widehat{uu_x}(\xi,\tau)\, d\tau.
\end{equation}
We proceed to study \eqref{twoeq2} according to the cases $-1\leq \alpha\leq 1/2$ and $0<\alpha \leq 1/2$.


\subsection{Case \texorpdfstring{$-1 \leq \alpha \leq -1/2$}{}}

We first decompose the homogeneous term in \eqref{twoeq2} as follows
\begin{equation}\label{twoeq3}
\begin{aligned}
e^{i\xi|\xi|^{\alpha}t}\widehat{u_0}(\xi)=&e^{i\xi|\xi|^{\alpha}t}\widehat{u_0}(0)\phi+e^{i\xi|\xi|^{\alpha}t}(\widehat{u_0}(\xi)-\widehat{u_0}(0))\phi+e^{i\xi|\xi|^{\alpha}t}\widehat{u_0}(\xi)(1-\phi).
\end{aligned}
\end{equation}
Through the above decomposition we claim.
\begin{claim}\label{twotimclaim1}
Let $s>3/2$, consider $3/10<r<1/2$ for $\alpha=-1$, and  $\max\{1/2,9/10+3\alpha/5\}<r<3/2+\alpha$ for $-1<\alpha \leq -1/2$. If $u \in C([0,T],Z_{s,r}(\mathbb{R}))$ with $u_0\in Z_{s,(3/2)^{+}}(\mathbb{R})$ for $\alpha=-1$, and $u_0 \in Z_{s,3/2+\alpha}(\mathbb{R})$ when $-1<\alpha\leq -1/2$, then it holds
\begin{equation}\label{twoeq3.1}
\begin{aligned}
\langle \xi \rangle^{-2}\Big(e^{i\xi|\xi|^{\alpha}t}(\widehat{u_0}(\xi)-\widehat{u_0}(0))\phi&+e^{i\xi|\xi|^{\alpha}t}\widehat{u_0}(\xi)(1-\phi)\\
&-\int_0^t e^{i\xi|\xi|^{\alpha}(t-\tau)}\widehat{uu_x}(\xi,\tau)\, d\tau\Big) \in L^{\infty}([0,T];H^{3/2+\alpha}_{\xi}(\mathbb{R})).
\end{aligned}
\end{equation}
\end{claim}

Assuming for the moment the validity of Claim \ref{twotimclaim1}, from \eqref{twoeq1} and \eqref{twoeq2}, we find
\begin{equation}\label{twoeq3.2}
\begin{aligned}
\langle \xi \rangle^{-2}\widehat{u}(\xi,t)\in H^{3/2+\alpha}_{\xi}(\mathbb{R})\, \text{ if and only if } \, \langle \xi \rangle^{-2}(e^{i\xi|\xi|^{\alpha}t}\phi)\widehat{u_0}(0) \in H^{3/2+\alpha}_{\xi}(\mathbb{R}).
\end{aligned}
\end{equation}
Now, if $0<3/2+\alpha<1$, by writing
\begin{equation}\label{twoeq3.3}
\begin{aligned}
D^{3/2+\alpha}\big(\langle \xi \rangle^{-2}\widehat{u}(\xi,t)\big)=[D^{3/2+\alpha},\langle \xi \rangle^{-2}]\widehat{u}(\xi,t)+\langle \xi \rangle^{-2}(D^{3/2+\alpha}\widehat{u})(\xi,t),
\end{aligned}
\end{equation}
we employ Lemma \ref{lemmacomm1} and the ideas in \eqref{inecontralpde1} to obtain 
\begin{equation*}
\begin{aligned}
\|[D^{3/2+\alpha},\langle \xi \rangle^{-2}]\widehat{u}\|_{L^2}\lesssim \|D^{3/2+\alpha}(\langle \xi \rangle^{-2})\|_{L^{\infty}}\|u\|_{L^2} \lesssim \big(\|\langle \xi \rangle^{-2}\|_{L^2}+\|\partial_{\xi}\langle \xi \rangle^{-2}\|_{L^{\infty}}\big)\|u\|_{L^2}.
\end{aligned}
\end{equation*}
Notice that if $3/2+\alpha=1$ (i.e., $\alpha=-1/2$), we can argue as above, replacing the fractional derivative by the local operator $\partial_{\xi}$. We omit these computations. Hence, we find
\begin{equation}\label{twoeq4}
\begin{aligned}
\langle \xi \rangle^{-2}(D^{3/2+\alpha}\widehat{u})(\xi,t)\in L^{2}(\mathbb{R})\, \text{ if and only if } \, D^{3/2+\alpha}(\langle \xi \rangle^{-2} e^{i\xi|\xi|^{\alpha}t}\phi)\widehat{u_0}(0) \in L^{2}(\mathbb{R}).
\end{aligned}
\end{equation}
When $\alpha=-1/2$, we replace $D$ by $\partial_{\xi}$ in \eqref{twoeq4}. We will prove below that for all $t>0$, 
\begin{equation}\label{twoeq4.0}
D^{3/2+\alpha}(\langle \xi \rangle^{-2} e^{i\xi|\xi|^{\alpha}t}\phi), \, \, \partial_{\xi}(\langle \xi \rangle^{-2}e^{i\xi|\xi|^{-1/2}t}\phi) \notin L^{2}(\mathbb{R}),
\end{equation}
whenever $-1\leq \alpha<1/2$. Consequently, since \eqref{twoeq1} implies that \eqref{twoeq4} is true at $t=t_2>0$, it must follow that $\widehat{u_0}(0)=0$ as required. This in turn determines the conclusion of Theorem \ref{Theotwotimcondi} (ii) for $-1<\alpha \leq 1/2$, and in virtue of Claim \ref{claimmeanzero}, we prove Theorem \ref{Theotwotimcondi} (i) case $\alpha=-1$. 

Now, we prove \eqref{twoeq4.0}. The case $\alpha=-1/2$ is deduced noticing that $\partial_{\xi}(e^{i\xi|\xi|^{-1/2}t}\phi)\notin L^2(\mathbb{R})$ as $|x|^{-1/2}$ is not square-integrable at the origin. If $\alpha \in (-1,-1/2)$, we let $A(t)=\{z:0<z<\frac{1}{2}\min\{1,t^{-\frac{1}{1+\alpha}}\}\}$ for fixed $t>0$. Employing the definition (ii) in Theorem \ref{TheoSteDer}, we deduce
\begin{equation*}
\begin{aligned}
\big(\mathcal{D}_{\xi}^{3/2+\alpha}(\langle \xi \rangle^{-2}&e^{i\xi|\xi|^{\alpha}t}\phi)\big)^2(x)\chi_{A(t)}(x)\\
& \hspace{0.2cm}\gtrsim \int_{\big\{0<y< \big(\frac{2}{3}\big)^{\frac{1}{(1+\alpha)}}x \big\}} \frac{|\langle x \rangle^{-2}e^{ix|x|^{\alpha}t}\phi(x)-\langle y\rangle^{-2}e^{iy|y|^{\alpha}t}\phi(y)|^2}{|x-y|^{1+2(3/2+\alpha)}}\, dy \, \chi_{A(t)}(x).
\end{aligned}
\end{equation*}
Let $x \in A(t)$ and $0<y<\big(\frac{2}{3}\big)^{\frac{1}{(1+\alpha)}}x$, since $\phi(x')= 1$ and $\langle x' \rangle \sim 1$ for  $|x'|\leq 1$, we have
\begin{equation*}
\begin{aligned}
|\langle x\rangle^{-2}e^{ix|x|^{\alpha}t}\phi(x)-\langle y\rangle^{-2}e^{iy|y|^{\alpha}t}\phi(y)|&=|\langle x\rangle^{-2}(e^{(ix|x|^{\alpha}t-iy|y|^{\alpha}t)}-1)+(\langle x\rangle^{-2}-\langle y\rangle^{-2})|\\
& \gtrsim |\sin(x|x|^{\alpha}t-y|y|^{\alpha}t)|\\
& \gtrsim |x|^{1+\alpha}t,
\end{aligned}
\end{equation*}
where given that $\sin(x)/x\geq 2/\pi$ for $|x|\leq \pi/2$, we have used that $|\sin(x|x|^{\alpha}t-y|y|^{\alpha}t)| \gtrsim |x|^{1+\alpha} t$. Gathering the previous estimates, it is seen that
\begin{equation}\label{derivnotL2.1}
\begin{aligned}
\big(\mathcal{D}_{\xi}^{3/2+\alpha}(\langle \xi \rangle^{-2}e^{i\xi|\xi|^{\alpha}t}\phi)\big)^2(x)\chi_{A(t)}(x) &  \gtrsim \int_{\big\{0<y< \big(\frac{2}{3}\big)^{\frac{1}{(1+\alpha)}}x \big\}} \frac{|x|^{2(1+\alpha)}t^2}{|x|^{1+2(3/2+\alpha)}}\, dy \, \chi_{A(t)}(x) \\
&  \sim (|x|^{-1}t^2)\chi_{A(t)}(x).
\end{aligned}
\end{equation}
Therefore, since $ |x|^{-1/2} \chi_{A(t)} \notin L^{2}(\mathbb{R})$, by Theorem \ref{TheoSteDer}, $D^{3/2+\alpha}(\langle \xi \rangle^{-2}e^{i\xi|\xi|^{\alpha}t}\phi) \notin L^{2}(\mathbb{R})$ as long as $t>0$. When $\alpha=-1$, the estimate for $\mathcal{D}^{1/2}_{\xi}(e^{i\sign(\xi)t}\phi)$ is obtained by similar considerations as above, so we omit this analysis.

\begin{proof}[Proof of Claim \ref{twotimclaim1}]

We will only prove the case $3/2+\alpha \in (0,1)$, i.e., $\alpha \in [-1,1/2)$. Whereas the case $\alpha=-1/2$ follows by replacing the nonlocal derivative by the local derivative in our analysis below. We first notice that as a consequence of property \eqref{prelimneq}, 
\begin{equation}\label{propeweigh}
\|\langle x \rangle^{-2} f\|_{H^{3/2+\alpha}}\lesssim \|f\|_{H^{3/2+\alpha}},
\end{equation}
for any $f\in H^{3/2+\alpha}(\mathbb{R})$. Thus, excluding the weight $\langle \xi \rangle^{-2}$, it is enough to prove that apart of the integral factor, all the terms on the right-hand side of \eqref{twoeq3.1} belong to $H^{3/2+\alpha}_{\xi}(\mathbb{R})$. If $-1<\alpha<-1/2$, in virtue of Lemma \ref{derivexp} and \eqref{prelimneq}, it is seen that
\begin{equation}\label{twoeq4.1}
\begin{aligned}
\|\mathcal{D}^{3/2+\alpha}_{\xi}& \big(e^{i\xi|\xi|^{\alpha}t}(\widehat{u_0}(\xi)-\widehat{u_0}(0))\phi\big)\|_{L^2} \\
&\lesssim \|\mathcal{D}^{3/2+\alpha}_{\xi}\big(e^{i\xi|\xi|^{\alpha}t}\big)(\widehat{u_0}(\xi)-\widehat{u_0}(0))\phi\|_{L^2}+\|e^{i\xi|\xi|^{\alpha}t}\phi \mathcal{D}_{\xi}^{3/2+\alpha}\widehat{u_0}\|_{L^2}\\
&\lesssim _T \|(\widehat{u}(\xi)-\widehat{u_0}(0))\phi\|_{L^2}+\||\xi|^{-1/2}(\widehat{u_0}(\xi)-\widehat{u_0}(0))\phi\|_{L^2}+\|D_{\xi}^{3/2+\alpha}\widehat{u_0}\|_{L^2} \\
&\lesssim_T  \|\widehat{u_0}\|_{H^{3/2+\alpha}_{\xi}}(\|\phi\|_{L^2}+\||\xi|^{1/2+\alpha}\phi\|_{L^2})+\|\widehat{u_0}\|_{H^{3/2+\alpha}},
\end{aligned}
\end{equation}
where we have employed \eqref{Holdercond} with $\gamma=1+\alpha$, i.e., the embedding $C^{0,1+\alpha}(\mathbb{R})\hookrightarrow H^{3/2+\alpha}(\mathbb{R})$, and that $|\xi|^{1/2+\alpha} \in L^2_{loc}(\mathbb{R})$. Now, if $\alpha=-1$, by hypothesis, $\widehat{u_0}\in H^{1/2+\epsilon}(\mathbb{R})$ for some $0<\epsilon \ll 1$, then we can follow the same arguments in \eqref{twoeq4.1} employing \eqref{prelimneq2.0} and the embedding $C^{0,\epsilon}(\mathbb{R})\hookrightarrow H^{1/2+\epsilon}(\mathbb{R})$ to conclude the desired estimate. Notice that in this part we also needed to use $|\xi|^{-1/2+\epsilon}\in L^2_{loc}(\mathbb{R})$.

Now, by using that $\|\widehat{u_0}\|_{L^{\infty}}\lesssim \|J^{(1/2)^{+}}\widehat{u_0}\|_{L^2}$, and the above ideas, we control $\|D^{3/2+\alpha}\big(e^{i\xi|\xi|^{\alpha}t}\widehat{u_0}(\xi)(1-\phi)\big)\|_{L^2}$. Thereby, we got left to estimate the integral part in \eqref{twoeq3.1}.  Applying Theorem \ref{TheoSteDer}, Lemma \ref{derivexp} and \eqref{prelimneq}, we get
\begin{equation}\label{twoeq5}
\begin{aligned}
\|D^{3/2+\alpha}\Big(\int_0^t  e^{\xi|\xi|^{\alpha}(t-t')}& \frac{\xi}{\langle \xi \rangle^2}\widehat{u^2}(\xi,\tau)\, d\tau \Big)\|_{L^2}\\
\lesssim & \int_0^T \Big(\|\frac{(|\xi|^{1/2}+|\xi|)}{\langle \xi \rangle^{2}}\widehat{u^2}(\xi,\tau)\|_{L^2}+\|\mathcal{D}_{\xi}^{3/2+\alpha}\big(\xi\langle \xi \rangle^{-2}\big)\widehat{u^2}(\xi,\tau)\|_{L^2}\\
& +\|\big(\mathcal{D}_{\xi}^{3/2+\alpha}\widehat{u^2}\big)(\xi,\tau)\|_{L^2}\Big)\, d\tau.
\end{aligned}
\end{equation}
The first two terms on the r.h.s of the above inequality can be estimated by employing \eqref{prelimneq1} and the fact that $u\in C([0,T];H^s(\mathbb{R}))$, $s>3/2$. On the other hand, by Sobolev embedding $L^{4}(\mathbb{R})\hookrightarrow H^{1/4}(\mathbb{R})$ and complex interpolation Lemma \ref{complexinterpo}, we find
\begin{equation}\label{twoeq5.1}
\begin{aligned}
\|\big(\mathcal{D}_{\xi}^{3/2+\alpha}\widehat{u^2}\big)(\xi)\|_{L^2} \lesssim \|\langle x \rangle^{3/2+\alpha} u^2\|_{L^2} &\lesssim  \|\langle x \rangle^{3/4+\alpha/2} u\|_{L^4}^2 \\
&\lesssim \|J^{1/4}\big(\langle x \rangle^{3/4+\alpha/2} u\big)\|_{L^2}^2 \\
&\lesssim \|\langle x \rangle^{9/10+3\alpha/5}u\|_{L^2}^{5/3}\|J^{3/2}u\|_{L^2}^{1/3}.
\end{aligned}
\end{equation}
Since $9/10+3\alpha/5<r<3/2+\alpha$ and $s>3/2$, we have that the above expression is controlled by the assumption $u\in C([0,T];Z_{s,r}(\mathbb{R}))$. This completes the study of \eqref{twoeq5} and in turn the proof of the claim. 
\end{proof}


\subsection{Case \texorpdfstring{$0 < \alpha \leq 1/2$}{}}

Since the weight limit $3/2+\alpha \in (3/2,2]$, we require to differentiate \eqref{twoeq2} to proceed with our estimates. But before, let us introduce some notation to simplify the exposition of our results. For each integer $j\geq 0$, we define
\begin{equation}\label{Fnotation}
\begin{aligned}
\mathfrak{F}_{j}(\alpha,t,\xi,f)=\partial_{\xi}^j\big(e^{it\xi|\xi|^{\alpha}}f(\xi)\big).
\end{aligned}
\end{equation}
For instance, we have
\begin{equation}\label{Fnotationex1}
\begin{aligned}
\mathfrak{F}_1(t,\xi,f)=\partial_{\xi}(it\xi|\xi|^{\alpha})\mathfrak{F}_0(t,\xi,f)+\mathfrak{F}_0(t,\xi,\partial_{\xi}f).
\end{aligned}
\end{equation}
Differentiating the integral equation \eqref{twoeq2} produces
\begin{equation}\label{inteequonede} 
\begin{aligned}
\partial_{\xi}\widehat{u}(\xi,t)=\mathfrak{F}_1(t,\xi,\widehat{u_0})-\int_0^t \mathfrak{F}_1(t-t',\xi,\widehat{u \partial_x u})\, dt'. 
\end{aligned}
\end{equation}
We write 
\begin{equation}\label{twoeq6}
\begin{aligned}
\partial_{\xi}(it\xi|\xi|^{\alpha})\mathfrak{F}_0(t,\xi,\widehat{u_0})=&\partial_{\xi}(it\xi|\xi|^{\alpha})\Big(\widehat{u_0}(0)\phi(\xi)+(e^{it\xi|\xi|^{\alpha}}-1)\widehat{u_0}(\xi)\phi(\xi)\\
&+(\widehat{u_0}(\xi)-\widehat{u_0}(0))\phi(\xi)+e^{it\xi|\xi|^{\alpha}}\widehat{u_0}(\xi)(1-\phi(\xi)) \Big)\\
=:&\partial_{\xi}(it\xi|\xi|^{\alpha})\widehat{u_0}(0)\phi(\xi)+\sum_{j=1}^3 \mathfrak{F}_{0,j}(t,\xi,\widehat{u_0}).
\end{aligned}
\end{equation}
We state some spatial decay properties for all the terms in the integral equation \eqref{inteequonede} excepting the first factor on the r.h.s of \eqref{twoeq6}.
\begin{claim}\label{twotimclaim2}
Let $((3+2\alpha)(12-3\alpha))/(4(10-3\alpha))<r<3/2+\alpha$, and $s\geq \{s_{\alpha}^{+},\alpha r\}$. If $u\in C([0,T];Z_{s,r}(\mathbb{R}))$ with $u_0 \in Z_{s,3/2+\alpha}(\mathbb{R})$, then it holds
\begin{equation}
\begin{aligned}
\langle \xi \rangle^{-2}\Big(\sum_{j=1}^3 \mathfrak{F}_{0,j}(t,\xi,\widehat{u_0})+\mathfrak{F}_0(t,\xi,\partial_{\xi}\widehat{u_0})-\int_0^t \mathfrak{F}_1(t-\tau,\xi,\widehat{uu_x})\, d\tau\Big) \in L^{\infty}([0,T];H^{1/2+\alpha}_{\xi}(\mathbb{R})).
\end{aligned}
\end{equation}
\end{claim}
Assuming for the moment the conclusion of Claim \ref{twotimclaim2}, and arguing as in \eqref{twoeq3.3}, it follows from \eqref{inteequonede} and \eqref{twoeq6} that
\begin{equation}\label{twoeq7}
\begin{aligned}
\langle \xi \rangle^{-2}D^{1/2+\alpha}\partial_{\xi}\widehat{u}(\xi,t) \in L^2(\mathbb{R}) \text{ if and only if }  D^{1/2+\alpha}\big(\langle \xi \rangle^{-2}\partial_{\xi}(it\xi|\xi|^{\alpha})\phi\big)\widehat{u_0}(0) \in L^2(\mathbb{R}),
\end{aligned}
\end{equation} 
when $\alpha=-1/2$, we replace $D^{1/2+\alpha}$ by $\partial_x$ in the above statement. Hence, once we have proved that 
\begin{equation}\label{twoeq8}
 D^{1/2+\alpha}\big(\langle \xi \rangle^{-2}\partial_{\xi}(it\xi|\xi|^{\alpha})\phi\big), \, \, \partial_{\xi}\big(\langle \xi \rangle^{-2}\partial_{\xi}(it\xi|\xi|^{1/2})\phi\big)\notin L^2(\mathbb{R}),
\end{equation}
by the assumption \eqref{twoeq1} and \eqref{twoeq7} at $t=t_2$, it must follow $\widehat{u_0}(0)=0$. This completes the proof of Theorem \ref{Theotwotimcondi} (iii) for $0<\alpha \leq 1/2$.

We turn to \eqref{twoeq8}. Since $|\partial_{\xi}^{2}(it\xi|\xi|^{1/2})|\sim |\xi|^{-1/2}$ outside of the origin, we verify \eqref{twoeq8} for $\alpha=-1/2$, i.e., $\partial_{\xi}\big(\langle \xi \rangle^{-2}\partial_{\xi}(it\xi|\xi|^{1/2})\phi\big)\notin L^2(\mathbb{R})$. On the other hand, setting $0<\alpha<1/2$, we obtain
\begin{equation}\label{twoeq8.1} 
\begin{aligned}
\Big(\mathcal{D}^{1/2+\alpha}\big(\langle \xi\rangle^{-2} |\xi |^{\alpha}\phi) \Big)^2(x)\chi_{\{0<|x|\leq 1\}}&\gtrsim \int_{\{0<|y|<|x|/(2\langle 2 \rangle^2)^{1/\alpha} \}} \frac{|\langle x\rangle^{-2} |x |^{\alpha}-\langle y\rangle^{-2} |y |^{\alpha}|^2}{|x-y|^{2+2\alpha}} \, dy \, \chi_{\{0<|x|\leq 1\}} \\
&\gtrsim \int_{\{0<|y|<|x|/(2\langle 2 \rangle^2)^{1/\alpha} \}} \frac{|x|^{2\alpha}}{|x|^{2+2\alpha}} \, dy \, \chi_{\{0<|x|<1\}}\\
&\sim |x|^{-1}\chi_{\{0<|x|<1\}}. 
\end{aligned}
\end{equation}
From the above estimate and Theorem \ref{TheoSteDer}, we get $ D^{1/2+\alpha}\big(\langle \xi \rangle^{-2}\partial_{\xi}(it\xi|\xi|^{\alpha})\phi\big)\notin L^2(\mathbb{R})$. This completes the proof of \eqref{twoeq8}.

\begin{proof}[Proof of Claim \ref{twotimclaim2}]

We focus on the parameters $\alpha\in(0,1/2)$, since the case $\alpha=-1$ follows by replacing the fractional derivative by the local one in our reasoning below. 
Writing $\phi=\phi \widetilde{\phi}$, with $\widetilde{\phi} \in C^{\infty}_c(\mathbb{R})$, \eqref{prelimneq}, Proposition \ref{steinderiweighbet} and Lemma \ref{derivexp} yield
\begin{equation*}
\begin{aligned}
\|&\mathcal{D}_{\xi}^{1/2+\alpha}\big(\mathfrak{F}_{0,1}(t,\xi,\widehat{u_0})\big)\|_{L^2} \\
&\lesssim \||\xi|^{1+\alpha}\mathcal{D}_{\xi}^{1/2+\alpha}(|\xi|^{\alpha}\widetilde{\phi})\phi\|_{L^2}\|\widehat{u_0}\|_{L^{\infty}}+\|\mathcal{D}_{\xi}^{1/2+\alpha}(e^{it\xi|\xi|^{\alpha}})\widehat{u_0}\phi\|_{L^2}+\|\mathcal{D}^{1/2+\alpha}_{\xi}(\widehat{u_0}\phi)\|_{L^2}\\
&\lesssim \|J^{(1/2)^{+}}_{\xi}\widehat{u_0}\|_{L^2}+\||\xi|^{\alpha(1/2+\alpha)}\widehat{u}_0 \phi\|_{L^2}+\|J_{\xi}^{1/2+\alpha}\widehat{u_0}\|_{L^2} \\
& \lesssim \|\langle x\rangle^{1/2+\alpha}u_0\|_{L^2},
\end{aligned}
\end{equation*}
where we have also applied the embedding $H^{(1/2)^{+}}(\mathbb{R})\hookrightarrow L^{\infty}(\mathbb{R})$. Next, by using the mean-value inequality and Sobolev embedding, we get $|\widehat{u_0}(\xi)-\widehat{u_0}(0)|\leq |\xi|\|J^{3/2+\alpha}\widehat{u_0}\|_{L^2}$, thus we can adapt the previous ideas to get
\begin{equation}
\|\mathcal{D}_{\xi}^{1/2+\alpha}\big(\mathfrak{F}_{0,2}(t,\xi,\widehat{u_0})\big)\|_{L^2}\lesssim \|J^{3/2+\alpha}\widehat{u_0}\|_{L^2}\lesssim \|\langle x \rangle^{3/2+\alpha}u_0\|_{L^2}.
\end{equation}
Since $\|\langle \xi \rangle^{-2}f\|_{H^{1/2+\alpha}}\lesssim \|f\|_{H^{1/2+\alpha}}$, whenever $f\in H^{1/2}(\mathbb{R})$, we have that the above estimates yield the desires conclusion for $\mathfrak{F}_{0,j}(t,\xi,\widehat{u_0})$, $j=1,2$. We follow by using properties \eqref{prelimneq} and \eqref{prelimneq1} to deduce
\begin{equation}\label{twoeq9}
\begin{aligned}
\|\mathcal{D}_{\xi}^{1/2+\alpha}\big(\mathfrak{F}_{0,3}(t,\xi,\widehat{u_0})\big)\|_{L^2} \lesssim & \|\mathcal{D}^{1/2+\alpha}_{\xi}\big( \langle \xi \rangle^{-2}|\xi|^{\alpha}(1-\phi)e^{it\xi|\xi|^{\alpha}}\big)\widehat{u_0}\|_{L^2}\\
&+\|\langle \xi \rangle^{-2}|\xi|^{\alpha}(1-\phi)e^{it\xi|\xi|^{\alpha}}\mathcal{D}^{1/2+\alpha}_{\xi}\widehat{u_0}\|_{L^2}\\
\lesssim & \|\mathcal{D}^{1/2+\alpha}_{\xi}\big( \langle \xi \rangle^{-2}|\xi|^{\alpha}(1-\phi)e^{it\xi|\xi|^{\alpha}}\big)\|_{L^{\infty}}\|\widehat{u_0}\|_{L^2}\\
&+\|\langle \xi \rangle^{-2}|\xi|^{\alpha}(1-\phi)e^{it\xi|\xi|^{\alpha}}\|_{L^{\infty}}\|\mathcal{D}^{1/2+\alpha}_{\xi}\widehat{u_0}\|_{L^2}\\
 \lesssim &\|\langle x \rangle^{1/2+\alpha} u_0\|_{L^2}.
\end{aligned}
\end{equation}
By simple modification to the above argument, 
\begin{equation}
\begin{aligned}
\|\mathcal{D}_{\xi}^{1/2+\alpha}\big(\mathfrak{F}_0(t,\xi,\partial_{\xi}\widehat{u_0})\big)\|_{L^2}\lesssim \|\langle x \rangle^{3/2+\alpha}u_0\|_{L^2}.
\end{aligned}
\end{equation}
Thus, it remains to study the integral term. Writing $\widehat{uu_x}=i\xi \widehat{u^2}$, and recalling \eqref{Fnotationex1}, we get
\begin{equation}
\begin{aligned}
\int_0^T\| \mathcal{D}_{\xi}^{1/2+\alpha}\big(\langle \xi \rangle^{-2} \mathfrak{F}_1(t-\tau,\xi,\widehat{uu_x})\big)\|_{L^2}\, d\tau \lesssim &\int_0^T\| \mathcal{D}_{\xi}^{1/2+\alpha}\big(\langle \xi \rangle^{-2}\partial_{\xi}(\xi|\xi|^{\alpha})\xi \, \mathfrak{F}_0(t-\tau,\xi,\widehat{u^2})\big)\|_{L^2}\, d\tau \\
&+\int_0^T\| \mathcal{D}_{\xi}^{1/2+\alpha}\big(\langle \xi \rangle^{-2}\mathfrak{F}_0(t-\tau,\xi,i\widehat{u^2}+i\xi \partial_{\xi}\widehat{u^2})\big)\|_{L^2}\, d\tau \\
=:&\mathcal{I}+\mathcal{II}.
\end{aligned}
\end{equation}
Since the function $\xi \mapsto \langle \xi \rangle^{-2}\partial_{\xi}(\xi|\xi|^{\alpha})\xi e^{it\xi|\xi|^{\alpha}}$, and its derivative belong to $L^{\infty}(\mathbb{R})$, the same argument as in \eqref{twoeq9} shows
\begin{equation}
\mathcal{I} \lesssim \int_0^T\|\langle x \rangle^{1/2+\alpha}u^2(\tau)\|_{L^2}\, d\tau \leq \big(\int_0^T\|u(\tau)\|_{L^{\infty}}\, d\tau\big) \sup_{t\in[0,T]}\|\langle x \rangle^{1/2+\alpha}u(t)\|_{L^2},
\end{equation}
and
\begin{equation}\label{twoeq10}
\mathcal{II} \lesssim \int_0^T\|\langle x \rangle^{1/2+\alpha}u^2(\tau)\|_{L^2}\, d\tau + \int_0^T\|\langle x \rangle^{3/2+\alpha}u^2(\tau)\|_{L^2}\, d\tau.
\end{equation}
The first factor on the r.h.s of \eqref{twoeq10} is estimated as we did with $\mathcal{I}$. On the other hand, complex interpolation Lemma \ref{complexinterpo} yields
\begin{equation}
\begin{aligned}
\|\langle x \rangle^{3/2+\alpha}u^2\|_{L^2}
\lesssim \|J^{1/4}(\langle x \rangle^{3/4+\alpha/2}u)\|_{L^2}^2 \lesssim \|\langle x \rangle^{\frac{(3+2\alpha)(12-3\alpha)}{4(10-3\alpha)}}u\|_{L^2}^{2(\frac{10-3\alpha}{12-3\alpha})}\|J^{3/2-3\alpha/8}u\|_{L^2}^{\frac{4}{12-3\alpha}}.
\end{aligned}
\end{equation}
Recalling that $u\in C([0,T];Z_{s,r}(\mathbb{R}))$ with $\frac{(3+2\alpha)(12-3\alpha)}{4(10-3\alpha)}<r<3/2+\alpha$, and $s\geq \{s_{\alpha}^{+},\alpha r\}$, the above inequality completes the estimate for $\mathcal{II}$, and at the same time the proof of the claim.
\end{proof}
 

\section{Proof of Theorem \ref{Theothretimcondi}} \label{main4}

We prove Theorem \ref{Theothretimcondi} only for the restrictions $-1 \leq \alpha\leq -1/2$ and $0<\alpha\leq 1/2$. We have made these choices to illustrate the main difficulties and arguments required to  cover all the cases where $\alpha$ is positive and negative. We remark that the remaining considerations on $\alpha$ are obtained by simple modifications to the arguments illustrated below. We omit these computations to prevent falling into repetitions. However, we will indicate the general strategy involving the proof of Theorem \ref{Theothretimcondi} whenever $-\alpha \in (-1,1)\setminus \{0\}$.

\subsection{\texorpdfstring{Case $\alpha=-1$}{}}

Bearing in mind that we require to deal with spatial decay greater than one, we consider the integral equation associated to \eqref{BDBO} obtained after taking one derivative in the frequency domain, that is to say,
\begin{equation}\label{integralonedera1}
\begin{aligned}
\partial_{\xi}\widehat{u}(\xi,t)=e^{it\sign(\xi)}\partial_{\xi}\widehat{u_0}(\xi)-\int_0^t e^{i(t-\tau)\sign(\xi)}\partial_{\xi}\widehat{uu_x}(\xi,\tau)\, d\tau,
\end{aligned}
\end{equation}
where we have used $\widehat{u_0}(0)=\widehat{uu_x}(0,\tau)=0$, $\tau \in (0,T)$, and the identity
\begin{equation}
\partial_{\xi}(e^{it\sign(\xi)})\phi(\xi)=2i\sin(t)\phi(0).
\end{equation}
Without loss of generality, we let $t_1 = 0 < t_2 <T$, that is,
\begin{equation}\label{assumpthree0}
u_0 \in \dot{Z}_{s, (3/2)^{+}}(\mathbb{R}), \, u(t_2)  \in \dot{Z}_{s,3/2}(\mathbb{R}),
\end{equation} 
where $s>3/2$. We proceed to decompose the integral equation \eqref{integralonedera1} into parts that ultimately lead to some restrictions arising from the hypothesis at $t=t_2$. We note
\begin{equation}
\begin{aligned}
\text{r.h.s } \eqref{integralonedera1}=& \, \,
e^{it\sign(\xi)}(\partial_{\xi}\widehat{u_0}(\xi)-\partial_{\xi}\widehat{u_0}(0))\phi+e^{it\sign(\xi)}\partial_{\xi}\widehat{u_0}(0)\phi+e^{it\sign(\xi)}\partial_{\xi}\widehat{u_0}(\xi)(1-\phi)\\
&-\int_0^t e^{i(t-\tau)\sign(\xi)}\Big((\partial_{\xi}\widehat{u \partial_x u}(\xi,\tau)-\partial_{\xi}\widehat{u \partial_x u}(0,\tau))\phi+\partial_{\xi}\widehat{u \partial_x u}(0,\tau)\phi+\widehat{u \partial_x u}(\xi,\tau)(1-\phi)\Big) \, d\tau. \\
=:&\sum_{j=1}^6 A_j(\xi,t).
\end{aligned}
\end{equation}
\begin{claim}\label{claimthreet1}
Let $s>3/2$, $9/10<r<3/2$ and $u\in C([0,T];\dot{Z}_{s,r}(\mathbb{R}))$. Assume that $\widehat{u_0} \in \dot{Z}_{s,(3/2)^{+}}(\mathbb{R})$, then 
\begin{equation}\label{eqclaimthreet1}
\begin{aligned}
&\langle \xi \rangle^{-1/2}\sum_{\substack{j=1, \\ j\neq 2, 5}}^{6} A_j(\xi,t) \in L^{\infty}([0,T];H^{1/2}(\mathbb{R})).
\end{aligned}
\end{equation}
\end{claim}
Assuming for the moment Claim \ref{claimthreet1}, as a consequence of \eqref{twoeq3.3}, we arrive at 
\begin{equation}\label{eqintegralonedera1}
\begin{aligned}
\langle \xi \rangle^{-2} D^{1/2}(\partial_{\xi}\widehat{u}(\xi,t))\in L^{2}(\mathbb{R}), \, \, \text{ if and only if } D^{1/2}\big( \langle \xi \rangle^{-2}(A_2(\xi,t)+A_5(\xi,t))\big) \in L^2(\mathbb{R}).
\end{aligned}
\end{equation}
Now, we find
\begin{equation}\label{eqintegralonedera2}
\begin{aligned}
A_{2}(\xi,t)+A_{5}(\xi,t)=\Big(e^{it\sign(\xi)}\partial_{\xi}\widehat{u_0}(0)-\frac{i}{2}\int_0^t e^{i(t-\tau)\sign(\xi)}\widehat{u^2}(0,\tau)\, d\tau\Big)\phi.
\end{aligned}
\end{equation}
We recall the following result which is a direct consequence of the definition (ii) in Theorem \ref{TheoSteDer}.
\begin{prop}\label{optm1}
Let $p\in (1, \infty)$. If $f\in L^{p}(\mathbb{R})$ such that there exists $x_0\in \mathbb{R}$ for which $f(x_0^{+})$, $f(x_0^{-})$ are defined and $f(x_0^{+})\neq f(x_0^{-})$, then for any $\delta>0 $, $\mathcal{D}^{1/p}f \notin L^p_{loc}(B(x_0,\delta))$ and consequently $f\notin L^p_{1/p}(\mathbb{R})$.
\end{prop}
Therefore, since \eqref{eqintegralonedera1} holds at $t=t_2$, Proposition \ref{optm1} imposes
\begin{equation*}
A_{2}(0^{+},t_2)+A_{5}(0^{+},t_2)=A_{2}(0^{-},t_2)+A_{5}(0^{-},t_2),
\end{equation*}
in other words, in virtue of \eqref{eqintegralonedera2} and the $L^2$ conservation law,
\begin{equation}
\begin{aligned}
2i\sin(t_2)\partial_{\xi}\widehat{u_0}(0)=(\cos(t_2)-1)\|u_0\|_{L^2}^2.
\end{aligned}
\end{equation}
Reversing the Fourier variable, the above equality is \eqref{identnegatvalp}. Once we have proved Claim \ref{claimthreet1}, the proof of Theorem \ref{Theothretimcondi} (i) will be complete. 
\begin{proof}[Proof of Claim \ref{claimthreet1}]
Since the $L^2$-norm of the l.h.s of \eqref{eqclaimthreet1} is controlled directly by the hypothesis over $u_0$ and the solutions $u$, we will only compute the derivative component of the $H^{1/2}(\mathbb{R})$-norm, or equivalently, by Theorem \ref{TheoSteDer}, the $L^2$-norm of $\mathcal{D}_{\xi}^{1/2}$. Since $u_0\in L^{2}(|x|^{3^{+}}\, dx)$, there exists $0<\epsilon \ll 1$ such that $\widehat{u_0} \in H^{3/2+\epsilon}_{\xi}(\mathbb{R})$, then by the embedding $C^{1,\epsilon}(\mathbb{R})\hookrightarrow H^{3/2+\epsilon}(\mathbb{R})$, we find
\begin{equation}\label{eqintegralonedera0}
|\partial_{\xi}\widehat{u_0}(\xi)-\partial_{\xi}\widehat{u_0}(0)|\lesssim |\xi|^{\epsilon}\|\widehat{u_0}\|_{H^{3/2+\epsilon}}.
\end{equation}
Let $\widetilde{\phi}\in C^{\infty}_c(\mathbb{R})$ such that $\widetilde{\phi}\phi=\phi$. By applying property \eqref{prelimneq}, the above inequality and \eqref{prelimneq2.0} in Lemma \ref{derivexp}, we deduce
\begin{equation}
\begin{aligned}
\|\mathcal{D}^{1/2}_{\xi}(A_1(\xi))\|_{L^2} &\lesssim \|\widehat{u_0}\|_{H^{3/2+\epsilon}}\|\mathcal{D}^{1/2}_{\xi}(e^{it\sign(\xi)}\widetilde{\phi})|\xi|^{\epsilon}\phi\|_{L^2}+\|\mathcal{D}^{1/2}_{\xi}(\partial_{\xi}\widehat{u_0})\phi\|_{L^2}+\|\partial_{\xi}\widehat{u_0}\mathcal{D}^{1/2}_{\xi}\phi\|_{L^2}\\
&\lesssim \|\widehat{u_0}\|_{H^{3/2+\epsilon}} \||\xi|^{-1/2+\epsilon}\phi\|_{L^2}+\|\widehat{u_0}\|_{H^{3/2}} \\
&\lesssim \|\langle x \rangle^{3/2+\epsilon}u_0\|_{L^2}. 
\end{aligned}
\end{equation}
Now, since the function $e^{it\sign(\xi)}(1-\phi)$ is smooth with all its derivatives bounded, the estimate for $A_3$ is a consequence of \eqref{prelimneq} and \eqref{prelimneq1}. Next, we employ the identity
\begin{equation}\label{eqthreetim1.4}
\partial_{\xi}\widehat{uu_x}(\xi,\tau)=i\widehat{u^2}(\xi,\tau)+i\xi \partial_{\xi}\widehat{u^2}(\xi,\tau),
\end{equation}
and the embedding $C^{0,\epsilon}(\mathbb{R})\hookrightarrow H^{1/2+\epsilon}(\mathbb{R})$ to get
\begin{equation}\label{eqthreetim1.4.1}
\begin{aligned}
\|\mathcal{D}^{1/2}_{\xi}&(e^{i(t-\tau)\sign(\xi)}\big(\langle \xi \rangle^{-2}(\partial_{\xi}\widehat{u \partial_x u}(\xi)-\partial_{\xi}\widehat{u \partial_x u}(0))\phi \big)\|_{L^2} \\
\lesssim & \|\mathcal{D}_{\xi}^{1/2}\big(e^{i(t-\tau)\sign(\xi)}\widetilde{\phi}\big)(\widehat{u^2}(\xi)-\widehat{u^2}(0))\langle \xi \rangle^{-2}\phi\|_{L^2}+ \|\mathcal{D}_{\xi}^{1/2}(\widehat{u^2})\langle \xi \rangle^{-2}\phi\|_{L^2}\\
&+ \|\widehat{u^2}\mathcal{D}_{\xi}^{1/2}(\langle \xi \rangle^{-2}\phi)\|_{L^2} +\|\mathcal{D}_{\xi}^{1/2}\big(e^{i(t-\tau)\sign(\xi)}\widetilde{\phi}\big)\partial_{\xi}\widehat{u^2}(\xi)\xi\langle \xi \rangle^{-2}\phi\|_{L^2}+\|\mathcal{D}_{\xi}^{1/2}(\partial_{\xi}\widehat{u^2})\xi\langle \xi \rangle^{-2}\phi\|_{L^2}\\
&+ \|\partial_{\xi}\widehat{u^2}\mathcal{D}_{\xi}^{1/2}(\xi \langle \xi \rangle^{-2}\phi)\|_{L^2}\\
\lesssim & \big(\||\xi|^{-1/2+\epsilon}\langle \xi \rangle^{-2}\phi\|_{L^2}+\|\langle \xi \rangle^{-2}\phi\|_{L^{\infty}}+\|\mathcal{D}_{\xi}^{1/2}(\langle \xi \rangle^{-2}\phi)\|_{L^2}\big)\|\widehat{u}\|_{H^{1/2+\epsilon}}\\
&+\big(\||\xi|^{1/2}\langle \xi \rangle^{-2}\phi\|_{L^{\infty}}+\|\xi\langle \xi \rangle^{-2}\phi\|_{L^{\infty}}+\|\mathcal{D}^{1/2}_{\xi}(\xi\langle \xi \rangle^{-2}\phi)\|_{L^{\infty}}\big)\|\widehat{u}\|_{H^{3/2}_{\xi}} \\
\lesssim & \|\langle x\rangle^{3/2} u^2\|_{L^2}.
\end{aligned}
\end{equation}
To complete the estimate for the r.h.s of the above inequality, we use Sobolev embedding and complex interpolation to get
\begin{equation}
\begin{aligned}
\|\langle x \rangle^{3/2}u^2(t)\|_{L^2} \lesssim \|J^{1/4}(\langle x \rangle^{3/4}u(t))\|_{L^2}^2 \lesssim \|\langle x\rangle^{9/10}u(t)\|_{L^2}^{5/3}\|J^{3/2}u(t)\|_{L^2}^{1/3}.
\end{aligned}
\end{equation}
Thus, taking the supremum on time $t \in [0,T]$ to the above expression and integrating in time, we complete the estimate for $A_4(\xi)$. Next, by writing $\widehat{uu_x}(\xi)=\frac{i}{2}\xi\widehat{u^2}(\xi)$, and using that $\xi \langle \xi \rangle^{-2}(1-\phi)$ is bounded with bounded derivatives, the estimate for $A_6(\xi)$ is a consequence of \eqref{prelimneq} and \eqref{prelimneq1}, we omit this analysis.
\end{proof}


\subsection{\texorpdfstring{General considerations $-1<\alpha <1$, $\alpha \neq 0$}{}}

Let us discuss the main ideas leading to the proof Theorem \ref{Theothretimcondi} under the present restrictions. By assumption there exist three different times $t_1,t_2,t_3 \in [0,T]$ such that
\begin{equation}
u(t_1), \, u(t_2), \, u(t_3) \in \dot{Z}_{s,5/2+\alpha}(\mathbb{R}).
\end{equation} 
The equation in \eqref{BDBO} and its corresponding $L^2$-conservation law provide the following identity
\begin{equation}\label{threetimeide1}
\frac{d}{dt}\int_{\mathbb{R}} xu(x,t)\, dx=\frac{1}{2}\|u_0\|_{L^2}^2,
\end{equation}
where since $u \in C([0,T];\dot{Z}_{s,r}(\mathbb{R}))$ provided that $3/2<r<5/2+\alpha$, we have $D^{\alpha}u\in C([0,T];L^2(|x|^{1^{+}}\, dx))$, thus $D^{\alpha}u\in C([0,T];L^1(\mathbb{R}))$, and $D^{\alpha}u(x,t)$ has null-integral in space. Notice that this last conclusion does not hold when $\alpha=-1$. Thus, we find
\begin{equation}\label{threetimeide2}
\int_{\mathbb{R}} xu(x,t)\, dx=\int_{\mathbb{R}} xu_0(x)\, dx+\frac{t}{2}\|u_0\|_{L^2}^2.
\end{equation}
If we prove that there exist $\widetilde{t}_1 \in (t_1,t_2)$ and $\widetilde{t}_2 \in (t_1,t_2)$ such that
\begin{equation}\label{eqthreetim1.1}
\int_{\mathbb{R}} xu(x,\widetilde{t}_j)\, dx=0 \, \, \text{ for each } \, \,  j=1,2,
\end{equation}
then in virtue of \eqref{threetimeide1}, it must be the case $u\equiv 0$, which is the desired conclusion. By the symmetry of the argument, we will verify the above equality only for $\widetilde{t}_1 \in (t_1,t_2)$. Additionally, without loss of generality, we let $t_1 = 0 < t_2 < t_3$, that is,
\begin{equation}\label{assumpthree}
u(0), \, u(t_j)  \in \dot{Z}_{s,5/2+\alpha}(\mathbb{R}), \, \, j=2,3.
\end{equation} 
Now we are in condition to prove Theorem \ref{Theothretimcondi} for $-1<\alpha \leq -1/2$ and $0<\alpha\leq 1/2$.

\subsection{\texorpdfstring{Case $-1<\alpha\leq -1/2$}{}}

Since the maximal spatial decay allowed satisfies $3/2<5/2+\alpha \leq 2$, our argument follow from estimating the integral equation \eqref{inteequonede}.  Let us introduce some further notation. For an appropriated function $f$ with enough decay and regularity, we define
\begin{equation}
\begin{aligned}
\widetilde{\mathfrak{F}}_{0,1}(t,\xi,f)=&\partial_{\xi}(it\xi |\xi|^{\alpha})(e^{it\xi|\xi|^{\alpha}}-1)f(\xi)\phi, \, \, \, \widetilde{\mathfrak{F}}_{0,2}(t,\xi,f)=\partial_{\xi}(it\xi |\xi|^{\alpha}) \xi \big(\int_0^1 ( \partial_{\xi}f(\sigma\xi)-\partial_{\xi}f(0) )\, d\sigma\big)\phi, \\
\widetilde{\mathfrak{F}}_{0,3}(t,\xi,f)=&\partial_{\xi}(it\xi |\xi|^{\alpha}) \xi \partial_{\xi}f(0)\phi,\, \, \, \widetilde{\mathfrak{F}}_{0,4}(t,\xi,f)=\partial_{\xi}(it\xi |\xi|^{\alpha})e^{it\xi|\xi|^{\alpha}}f(\xi)(1-\phi), \\
\widetilde{\mathfrak{F}}_{0,5}(t,\xi,f)=&e^{it\xi|\xi|^{\alpha}}\big(\partial_{\xi}f(\xi)-\partial_{\xi}f(0)\big)\phi, \hspace{0.4cm} \widetilde{\mathfrak{F}}_{0,6}(t,\xi,f)=e^{it\xi|\xi|^{\alpha}}\partial_{\xi}f(0)\phi,\\
\widetilde{\mathfrak{F}}_{0,7}(t,\xi,f)=&e^{it\xi|\xi|^{\alpha}}\partial_{\xi}f(\xi)(1-\phi).
\end{aligned}
\end{equation}
Then the following identity holds true
\begin{equation}\label{factoriden}
\begin{aligned}
\partial_{\xi}\widehat{u}(\xi,t)=\sum_{j=1}^7 \big( \widetilde{\mathfrak{F}}_{0,j}(t,\xi,\widehat{u}_0)-\int_0^t \widetilde{\mathfrak{F}}_{0,j}(t-\tau,\xi,\widehat{uu_x})\, d\tau \big).
\end{aligned}
\end{equation}

\begin{claim}\label{claimthreetimes1}
Let $-1<\alpha \leq -1/2$, $s>3/2$, $3/2+3\alpha/5<r<5/2+\alpha$ and $u\in C([0,T];\dot{Z}_{s,r}(\mathbb{R}))$. Furthermore, assume $u_0\in \dot{Z}_{s,5/2+\alpha}(\mathbb{R})$, then
\begin{equation}\label{eqclaimthreetimes1}
\langle \xi \rangle^{-2}\sum_{\substack{j=1 \\ j\neq 3, 6}}^7 \big( \widetilde{\mathfrak{F}}_{0,j}(t,\xi,\widehat{u}_0)-\int_0^t \widetilde{\mathfrak{F}}_{0,j}(t-\tau,\xi,\widehat{uu_x})\, d\tau \big) \in L^{\infty}([0,T];H^{3/2+\alpha}_{\xi}(\mathbb{R})).
\end{equation}
\end{claim}

Let us assume for the moment the results in Claim \ref{claimthreetimes1}. We notice that the terms not contemplated in Claim \ref{claimthreetimes1} in \eqref{factoriden} satisfy the identity
\begin{equation}\label{eqthreetim2}
\begin{aligned}
\sum_{j=3,6}&\big( \widetilde{\mathfrak{F}}_{0,j}(t,\xi,\widehat{u}_0)-\int_0^t \widetilde{\mathfrak{F}}_{0,j}(t-\tau,\xi,\widehat{uu_x})\, d\tau \big)\\
=&\big(\partial_{\xi}(it\xi |\xi|^{\alpha})\xi+e^{it\xi|\xi|^{\alpha}}\big)\partial_{\xi}\widehat{u_0}(0)\phi-\int_0^t \big( \partial_{\xi}(i(t-\tau)\xi |\xi|^{\alpha})\xi +e^{i(t-\tau)\xi|\xi|^{\alpha}}\big)\partial_{\xi}\widehat{u \partial_x u}(0,\tau)\, d\tau \phi.
\end{aligned}
\end{equation}
Let us obtain a more convenient expression for the above equality. Integrating by parts and applying \eqref{threetimeide1}, we deduce
\begin{equation}
\partial_{\xi}\widehat{uu_x}(0,\tau)=-i\int x u \partial_x u(x,\tau)\, d\tau=i \frac{d}{d\tau}\int xu(x,\tau)\, dx.
\end{equation}
Hence, integrating by parts, we write the integral term on the r.h.s of \eqref{eqthreetim2}  as follows
\begin{equation}\label{eqthreetim3}
\begin{aligned}
-i\int_0^t &\big( \partial_{\xi}(i(t-\tau)\xi |\xi|^{\alpha})\xi+e^{i(t-\tau)\xi|\xi|^{\alpha}}\big) \frac{d}{d\tau}\int xu(x,\tau)\, dx\, d\tau \\
=&\partial_{\xi}\widehat{u}(0,t)-\big(\partial_{\xi}(it\xi |\xi|^{\alpha})\xi+e^{it\xi |\xi|^{\alpha}}\big)\partial_{\xi}\widehat{u_0}(0)-i\big(\partial_{\xi}(i\xi|\xi|^{\alpha})\xi+i\xi|\xi|^{\alpha}\big)\int_0^t\int xu(x,\tau)\, dx\, d\tau\\
&+(i\xi|\xi|^{\alpha})\int_0^t (e^{i(t-\tau)\xi|\xi|^{\alpha}}-1)\partial_{\xi}\widehat{u}(0,\tau)\, d\tau.
\end{aligned}
\end{equation}
Plugging the above equality into \eqref{eqthreetim2} yields
\begin{equation}\label{eqthreetim3.1}
\begin{aligned}
\sum_{j=3,6}\big( \widetilde{\mathfrak{F}}_{0,j}&(t,\xi,\widehat{u}_0)-\int_0^t \widetilde{\mathfrak{F}}_{0,j}(t-\tau,\xi,\widehat{uu_x})\, d\tau \big)\\
=&\partial_{\xi}\widehat{u}(0,t)\phi-i(\partial_{\xi}(i\xi|\xi|^{\alpha})\xi+i\xi|\xi|^{\alpha})\int_0^t\int xu(x,\tau)\, dx\, d\tau \phi\\
&+(i\xi|\xi|^{\alpha})\int_0^t (e^{i(t-\tau)\xi|\xi|^{\alpha}}-1)\partial_{\xi}\widehat{u}(0,\tau)\, d\tau \phi.
\end{aligned}
\end{equation}
Adapting the ideas in the proof of Claim \ref{claimthreetimes1} below, it is not difficult to see that
\begin{equation}\label{eqthreetim3.2}
\partial_{\xi}\widehat{u}(0,t)\phi+(i\xi|\xi|^{\alpha})\int_0^t (e^{i(t-\tau)\xi|\xi|^{\alpha}}-1)\partial_{\xi}\widehat{u}(0,\tau)\, d\tau \phi \in H^{3/2+\alpha}(\mathbb{R}).
\end{equation}
Therefore, gathering Claim \ref{claimthreetimes1}, \eqref{eqthreetim3.1}, \eqref{eqthreetim3.2}  and using the ideas around \eqref{twoeq3.3}, we find
\begin{equation}\label{eqthreetim4}
\begin{aligned}
\langle \xi \rangle^{-2}(D^{3/2+\alpha}\partial_{\xi}\widehat{u})&(\xi,t)\in L^{2}(\mathbb{R}) \\
 &\text{ if and only if } \, D^{3/2+\alpha}\big(\langle \xi \rangle^{-2}\xi|\xi|^{\alpha} \phi\big)\int_0^t\int_{\mathbb{R}} xu(x,\tau)\, dx\, d\tau \in L^{2}(\mathbb{R}),
\end{aligned}
\end{equation}
where we have written $(\partial_{\xi}(\xi|\xi|^{\alpha})\xi+\xi|\xi|^{\alpha}=(2+\alpha)\xi|\xi|^{\alpha}$. For $\alpha=-1/2$, we employ \eqref{eqthreetim4} with the partial derivative $\partial_{\xi}$ instead of the nonlocal one $D$. By employing similar arguments as above (see, \eqref{derivnotL2.1} and \eqref{twoeq8.1}), it is not difficult to deduce $D^{3/2+\alpha}\big(\langle \xi \rangle^{-2}\xi|\xi|^{\alpha} \phi\big) \notin  L^{2}(\mathbb{R})$, whenever $-1<\alpha<0$. A similar conclusion holds true replacing the operator $D$ by the partial derivative in the case $\alpha=-1/2$. Thus, since \eqref{assumpthree} assures \eqref{eqthreetim4} at $t=t_2$, \eqref{eqthreetim4} forces us  to have
\begin{equation}\label{eqthreetim5.1} 
\int_0^{t_2}\int_{\mathbb{R}} xu(x,\tau)\, dxd\tau=0.
\end{equation}
However, the above equality and  the continuity of the map $\tau \mapsto \int_{\mathbb{R}} xu(x,\tau)$ assures that there exists a time $\widetilde{t}\in(0,t_2)$ for which \eqref{eqthreetim1.1} is true. According to previous discussions, this completes the proof of Theorem \ref{Theothretimcondi}, whenever $-1<\alpha\leq 1/2$.

\begin{proof}[Proof of Claim \ref{claimthreetimes1}]
We only focus on the estimate for the homogeneous derivative in the $H^{3/2+\alpha}$-norm, that is, according to Theorem \ref{TheoSteDer}, we deal with the derivative $\mathcal{D}^{3/2+\alpha}_{\xi}$. Additionally, we will consider the restrictions $\alpha \in (-1,-1/2)$, while the case $\alpha=-1/2$ follows from the same analysis replacing the fractional derivative by the local operator $\partial_x$. 

We begin by giving a better estimate without incorporating the weight $\langle \xi \rangle^{-2}$ for the terms provided by the homogeneous part of the integral equation in \eqref{eqclaimthreetimes1}. Let $\widetilde{\phi}\in C^{\infty}_c(\mathbb{R})$ be such that $\phi\widetilde{\phi}=\phi$. We employ \eqref{prelimneq0}, \eqref{prelimneq}, \eqref{eqsteinderiweig1} in Proposition \ref{steinderiweighbet2} with $\gamma=1$, and \eqref{prelimneq2.1} in Lemma \ref{derivexp} to get
\begin{equation}\label{eqthreetim5.1.1}
\begin{aligned}
\|\mathcal{D}^{3/2+\alpha}_\xi &\big(\widetilde{\mathfrak{F}}_{0,1}(t,\xi,\widehat{u}_0)\big)\|_{L^2}\\
\lesssim & \|\mathcal{D}^{3/2+\alpha}(|\xi|^{\alpha}\widehat{u_0}\phi)(e^{it\xi|\xi|^{\alpha}}-1)\widetilde{\phi}\|_{L^2}+\||\xi|^{\alpha}\widehat{u_0}\phi\mathcal{D}^{3/2+\alpha}\big((e^{it\xi|\xi|^{\alpha}}-1)\widetilde{\phi}\big)\|_{L^2} \\
\lesssim &\big(\|\widehat{u_0}\phi\|_{L^{\infty}}+\|\partial_{\xi}(\widehat{u_0}\phi)\|_{L^{\infty}}+(1+\|\mathcal{D}_{\xi}^{3/2+\alpha}\widetilde{\phi}\|_{L^{\infty}})\|\partial_{\xi}\widehat{u_0}\|_{L^{\infty}}\big) \big(\||\xi|^{1/2+\alpha}\widetilde{\phi}\|_{L^2}+\||\xi|^{1+\alpha}\phi\|_{L^2}\big) \\
\lesssim & \|\langle x \rangle^{5/2+\alpha} u_0\|_{L^2},
\end{aligned}
\end{equation}
where we have also used $\|\partial_{\xi}\widehat{u_0}\|_{L^{\infty}} \lesssim \|J_{\xi}^{5/2+\alpha}\widehat{u_0}\|_{L^{2}}$ as $3/2<5/2+\alpha$. Next, by \eqref{prelimneq}, 
\begin{equation*}
\begin{aligned}
\|\mathcal{D}^{3/2+\alpha}_\xi \big(\widetilde{\mathfrak{F}}_{0,2}(t,\xi,\widehat{u}_0)\big)\|_{L^2}\lesssim & \|\mathcal{D}_{\xi}^{3/2+\alpha}(\sign(\xi)|\xi|^{1+\alpha}\widetilde{\phi})\int_0^1\big(\partial_{\xi}\widehat{u_0}(\sigma \xi)-\partial_{\xi}\widehat{u_0}(0)\big)\, d \sigma \, \phi\|_{L^2}\\
&+\||\xi|^{1+\alpha}\widetilde{\phi}\|_{L^{\infty}}\bigg(\|\int_0^1\big(\partial_{\xi}\widehat{u_0}(\sigma \xi)-\partial_{\xi}\widehat{u_0}(0)\big)\, d \sigma \,  \phi\|_{L^2}\\
&\hspace{0.2cm}+\|D_{\xi}^{3/2+\alpha}\Big(\int_0^1\big(\partial_{\xi}\widehat{u_0}(\sigma \xi)-\partial_{\xi}\widehat{u_0}(0)\big)\, d \sigma \, \phi \Big) \|_{L^2}\bigg) \\
=:& \mathcal{I}_1+\mathcal{I}_2+\mathcal{I}_3.
\end{aligned}
\end{equation*}
By Sobolev embedding $C^{1,\gamma}(\mathbb{R})\hookrightarrow H^{5/2+\alpha}(\mathbb{R})$ for any $0<\gamma \leq 1+\alpha$, we find
\begin{equation}\label{eqthreetim1.2}
\begin{aligned}
\Big|\int_0^1\big(\partial_{\xi}\widehat{u_0}(\sigma \xi)-\partial_{\xi}\widehat{u_0}(0)\big)\, d \sigma \Big| \lesssim \|\widehat{u_0}\|_{H^{5/2+\alpha}}|\xi|^{1+\alpha}.
\end{aligned}
\end{equation}
The above inequality and Proposition \ref{steinderiweighbet} show
\begin{equation}
\begin{aligned}
\mathcal{I}_1+\mathcal{I}_2 &\lesssim  \|\widehat{u_0}\|_{H^{5/2+\alpha}}\big(\|\mathcal{D}_{\xi}^{3/2+\alpha}(\sign(\xi)|\xi|^{1+\alpha}\widetilde{\phi})|\xi|^{1+\alpha}\phi\|_{L^2}+\||\xi|^{1+\alpha}\phi\|_{L^2}\big)\\
&\sim \|\langle x \rangle^{5/2+\alpha}u_0\|_{L^2}.
\end{aligned}
\end{equation}
Applying the embedding $L^{\infty}(\mathbb{R})\hookrightarrow H^{(1/2)^{+}}(\mathbb{R})$ and Lemma \ref{lemmacomm1}, we get
\begin{equation}\label{eqthreetim1.3}
\begin{aligned}
\mathcal{I}_3 \lesssim & \int_0^1 \|D_{\xi}^{3/2+\alpha}(\partial_{\xi}\widehat{u_0}(\sigma\xi)\phi)\|_{L^2}\, d\sigma +\|J^{(3/2)^{+}}\widehat{u_0}\|_{L^2}\|D_{\xi}^{3/2+\alpha}\phi\|_{L^2}\\
\lesssim & \int_0^1 \big(\|[D_{\xi}^{3/2+\alpha},\phi](\partial_{\xi}\widehat{u_0}(\sigma\xi))\|_{L^2}+\sigma^{3/2+\alpha}\|\phi (D_{\xi}^{3/2+\alpha}\partial_{\xi}\widehat{u_0})(\sigma\xi))\|_{L^2}\big)\, d\sigma \\
&+\|J^{(3/2)^{+}}\widehat{u_0}\|_{L^2}\|D_{\xi}^{3/2+\alpha}\phi\|_{L^2} \\
\lesssim & \|D_{\xi}^{3/2+\alpha}\phi\|_{L^{2}}\|\partial_{\xi}\widehat{u_0}\|_{L^{\infty}}+\big(\int_0^1 \sigma^{1+\alpha}\, d\sigma \big) \|\phi\|_{L^{\infty}}\|D_{\xi}^{3/2+\alpha}\partial_{\xi}\widehat{u_0}\|_{L^2}  +\|J^{(3/2)^{+}}\widehat{u_0}\|_{L^2} \\
\lesssim & \|\langle x \rangle^{5/2+\alpha}u_0\|_{L^2}.
\end{aligned}
\end{equation}
This completes the study of $\widetilde{\mathfrak{F}}_{0,2}(t,\xi,\widehat{u_0})$. The estimate for $\widetilde{\mathfrak{F}}_{0,4}(t,\xi,\widehat{u_0})$ is obtained by properties \eqref{prelimneq} and \eqref{prelimneq1}, and the fact that $|\xi|^{\alpha}e^{it\xi|\xi|^{\alpha}}(1-\phi)$ is bounded with bounded derivatives. Since this is also true for $|\xi|^{\alpha}e^{it\xi|\xi|^{\alpha}}(1-\phi)$, the estimate for $\widetilde{\mathfrak{F}}_{0,7}(t,\xi,\widehat{u_0})$ follows in a similar fashion. 

Now, as we have used in  \eqref{eqthreetim1.2}, the embedding $C^{1,\gamma}(\mathbb{R})\hookrightarrow H^{5/2+\alpha}(\mathbb{R})$, $0<\gamma\leq 1+\alpha$ allows us to deduce
\begin{equation}
\begin{aligned}
\big|\partial_{\xi}\widehat{u_0}(\xi)-\partial_{\xi}\widehat{u_0}(0)\big| \lesssim \|\widehat{u_0}\|_{H^{5/2+\alpha}}|\xi|^{1+\alpha}.
\end{aligned}
\end{equation}
Then, we can adapt previous arguments, employing the above inequality and \eqref{prelimneq2.1} in Lemma \ref{derivexp} to obtain the desired estimate for $\widetilde{\mathfrak{F}}_{0,5}(t,\xi,\widehat{u_0})$.

The estimate for the factors arising from the integral term are obtained by writing $\widehat{uu_x}(\xi)=\frac{i\xi}{2}\widehat{u^{2}}$, and incorporating the negative weight $\langle \xi \rangle^{-2}$. Since these estimates follow by repeated arguments such as in \eqref{eqthreetim1.4.1}, we omit these computations. 
\end{proof}


\subsection{\texorpdfstring{Case $0<\alpha\leq 1/2$}{}}

We notice that when $0<\alpha \leq 1/2$, the maximum decay parameter satisfies $2<5/2+\alpha \leq 3$, so we consider the integral equation obtained after applying two derivatives to \eqref{twoeq2}. This motives us to introduce the following notation:
\begin{equation}
\begin{aligned}
\mathfrak{F}_{1,1}(t,\xi,f)=&\partial_{\xi}^2(it\xi|\xi|^{\alpha})(e^{it\xi|\xi|^{\alpha}}-1)f(\xi)\phi, \hspace{0.1cm} \mathfrak{F}_{1,2}(t,\xi,f)=\partial_{\xi}^2(it\xi|\xi|^{\alpha}) \xi \big(\int_0^1 (\partial_{\xi}f(\sigma \xi)-\partial_{\xi}f(0))\, d \sigma\big)\phi, \\
\mathfrak{F}_{1,3}(t,\xi,f)=&\big(\partial_{\xi}^2(it\xi|\xi|^{\alpha})\xi+ 2\partial_{\xi}(it\xi|\xi|^{\alpha})\big) \partial_{\xi}f(0) \, \phi, \\
\mathfrak{F}_{1,4}(t,\xi,f)=&\partial_{\xi}^2(it\xi|\xi|^{\alpha})e^{it\xi|\xi|^{\alpha}}f(\xi)(1-\phi),  \hspace{0.3cm} \mathfrak{F}_{1,5}(t,\xi,f)=\big(\partial_{\xi}(it\xi|\xi|^{\alpha})\big)^2e^{it\xi|\xi|^{\alpha}}f(\xi), \\  
\mathfrak{F}_{1,6}(t,\xi,f)=&2\partial_{\xi}(it\xi|\xi|^{\alpha})( e^{it\xi|\xi|^{\alpha}}-1)\partial_{\xi}f(\xi)\phi,  \hspace{0.3cm} \mathfrak{F}_{1,7}(t,\xi,f)=2\partial_{\xi}(it\xi|\xi|^{\alpha})( \partial_{\xi}f(\xi)-\partial_{\xi}f(0))\phi,  \\
\mathfrak{F}_{1,8}(t,\xi,f)=&2\partial_{\xi}(it\xi|\xi|^{\alpha})e^{it\xi|\xi|^{\alpha}}\partial_{\xi}f(\xi)(1-\phi), \hspace{0.3cm} \mathfrak{F}_{1,9}(t,\xi,f)=e^{it\xi|\xi|^{\alpha}}\partial_{\xi}^2f(\xi)\phi.
\end{aligned}
\end{equation}
Then, we find
\begin{equation}\label{eqthreetim5.2}
\partial_{\xi}^2 \widehat{u}(\xi,t)=\sum_{j=1}^{9}\Big(\mathfrak{F}_{1,j}(t,\xi,\widehat{u_0})-\int_0^t \mathfrak{F}_{1,j}(t-\tau,\xi,\widehat{uu_x}) \, d\tau \Big).
\end{equation}

\begin{claim}\label{eqclaimthreetimes2}
Let $((5+2\alpha)(12-3\alpha))/(4(10-3\alpha))<r<5/2+\alpha$, and $s\geq \{s_{\alpha}^{+},\alpha r\}$. If $u\in C([0,T];\dot{Z}_{s,r}(\mathbb{R}))$ with $u_0 \in \dot{Z}_{s,3/2+\alpha}(\mathbb{R})$, then it holds
\begin{equation}
\langle \xi \rangle^{-4}\sum_{\substack{j=1 \\ j\neq 3}}^9 \big( \mathfrak{F}_{1,j}(t,\xi,\widehat{u}_0)-\int_0^t \mathfrak{F}_{1,j}(t-\tau,\xi,\widehat{uu_x})\, d\tau \big) \in L^{\infty}([0,T];H^{1/2+\alpha}_{\xi}(\mathbb{R})).
\end{equation}
\end{claim}
We omit the proof of Claim \ref{eqclaimthreetimes2} as it follows by rather similar arguments developed in the proof of Claim \ref{claimthreetimes1}, we just need to perform simple modification to treat with dispersions $0<\alpha \leq 1/2$. Now, given that $\partial_{\xi}^2(it\xi|\xi|^{\alpha})\xi+ 2\partial_{\xi}(it\xi|\xi|^{\alpha})=it(1+\alpha)(2+\alpha)|\xi|^{\alpha}$, we have
\begin{equation}\label{eqclaimthreetimes2.1}
\begin{aligned}
\mathfrak{F}_{1,3}(t,\xi,\widehat{u_0})&-\int_0^t \mathfrak{F}_{1,3}(t-\tau,\xi,\widehat{uu_x}) \, d\tau \\
&=i(1+\alpha)(2+\alpha)|\xi|^{\alpha}\big(t\partial_{\xi}\widehat{u_0}(0)-\int_0^t (t-\tau)\partial_{\xi}\widehat{uu_x}(0,\tau)\, d \tau \, \big) \phi.
\end{aligned}
\end{equation}
In virtue of \eqref{eqthreetim5.2}, Claim \ref{eqclaimthreetimes2}, \eqref{eqclaimthreetimes2.1} and arguing as in \eqref{twoeq3.3}, we get
\begin{equation}\label{eqthreetim6}
\begin{aligned}
\langle \xi \rangle^{-4}D^{1/2+\alpha}\big(\partial_{\xi}^2\widehat{u}(\xi,t)\big) &\in L^{2}(\mathbb{R}) \, \text{ if and only if } \\
&D^{1/2+\alpha}\big(\langle \xi \rangle^{-4}|\xi|^{\alpha}\big(t\partial_{\xi}\widehat{u_0}(0)-\int_0^t (t-\tau)\partial_{\xi}\widehat{uu_x}(0,\tau)\, d \tau \, \big) \phi \big) \in  L^2(\mathbb{R}),
\end{aligned}
\end{equation}
when $\alpha=-1/2$, we replace the operator $D$ by $\partial_x$ in the above statement. Thus, by adapting the arguments in \eqref{twoeq8.1}, it is not difficult to deduce
\begin{equation}\label{eqthreetim7}
\partial_{\xi}\big(\langle \xi \rangle^{-4}|\xi|^{1/2}\phi\big), \, D^{1/2+\alpha}\big(\langle \xi \rangle^{-4}|\xi|^{\alpha}\phi\big) \notin L^2(\mathbb{R}),
\end{equation}
whenever $0<\alpha \leq 1/2$. Consequently, hypothesis \eqref{assumpthree} implies that \eqref{eqthreetim6} holds at $t=t_2$. However, as a result of \eqref{eqthreetim7}, it must be the case
\begin{equation}
\begin{aligned}
0=&t_2\partial_{\xi}\widehat{u_0}(0)-\int_0^{t_2} (t_2-\tau)\partial_{\xi}\widehat{uu_x}(0,\tau)\, d \tau\\
=&-it_2\int_{\mathbb{R}} x u_0(x)\, dx-i\int_0^{t_2}(t_2-\tau)\big(\frac{d}{d\tau}\int_{\mathbb{R}} x u(x,\tau)\, dx \big)\, d\tau \\
=&-i\int_0^{t_2}\int_{\mathbb{R}} x u(x,\tau)\, d \tau,
\end{aligned}
\end{equation}
where reversing the Fourier variables, we have applied \eqref{threetimeide2} and integration by parts. By the arguments below \eqref{eqthreetim5.1}, the above expression leads to the desired conclusion.

\begin{remark}\label{remarkthree} 
As we mentioned at the beginning of the proof of Theorem \ref{Theothretimcondi}, by similar considerations developed above, it is not difficult to deduce similar statements as in \eqref{eqthreetim4} and \eqref{eqthreetim6} for the cases $\alpha \in (-1/2,0)\cup (1/2,1)$. More precisely, under the hypothesis of Theorem \ref{Theothretimcondi}, whereas $-1/2<\alpha<0$ one finds
\begin{equation}\label{remreq1}
\begin{aligned}
\langle \xi \rangle^{-2}(D^{1/2+\alpha}\partial_{\xi}^2\widehat{u})&(\xi,t)\in L^{2}(\mathbb{R}) \text{ if and only if } \, \int_0^{t_2} \int x u(x,\tau)\, dx\, d\tau=0.
\end{aligned}
\end{equation}
Assuming now that $1/2<\alpha<1$, it follows
\begin{equation}\label{remreq2}
\begin{aligned}
\langle \xi \rangle^{-6}(D^{\alpha-1/2}\partial_{\xi}^3\widehat{u})&(\xi,t)\in L^{2}(\mathbb{R}) \text{ if and only if } \, \int_0^{t_2} \int x u(x,\tau)\, dx\, d\tau=0.
\end{aligned}
\end{equation}

\end{remark}


\section{Proof of Theorem \ref{reductwotimes}} \label{main5}

Without loss of generality we may assume that
\begin{equation}
t_1=0 \hspace{0.5cm} \text{ and } \hspace{0.5cm} \int xu_0(x)\, dx=0.
\end{equation}
If $-1<\alpha \leq -1/2$, we collect \eqref{threetimeide2}, \eqref{eqthreetim4} and \eqref{eqthreetim5.1} to deduce
\begin{equation}\label{eqreductwotimes}
\begin{aligned}
 \widehat{u}(\cdot,t_2) &\in H^{5/2+\alpha}(\mathbb{R}) \hspace{0.2cm} \text{ implies } \\
&\langle \xi \rangle^{-2}(D^{3/2+\alpha}\partial_{\xi}\widehat{u})(\xi,t_2) \in L^2(\mathbb{R}), \hspace{0.2cm} \text{ this holds if and only if }   \\
&0=\int_0^{t_2} \int x u(x,\tau)\, dx\, d\tau=\frac{1}{2}\int_0^{t_2} \tau\left\|u(\tau)\right\|_{L^2}^2 d\tau=\frac{t_2^2}{4} \left\|u_0\right\|_{L^2}^2.
\end{aligned}
\end{equation}
In virtue of \eqref{eqthreetim6}, \eqref{remreq1} and \eqref{remreq2} in  Remark \ref{remarkthree},  a similar conclusion to \eqref{eqreductwotimes} can be drawn replacing the second line in \eqref{eqreductwotimes} by $\langle \xi \rangle^{-2}(D^{1/2+\alpha}\partial_{\xi}^2\widehat{u})(\xi,t_2)$ for $-1/2<\alpha<0$, or by $\langle \xi \rangle^{-4}(D^{1/2+\alpha}\partial_{\xi}^2\widehat{u})(\xi,t_2)$ for $0<\alpha \leq 1/2$, or by $\langle \xi \rangle^{-6}(D^{\alpha-1/2}\partial_{\xi}^3\widehat{u})(\xi,t_2)$ for $1/2<\alpha<1$. This completes the proof of the theorem.


\section{Proof of Theorem \ref{timesharp} } \label{main6}

We consider $\alpha \in (-1,1)\setminus \{0\}$. Let $u\in C([0,T];\dot{Z}_{s,(5/2+\alpha)^{-}}(\mathbb{R}))$ solution of \eqref{BDBO}, where $s>5/2$ for $-1<\alpha<0$, and $s>\max\{(3+2\alpha(5+2\alpha))/(1+\alpha(5+2\alpha)),(5/2+\alpha)\alpha+1\}$ for $0<\alpha<1$. We first claim 
\begin{equation}\label{eqtimesharp1}
    u\partial_{x}u \in L^{\infty}([0,T];Z_{s-1,5/2+\alpha}(\mathbb{R})).
\end{equation}
Indeed, we write $u \partial_x u=\frac{1}{2}\partial_x(u^2)$, so that taking $0<\epsilon\ll 1$, Lemma \ref{complexinterpo} yields
\begin{equation}\label{eqtimesharp2}
\begin{aligned}
\|\langle x \rangle^{5/2+\alpha} u \partial_x u\|_{L^2} &\lesssim \|\langle x \rangle^{3/2+\alpha} u^2\|_{L^2}+\|J(\langle x \rangle^{5/2+\alpha} u^2)\|_{L^2} \\
&\lesssim \|u\|_{H^{(1/2)^{+}}}\|\langle x \rangle^{3/2+\alpha} u\|_{L^2}+\|J^{a_{1,\alpha}}(u^2)\|_{L^2}^{\theta_{1,\alpha}} \|\langle x\rangle^{b_{1,\alpha}} u^2\|_{L^2}^{1-\theta_{1,\alpha}}\\
&\lesssim \|u\|_{Z_{s,(5/2+\alpha)^{-}}}^2+\|J^{a_{1,\alpha}}u\|_{L^2}^{2\theta_{1,\alpha}}\|\langle x\rangle^{b_{1,\alpha}} u^2\|_{L^2}^{1-\theta_{1,\alpha}}.
\end{aligned}
\end{equation} 
To complete the estimate of the above inequality, we apply complex interpolation to find
\begin{equation}\label{eqtimesharp3}
\begin{aligned}
\|\langle x\rangle^{b_{1,\alpha}} u^2\|_{L^2}\sim \|\langle x\rangle^{\frac{b_{1,\alpha}}{2}} u\|_{L^4}^2 &\lesssim  \|J^{1/4}\big(\langle x\rangle^{\frac{b_{1,\alpha}}{2}} u \big)\|_{L^2}^2 \\
&\lesssim  \|\langle x\rangle^{5/2+\alpha-\epsilon} u\|_{L^2}^{2(1-\theta_{2,\alpha})}\|J^{a_{2,\alpha}}u\|_{L^2}^{2\theta_{2,\alpha}}.
\end{aligned}
\end{equation}
Let us choose $a_{1,\alpha},$ $b_{1,\alpha}$, $b_{2,\alpha}$, $\theta_{1,\alpha}$ and $\theta_{2,\alpha}$ according to the values of $\alpha \in (-1,1)\setminus\{0\}$. For $-1<\alpha<0$, we define
\begin{equation}
\begin{aligned}
a_{1,\alpha}&=\theta_{1,\alpha}^{-1}=\frac{5}{2}\Big(\frac{5+2\alpha-2\epsilon}{5+2\alpha-5\epsilon} \Big), \qquad b_{1,\alpha}=\frac{5}{6}(5+2\alpha-2\epsilon), \\
a_{2,\alpha}&=\frac{3}{2}, \qquad \theta_{2,\alpha}=\frac{1}{6}.
\end{aligned}
\end{equation}
We have chosen $a_{2,\alpha}$ above in concordance with the local theory in $H^{s_1}(\mathbb{R})$, $s_1>3/2$. As a consequence, $a_{1,\alpha}$ is obtained by technical reasons that require the condition $s>5/2$. Next, if $0<\alpha<1$, we consider 
\begin{equation}
\begin{aligned}
a_{1,\alpha}&=\theta_{1,\alpha}^{-1}=\frac{(3+2\alpha(5+2\alpha))(5+2\alpha-2\epsilon)}{(1+\alpha(5+2\alpha))(5+2\alpha-2\epsilon)-2\epsilon(2+\alpha(5+2\alpha))},\\
 b_{1,\alpha}&=\frac{(2(2+\alpha(5+2\alpha))-1)(5+2\alpha-2\epsilon)}{2(2+\alpha(5+2\alpha))}, \\
a_{2,\alpha} &=\frac{2+\alpha(5+2\alpha)}{2}, \qquad \theta_{2,\alpha}=\frac{1}{2(2+\alpha(5+2\alpha))}.
\end{aligned}
\end{equation}
We can justify the choice $a_{2,\alpha}$ for $0<\alpha<1$ from the relation between regularity and decay which imposes $uu_x \in H^{\alpha(5/2+\alpha)}(\mathbb{R})$, so we look for solutions satisfying $u \in H^{\alpha(5/2+\alpha)+1}(\mathbb{R})$. In contrast, the choice of $a_{1,\alpha}$ for $0<\alpha<1$ is completely technical. Hence, plugging \eqref{eqtimesharp3} into \eqref{eqtimesharp2}, and taking $0<\epsilon \ll 1$ sufficiently small, the decay and regularity assumptions on the solution $u$ of \eqref{BDBO} complete the deduction of \eqref{eqtimesharp1}.

Consequently, we can employ \eqref{eqtimesharp1} to perform all the computations developed in the proof of Theorem \ref{Theothretimcondi} directly in the space $ H^{5/2+\alpha}_{\xi}(\mathbb{R})$  without incorporating the weight $\langle \xi \rangle^{-m}$, where $m=2$ for $-1<\alpha<0$, $m=4$ for $0<\alpha \leq 1/2$, and $m=6$ for $1/2<\alpha<1$ (see Remark \ref{remarkthree}). This in turn yields
\begin{equation}
\begin{aligned}
\widehat{u}(\cdot,t) &\in H^{5/2+\alpha}_{\xi}(\mathbb{R}) , \hspace{0.2cm} \text{ if and only if }   \\
&0=\int_0^{t} \int x_1 u(x,\tau)\, dx\, d\tau =\int_0^{t} \big(\int x_1u_0(x)\, dx +\frac{\tau}{2}\left\|u_0\right\|_{L^2}^2 \big)d\tau=0, \hspace{0.2cm} \text{ if and only if } \\
&0= t \left( \int x_1u_0(x)\, dx +\frac{t}{4}\left\|u_0\right\|_{L^2}^2 \right).
\end{aligned}
\end{equation}
The proof of the theorem is complete.


\section{Appendix}\label{appendix}

This section concerns the deduction of Proposition \ref{fractfirstcaldcomm} and Lemma \ref{lemmacomm3}. 

\subsection{Appendix A: Proof of Proposition \ref{fractfirstcaldcomm}}\label{appendA}

We observe
\begin{equation}\label{eqComwellprel1}
\begin{aligned}
 \left[\mathcal{H},g\right]&D_x^{\beta} f(x)=-i\int |\xi_2|^{\beta}\big(\sign(\xi_1+\xi_2)-\sign(\xi_2)\big)\widehat{g}(\xi_1)\widehat{f}(\xi_2)e^{ix\cdot(\xi_1+\xi_2)} \, d\xi_1 d\xi_2,
\end{aligned}
\end{equation}
then neglecting the null measure sets where $\xi_1+\xi_2=0$ or $\xi_2=0$, it follows that the integral in \eqref{eqComwellprel1} is not null only when $(\xi_1+\xi_2)\xi_2<0$, in order words, when $|\xi_2|<|\xi_1|$. Consequently, recalling the family of projector introduced in \eqref{projectors}, by Bony's paraproduct decomposition we may write
\begin{equation*}
    \begin{aligned}
 \left[\mathcal{H},g\right]D^{\beta} f=&\mathcal{H}\big(\sum_{j} P_j g P_{<j-2} D^{\beta}f  \big)-\sum_{j} P_j g P_{<j-2} \mathcal{H} D^{\beta}f +\mathcal{H} \big(\sum_{j} P_j g \widetilde{P}_{j} D^{\beta}f  \big) -\sum_{j} P_j g \widetilde{P}_j\mathcal{H} D^{\beta}f \\
 =:& \mathcal{L}_1+\mathcal{L}_2+\mathcal{L}_3+\mathcal{L}_4,
\end{aligned}
\end{equation*}
where we have set $P_{<j-2}f=\sum_{k<j-2}P_{k}f$ and $\widetilde{P}_jf=\sum_{|j-k|\leq 2}P_{k}f$. Denoting by $\mathcal{M}(\cdot)$ the Hardy-Littlewood maximal function, by the Littlewood-Paley inequality and Fefferman-Stein inequality (see \cite{FefermStein}), we have 
\begin{equation}\label{eqComwellprel2}
\begin{aligned}
\|\mathcal{L}_1\|_{L^p}&\lesssim \|\big(\sum_{|j-k|\leq 2} {P_k}(2^{-j\beta}P^{\ast}_j  D^{\beta}g P_{<j-2} D^{\beta}f)\big)_{l^2_k}\|_{L^p} \\
& \lesssim \|\big(\sum_{|j-k|\leq 2} \mathcal{M}(2^{-j\beta}P_j^{\ast}   D^{\beta}g P_{<j-2} D^{\beta}f)\big)_{l^2_k}\|_{L^p}\\
& \lesssim \|\big( 2^{-j\beta}P_j^{\ast}  D^{\beta}g P_{<j-2} D^{\beta}f\big)_{l^2_j}\|_{L^p},
\end{aligned}
\end{equation}
for some modified projection $P_j^{\ast}$ supported on $|\xi|\sim j$, and we have used $|P_kf(x) |\lesssim \mathcal{M}(f)(x)$. We compute
\begin{equation*}
\begin{aligned}
\big( 2^{-j\beta}P_j^{\ast}  P_j D^{\beta}g P_{<j-2} D^{\beta}f\big)_{l^2_j} \lesssim \big(\sum_{l<j-2} 2^{-\beta(j-l)}\mathcal{M}(D^{\beta}g)\mathcal{M}(P_l f)\big)_{l^2_j}\\
\lesssim \sum_{2<m} 2^{-\beta m} \big(\mathcal{M}(D^{\beta}g)\mathcal{M}(P_{j-m} f)\big)_{l^2_j},
\end{aligned}
\end{equation*}
so that this estimate and \eqref{eqComwellprel2} yield
\begin{equation*}
\begin{aligned}
\|\mathcal{L}_1\|_{L^p}\lesssim \|\big(\mathcal{M}(D^{\beta}g)\mathcal{M}(P_{j} f)\big)_{l^2_j}\|_{L^p} \lesssim \|\mathcal{M}(D^{\beta}g)\|_{L^{\infty}}\|\big(\mathcal{M}(P_jf)\big)_{l^2_j}\|_{L^p} \lesssim \|D^{\beta}g\|_{L^{\infty}}\|f\|_{L^p}.
\end{aligned}
\end{equation*}
Replacing $f$ by $\mathcal{H}f$, the above argument yields the desired estimate for $\mathcal{L}_2$. Next, we have
\begin{equation*}
\mathcal{L}_3+\mathcal{L}_4=\mathcal{H}\sigma(D^{\beta}g,f)+\sigma(D^{\beta}g,\mathcal{H}f),
\end{equation*}
where we define the operators
\begin{equation*}
\sigma(h_1,h_2)=\int e^{ix \cdot (\xi_1+\xi_2)}\sigma(\xi_1,\xi_2)\widehat{h_1}(\xi_1)\widehat{h_2}(\xi_2)\, d\xi_1 d \xi_2,
\end{equation*}
determined by a slightly abuse of notations through the symbol
\begin{equation*}
\sigma(\xi_1,\xi_2)=\sum_j \sum_{|k-j|\leq 2} |\xi_1|^{-\beta}|\xi_2|^{\beta}\psi(2^{-j}\xi_1)\psi(2^{-k}\xi_2).
\end{equation*}
Then it is not difficult to see $\sigma(\xi_1,\xi_2)\in C^{\infty}(\mathbb{R}\times \mathbb{R}\setminus \{(0,0)\})$ and that
\begin{equation*}
|\partial_{\xi_1}^{\gamma_1}\partial_{\xi_2}^{\gamma_2}\sigma(\xi_1,\xi_2)|\lesssim_{\gamma_1,\gamma_2}(|\xi_1|+|\xi_2|)^{-|\gamma_1|-|\gamma_2|} 
\end{equation*} 
for all multi-index $\gamma_1, \gamma_2$ and all $(\xi_1,\xi_2)\neq (0,0)$. Then by  Coifman-Meyer multilinear theorem (see, \cite{CoifmanMeyer,coifman}) we have
\begin{equation*}
\begin{aligned}
\|\mathcal{L}_3+\mathcal{L}_4\|_{L^p}\lesssim \|\sigma(D^{\beta}g,f)\|_{L^{p}}+\|\sigma(D^{\beta}g,\mathcal{H}f)\|_{L^{p}} \lesssim \|D^{\beta}g\|_{L^{\infty}}\|f\|_{L^p},
\end{aligned}
\end{equation*}
which completes the proof.

\subsection{Appendix B: Proof of Lemma \ref{lemmacomm3} }\label{appendB}

In virtue of Bony's paraproduct decomposition we may write
\begin{equation*}
\begin{aligned}
D^{\beta}\left[P^{\phi},f\right]D^{\gamma}g 
=&\sum_j D^{\beta}P^{\phi}(P_jf D^{\gamma} P_{<j-2}g)-D^{\beta}\big(P_jf P^{\phi}D^{\gamma}P_{<j-2}g\big), \\
&+\sum_j D^{\beta}P^{\phi}(P_jf D^{\gamma} \widetilde{P}_{j}g)-D^{\beta}\big(P_jf P^{\phi}D^{\gamma}\widetilde{P}_{j}g\big), \\
&+\sum_j D^{\beta}P^{\phi}(P_{<j-2}f D^{\gamma} P_{j}g)-D^{\beta}\big(P_{<j-2}f P^{\phi}D^{\gamma}P_j g\big) \\
=&\mathcal{L}(l,h)+\mathcal{L}(h,h)+\mathcal{L}(h,l),
\end{aligned}
\end{equation*}
where $\mathcal{L}(h,l)$, $\mathcal{L}(h,h)$ and $\mathcal{L}(l,h)$ denote the high-low, high-high and low-high iterations respectively.  The estimates for $\mathcal{L}(h,l)$ and $\mathcal{L}(h,h)$ are obtained by analyzing separately each term of the commutator and following similar ideas as in the proof of Proposition \ref{fractfirstcaldcomm}. We omit these computations. Let us deal with the low-high iteration $\mathcal{L}(l,h)$. Since
\begin{equation*}
\begin{aligned}
|\xi_1+\xi_2|^{\beta}|\xi_2|^{\gamma}(\phi(\xi_1+\xi_2)-\phi(\xi_2))= \int_0^1|\xi_1+\xi_2|^{\beta}|\xi_2|^{\gamma}\frac{d}{dx}\phi(\lambda\xi_1+\xi_2)\, d \lambda \, \xi_1,
\end{aligned}
\end{equation*}
we write
\begin{equation*}
\mathcal{L}(l,h)=-i\sigma_3(\partial_x f,g),
\end{equation*}
where the operator $\sigma_3$ is associated to the symbol
\begin{equation*}
\begin{aligned}
\sigma_3&(\xi_1,\xi_2)= \\
&\sum_j \Big(\int_0^1 \frac{|\xi_1+\xi_2|^{\beta}|\xi_2|^{\gamma}}{|\lambda \xi_1+\xi_2|^{\beta+\gamma}}\big(|\lambda \xi_1+\xi_2|^{\beta+\gamma}\frac{d}{dx}\phi(\lambda\xi_1+\xi_2) \big)\, d\lambda \Big)\psi_{<j-2}(\xi_1)\psi(\xi_2/2^{j}),
\end{aligned}
\end{equation*}
we set $\psi_{<j-2}(\cdot)=\sum_{k<j-2}\psi(\cdot/2^k)$. Thus, given that $\sigma_3(\xi_1,\xi_2)$ is supported in the region $|\xi_1|\ll |\xi_2|$ and that $\phi\in C^{\infty}_c(\mathbb{R})$ with $\phi(\xi)=1$ for $|\xi|\leq 1$, it is not difficult to verify that $\sigma_3$ satisfies the hypothesis of the Coifman-Meyer multilinear theorem. Consequently,
\begin{equation*}
\|\mathcal{L}(l,h)\|_{L^p}=\|\sigma_3(\partial_x f,g)\|_{L^p} \lesssim \|\partial_x f\|_{L^{\infty}}\|g\|_{L^p}.
\end{equation*}
The proof of the lemma is complete.


\section*{Acknowledgment}

The author gratefully acknowledges the many helpful suggestions of Prof. Felipe Linares. The author wishes to extend his gratitude to Prof. Jean-Claude Saut for pointing out part (c) in Remark \ref{remark3}.


\bibliographystyle{abbrv}
\bibliography{bibli}

\end{document}